\documentclass[11pt, a4paper]{article}

\usepackage[dvipsnames]{xcolor}
\usepackage[pagebackref,backref=true,colorlinks=true, linkcolor=Sepia, citecolor=Sepia]{hyperref}
\usepackage{amsmath,amsthm,amssymb,fancyhdr,graphicx,bbm,cancel,mathrsfs,todonotes,xypic,pinlabel}
\usepackage{tikz}
\usepackage{extarrows,tipa}
\usetikzlibrary{matrix}
\usepackage[all]{xy}
\usepackage{mathtools}
\usepackage{bm}
\usepackage{verbatim}
\usepackage{microtype}
\usepackage{subcaption}

\DeclareFontFamily{T1}{cbgreek}{}
\DeclareFontShape{T1}{cbgreek}{m}{n}{<-6>  grmn0500 <6-7> grmn0600 <7-8> grmn0700 <8-9> grmn0800 <9-10> grmn0900 <10-12> grmn1000 <12-17> grmn1200 <17-> grmn1728}{}
\DeclareSymbolFont{quadratics}{T1}{cbgreek}{m}{n}
\DeclareMathSymbol{\qoppa}{\mathord}{quadratics}{19}
\DeclareMathSymbol{\Qoppa}{\mathord}{quadratics}{21}

\theoremstyle{theorem}
\newtheorem{thm}{Theorem}[section]
\newtheorem{mainthm}{Theorem}

\newtheorem{lem}[thm]{Lemma}
\newtheorem{prop}[thm]{Proposition}
\newtheorem{cor}[thm]{Corollary}
\theoremstyle{definition} 

\newtheorem{qu}[thm]{Question}

\newtheorem{rmk}[thm]{Remark}
\theoremstyle{remark} 

\usepackage{enumerate,pdfpages,stmaryrd} 
\usepackage[margin=1.2in, marginparwidth=1in]{geometry}

\usepackage{tikz}
\usepackage{dsfont}
\usetikzlibrary{matrix}

\newcommand{\Z}{\mathbb{Z}}

\newcommand{\pair}[1]{\langle #1 \rangle}

\DeclareMathOperator{\Mod}{Mod}
\DeclareMathOperator{\Aut}{Aut}
\newcommand{\rest}[2]{#1\bigr\vert_{#2}}

\date{\today}
\begin{document}

\title{Finite rigid sets and the separating curve complex}

\author{Junzhi Huang and Bena Tshishiku}


\maketitle
	
\begin{abstract}
For each closed surface $\Sigma$ of genus $g\ge3$, we find a finite subcomplex of the separating curve complex that is rigid with respect to incidence-preserving maps. 
\end{abstract}

\section{Introduction}\label{sec:introduction}

This paper is about rigidity properties of the complex $\mathcal C^s(\Sigma)$ of separating curves on a closed, oriented surface $\Sigma$ of genus $g\ge3$. 

For a subcomplex $Z\subset\mathcal C^s(\Sigma)$, a simplicial map $\phi:Z\rightarrow\mathcal C^s(\Sigma)$ is said to be \emph{incidence preserving} if $i(\alpha,\beta)\neq0$ implies that $i(\phi(\alpha),\phi(\beta))\neq0$ for vertices $\alpha,\beta$ of $Z$. Here $i(\cdot,\cdot)$ denotes the geometric intersection number (see \S\ref{sec:background} for definitions).


The goal of this paper is to prove the following result.

\begin{mainthm}[Finite rigid set of separating curves]\label{thm:main}
Fix a closed, oriented surface $\Sigma$ of genus $g\ge3$. 
There is a finite subcomplex $Y^s\subset\mathcal C^s(\Sigma)$ such that any incidence-preserving map $\phi:Y^s\rightarrow\mathcal C^s(\Sigma)$ is induced by an extended mapping class, i.e.\ there exists $h\in\Mod^\pm(\Sigma)$ such that $\rest{h}{Y^s}=\phi$. Furthermore, $h$ is unique. 
\end{mainthm}

The construction of $Y^s$ is explicit and is given in \S\ref{sec:rigid-set}. We assume $g\ge3$ because the $\mathcal C^s(\Sigma_2)$ is discrete and $\mathcal C^s(\Sigma_1)$ is empty. We also note that the case $g=3$ is easy because we can take $Y^s$ to be a single vertex. Below we restrict to $g\ge4$, an assumption necessary for our argument.

\begin{rmk}[Incidence preserving vs.\ superinjective]\label{rmk:superinjective}
Incidence-preserving maps $\mathcal C^s(\Sigma)\rightarrow\mathcal C^s(\Sigma)$ are typically called \emph{superinjective maps} in the literature. This terminology originates in work of Irmak \cite{irmak}. We do not use the term ``superinjective" because it would be misleading, since for general $Z\subset\mathcal C^s(\Sigma)$, an incidence-preserving map $Z\rightarrow\mathcal C^s(\Sigma)$ need not be injective.
\end{rmk}

\begin{rmk}[Similar rigidity results]
Theorem \ref{thm:main} may be viewed as a ``hybrid" of two results. 
\begin{itemize}
\item In \cite{AL}, Aramayona--Leininger construct a finite subcomplex $X$ of the curve complex $\mathcal C(\Sigma)$ with the property that any locally injective map $X\rightarrow\mathcal C(\Sigma)$ is induced by an extended mapping class. 
\item In \cite{BM}, Brendle--Margalit prove that any incidence-preserving map $\mathcal C^s(\Sigma)\rightarrow\mathcal C^s(\Sigma)$ is induced by an extended mapping class. 
\end{itemize}
These results have as common ancestor Ivanov's rigidity result $\Aut(\mathcal C(\Sigma))\cong\Mod^\pm(\Sigma)$ (see \cite{ivanov}), which itself is motivated by the fundamental theorem of projective geometry. 

For other results about finite rigid subcomplexes, see Shinkle's work \cite{shinkle,shinkle2}. 

One may view Theorem \ref{thm:main} (and the similar aforementioned rigidity results) as curve-complex analogues of the fact that an isometry of a Riemannian manifold is determined by its values on a small open set.
\end{rmk}

\begin{rmk}[Different notions of rigidity]
For the subcomplex $Y^s$ constructed in the proof of Theorem \ref{thm:main}, one can ask if any \emph{locally injective} map $Y^s\rightarrow\mathcal C^s(\Sigma)$ is induced by an extended mapping class, in analogy to the result of \cite{AL}. We do not know whether or not this is true, and it seems difficult to prove using the same strategy as \cite{AL}. 
\end{rmk}

Theorem \ref{thm:main} raises two questions: 

\begin{qu}\label{q:homology}
Is the finite rigid subset $Y^s\subset\mathcal C^s(\Sigma)$ produced in Theorem \ref{thm:main} homologically nontrivial in $\mathcal C^s(\Sigma)$? 
\end{qu}

The analogue of Question \ref{q:homology} for the Aramayona--Leininger rigid set $X\subset \mathcal C(\Sigma)$ is asked in \cite[Question 2]{AL} and is answered by Birman--Broaddus--Menasco \cite{BBM}. One difficulty in answering Question \ref{q:homology} is that the homotopy type of $\mathcal C^s(\Sigma)$ is not completely known. Looijenga proved that $\mathcal C^s(\Sigma)$ is $(g-3)$-connected \cite{looijenga}, and Looijenga and van der Kallen \cite{LvdK} proved that the quotient $\mathcal C^s(\Sigma)/\mathcal I_g$ by the Torelli group is homotopy equivalent to a wedge of $(g-2)$-dimensional spheres. The complex $Y^s$ that we define has dimension $2g-4$, which is equal to the dimension of $\mathcal C^s(\Sigma)$. 

In \S\ref{sec:genus3}, we construct a rigid set in the genus-3 separating curve complex that is homotopy equivalent to a wedge of 21 circles. This gives a very concrete example to test Question \ref{q:homology}.

\begin{qu}\label{q:diameter}
Does there exist finite rigid subsets of $\mathcal C^s(\Sigma)$ of arbitrarily large diameter? 
\end{qu}

The analogue of Question \ref{q:diameter} for $\mathcal C(\Sigma)$ is asked in \cite[Question 1]{AL} and answered affirmatively in \cite{AL2}.



\subsection{Strategy of the proof}

To define $Y^s$, we first define a subcomplex $Y$ of the (full) curve complex $\mathcal C(\Sigma)$, and then we define $Y^s=Y\cap\mathcal C^s(\Sigma)$ (see \S\ref{sec:rigid-set}). The proof proceeds in the following steps. 




\paragraph{Step 1 (rigidity of $Y$).} We show that any locally injective map $Y\rightarrow\mathcal C(\Sigma)$ is induced by a mapping class (Proposition \ref{prop:Y-rigid}). To prove this, we use that our $Y$ contains the finite rigid set $X\subset\mathcal C(\Sigma)$ of Aramayona--Leininger \cite{AL}. Although rigidity is not generally inherited by supersets, we can show $Y$ is rigid by showing that each element of $Y\setminus X$ is \emph{uniquely determined} (in the sense of \cite[Defn.\ 2.1]{AL}) by a subset of $X$. 

Given Step 1, we would like to show that any incidence-preserving map $\phi:Y^s\rightarrow\mathcal C^s(\Sigma)$ extends to a locally injective map $\widetilde\phi:Y\rightarrow\mathcal C(\Sigma)$. For then Step 1 implies that $\widetilde\phi$, and hence $\phi$, is induced by an extended mapping class $h$. Furthermore, $h$ is unique by the uniqueness part of \cite[Thm.\ 5.1]{AL}.

To define an extension, we need to define $\widetilde\phi(y)$ for nonseparating curves. We do this using \emph{sharing pairs} (defined by Brendle--Margalit \cite[\S4]{BM}), which give a way to specify a nonseparating curve by a pair of genus-1 separating curves. The key to using sharing pairs to define the extension is showing that certain sharing pairs are preserved under incidence-preserving maps and that the image is independent of our choice of the sharing pair.

\paragraph{Step 2 (sharing pairs are preserved).} An incidence-preserving map $\phi:Y^s\rightarrow\mathcal C^s(\Sigma)$ sends certain sharing pairs to sharing pairs (Theorem \ref{thm:sharing-pair}). In order to prove this, we also show that $\phi$ preserves the genus of each curve in $Y^s$ (Theorem \ref{thm:genus}). The proofs of these statements are intertwined: cases of Theorem \ref{thm:sharing-pair} are used to prove cases of Theorem \ref{thm:genus} which are used to prove other cases of Theorem \ref{thm:sharing-pair}, and so on. A key tool in this analysis is the characterization of sharing pairs given in \cite[Lem.\ 4.1]{BM}. 



\paragraph{Step 3 (extending from $Y^s$ to $Y$).} Any incidence-preserving map $\phi:Y^s\rightarrow\mathcal C^s(\Sigma)$ extends to a locally-injective map $\widetilde\phi:Y\rightarrow\mathcal C(\Sigma)$ (Theorem \ref{thm:extend}). Given $\phi$, for $y\in Y\setminus Y^s$, we define $\widetilde\phi(y)$ to be the curve shared by $\phi(\alpha),\phi(\beta)$, where $\alpha,\beta$ are a sharing pair for $y$. We need to prove that $\widetilde\phi$ is well-defined, simplicial, and locally injective. The proofs of these statements build on the results of Step 2. An important ingredient is to show that $\widetilde\phi$ preserves (certain) chains of nonseparating curves (Lem.\ \ref{lem:chains}). 


Steps 2 and 3 are fairly involved and are at the technical heart of the paper. 

\begin{rmk}[Comparison with \cite{BM}]
Steps 2 and 3 above are similar to steps in an argument of Brendle--Margalit \cite[\S4]{BM} (part of which was corrected by Kida \cite[\S5]{kida}) that proves that an incidence-preserving map $\mathcal C^s(\Sigma)\rightarrow\mathcal C^s(\Sigma)$ extends to an incidence preserving map $\mathcal C(\Sigma)\rightarrow\mathcal C(\Sigma)$. 
One difference in replacing the domain $\mathcal C^s(\Sigma)$ by a finite subcomplex $Z$ is that this greatly reduces the flexibility of certain constructions. As a simple example if $\alpha\neq\beta$ are disjoint separating curves, then it is easy to prove that there is another separating curve $\gamma$ with $i(\alpha,\gamma)\neq0$ and $i(\beta,\gamma)=0$. But this fact is not generally true if $\alpha,\beta,\gamma$ are restricted to a finite subcomplex $Z\subset\mathcal C^s(\Sigma)$. This is related to the fact that an incidence-preserving map on a finite subcomplex is not generally injective (c.f.\ Remark \ref{rmk:superinjective}).
\end{rmk}

\paragraph{Acknowledgements.} The authors thank C.\ Leininger and D.\ Margalit for several helpful comments on a draft of this paper. 

\section{Background and terminology}\label{sec:background}

\subsection{Curves and curve complexes} 

An essential simple closed curve $\alpha\subset\Sigma$ is an embedded circle that does not bound a disk. The \emph{curve complex} $\mathcal C(\Sigma)$ is the simplicial complex with a vertex for each isotopy class of essential, simple closed curve on $\Sigma$ and a $k$-simplex for every collection of $(k+1)$ isotopy classes that can be realized disjointly on $\Sigma$. See e.g.\ \cite[\S4.1]{farb-margalit}. Below ``curve" always means ``essential simple closed curve". We will usually not distinguish between a curve an its isotopy class. 

A curve $\alpha\subset\Sigma$ is separating if $\Sigma\setminus\alpha$ is disconnected. The \emph{genus} of a separating curve $\alpha\subset\Sigma$ is the minimum genus of the two components of $\Sigma\setminus\alpha$. The \emph{separating curve complex} $\mathcal C^s(\Sigma)\subset\mathcal C(\Sigma)$ is defined as the subcomplex spanned by  separating curves. 

For a pair of isotopy classes of curves $\alpha,\beta$, the \emph{geometric intersection number} $i(\alpha,\beta)$ is the minimum of $|a\cap b|$ over all representatives $a\in\alpha$ and $b\in\beta$. See \cite[\S1.2]{farb-margalit}. Whenever working with curves $\alpha,\beta\subset\Sigma$, we always assume that they are in \emph{minimal position}, i.e.\ realize the geometric intersection number of their isotopy classes. If $i(\alpha,\beta)=0$, we say that $\alpha$ and $\beta$ are \emph{disjoint}; otherwise we say that $\alpha$ and $\beta$ \emph{intersect}. 

A \emph{chain} is a sequence of curves $(x_1,\ldots,x_n)$ such that 
\[i(x_j,x_k)=\begin{cases}1&|k-j|=1\\0&\text{else}\end{cases}.\]
Given a chain, we use $N(x_1\cup\cdots\cup x_n)$ to denote a regular neighborhood. The boundary $\partial N(x_1\cup\cdots\cup x_n)$ is a simple closed curve when $n$ is even, and a bounding pair when $n$ is odd. We will frequently use chains to define curves on $\Sigma$; furthermore, in our pictures we will often indicate a curve by drawing a chain defining it, rather than drawing the curve itself.

\subsection{Locally injective and incidence-preserving maps} 

A simplicial map $f:Z_1\rightarrow Z_2$ is \emph{locally injective} if the restriction to the star of every vertex is injective. When $Z_2$ is a curve complex and $Z_1\subset Z_2$ is a subcomplex, this criterion in particular implies that if $\alpha\neq \beta$ are disjoint curves, then also $f(\alpha)\neq f(\beta)$ are disjoint curves.

Given a subcomplex $Z\subset\mathcal C^s(\Sigma)$, a simplicial map $\phi:Z\rightarrow\mathcal C^s(\Sigma)$ is \emph{incidence-preserving} if $i(\alpha,\beta)\neq0$ implies that $i(\phi(\alpha),\phi(\beta))\neq0$ for $\alpha,\beta\in Z$. Note that for a simplicial map $i(\alpha,\beta)=0$ implies that $i(\phi(\alpha),\phi(\beta))=0$, so for incidence preserving maps we have $i(\alpha,\beta)=0$ if and only if $i(\phi(\alpha),\phi(\beta))=0$. 

The proof of the following lemma is easy using the definitions; see \cite[Lem.\ 3.1]{irmak}

\begin{lem}\label{lem:injective-criterion}Fix $Z\subset\mathcal C^s(\Sigma)$. Assume that whenever $\alpha,\beta\in Z$ and $i(\alpha,\beta)=0$ then exists $\gamma\in Z$ so that $i(\alpha,\gamma)=0$ and $i(\beta,\gamma)\neq0$. Then any incidence-preserving map $Z\rightarrow\mathcal C^s(\Sigma)$ is injective. 
\end{lem}


\subsection{Sharing pairs and spines}

For a nonseparating curve $y\subset\Sigma$, we say that curves $\alpha_1,\alpha_2$ are a \emph{sharing pair that shares $y$} if (1) $\alpha_1,\alpha_2$ are genus-1 curves; (2) if $\Sigma_j$ denotes the genus-1 subsurfaces bounded by $\alpha_j$, then $\Sigma_1\cap\Sigma_2$ is an annulus containing $y$; and (3) the complement of $\Sigma_1\cup\Sigma_2$ in $\Sigma$ is connected. 

It is often convenient to specify a sharing pair by a choice of \emph{spine}. A triple $(x_1,y,x_2)$ of curves is called a \emph{spine} if $i(x_j,y)=1$ and $i(x_1,x_2)\le1$. If $(x_1,y,x_2)$ is a spine, then $\alpha_i:=\partial N(x_i\cup y)$ defines a sharing pair. We call $(x_1,y,x_2)$ a spine of the sharing pair $\alpha_1$ and $\alpha_2$. For example, a chain of length 3 is the spine of a sharing pair. 

\begin{rmk}
We will only ever need to use spines that are actually chains. Nevertheless, we will continue to use this more general notion to make it easier to reference the results of \cite{BM} that we use.  
\end{rmk} 

The following lemma from \cite[Lem.\ 4.1]{BM} gives a characterization of sharing pairs in terms of genus and intersection data of curves. The subsequent Corollary \ref{cor:sharing-pair} gives a checkable condition for showing that an incidence-preserving map preserves sharing pairs. 

\begin{lem}[Brendle--Margalit]\label{lem:sharing-pair}
Assume $\Sigma$ has genus $\ge4$. Let $\alpha,\beta$ be genus-$1$ separating curves in $\Sigma$. Then $(\alpha,\beta)$ is a sharing pair if and only if there exist separating curves $w,x,y,z$ in $\Sigma$ with the following properties: 
 $z$ bounds a genus-$2$ subsurface $\Sigma_z$; $\alpha$ and $\beta$ are in $\Sigma_z$ and intersect each other; $x$ and $y$ are disjoint; $w$ intersects $z$, but not $a$ and not $b$; $x$ intersects $a$ and $z$, but not $b$; $y$ intersects $b$ and $z$, but not $a$. 
\end{lem}


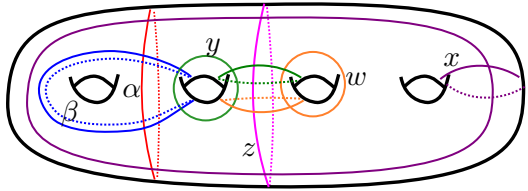
\begin{figure}[h!]
\labellist
\pinlabel $\alpha$ at 66 2535
\pinlabel $\beta$ at 40 2525
\pinlabel $w$ at 160 2540
\pinlabel $x$ at 200 2548
\pinlabel $y$ at 100 2553
\pinlabel $z$ at 115 2510
\endlabellist
\centering
\includegraphics[scale=.9]{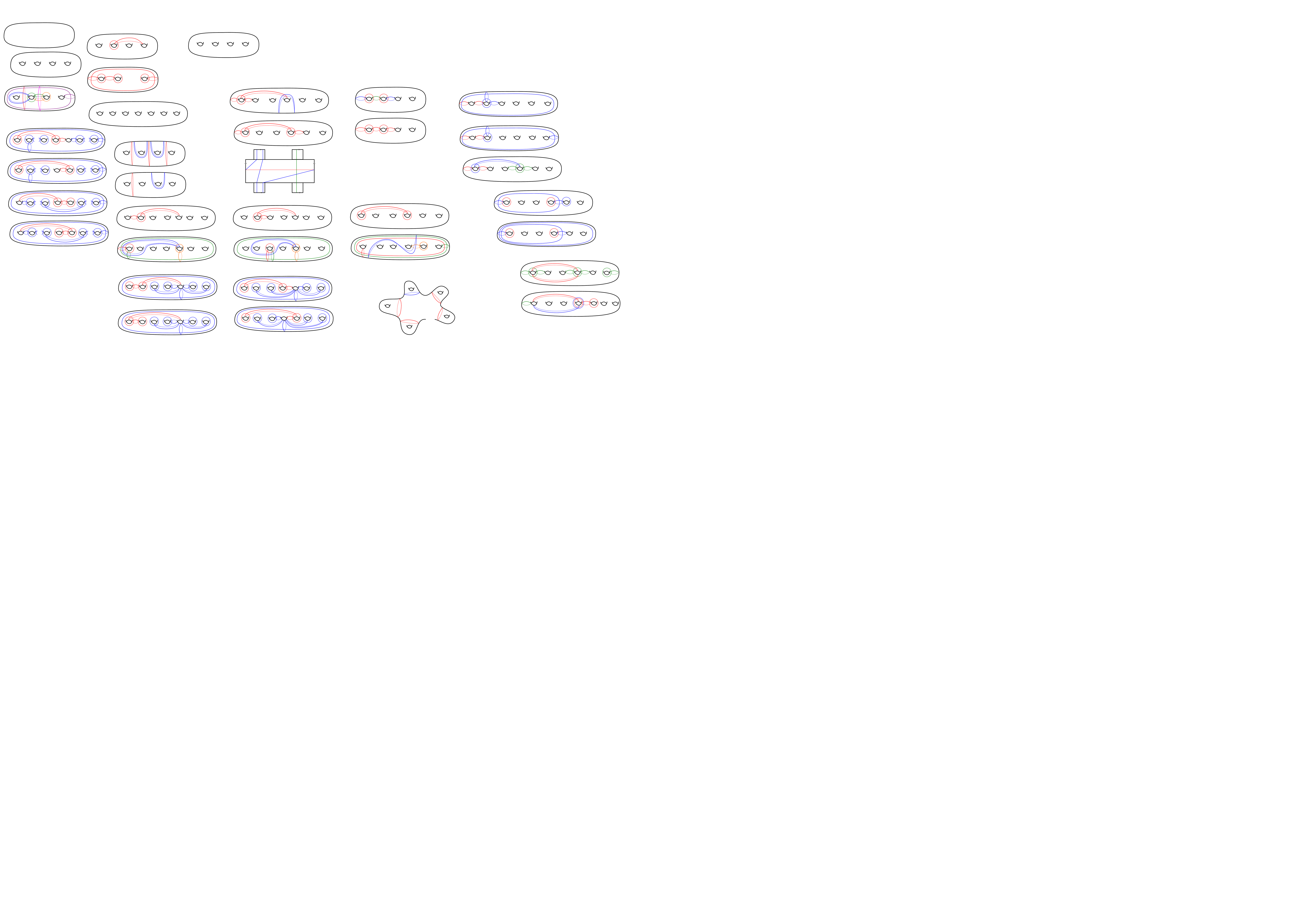}
\caption{Curves $\alpha,\beta$ and  $w,x,y,z$ as in the statement of Lemma \ref{lem:sharing-pair}.}
\label{fig:sharing-pair}
\end{figure}

We say that the curves $w,x,y,z$ \emph{certify} that $(\alpha,\beta)$ are a sharing pair.  

\begin{rmk}
The condition that $z$ has genus 2 is the only place where genus $g\ge4$ is needed in Lemma \ref{lem:sharing-pair} (for $g=3$, all separating curves have genus 1). 
\end{rmk}

\begin{cor}\label{cor:sharing-pair}
Let $Z\subset\mathcal C^s(\Sigma)$ be a subcomplex, and fix an incidence-preserving map $\phi:Z\rightarrow\mathcal C^s(\Sigma)$. Let $\alpha,\beta\in Z$ be a sharing pair certified by $w,x,y,z\in Z$. Suppose that $\phi(\alpha),\phi(\beta)$ are genus-$1$ curves, that $\phi(z)$ is a genus-$2$ curve, and that $\phi(\alpha)$ and $\phi(\beta)$ are in the genus-$2$ subsurface bounded by $\phi(z)$. Then $\phi(\alpha),\phi(\beta)$ is a sharing pair. 
\end{cor}

\begin{proof}
By Lemma \ref{lem:sharing-pair}, it suffices to show that $\phi(x),\phi(y),\phi(z),\phi(w)$ certify that $\phi(\alpha),\phi(\beta)$ is a sharing pair. This is true because $\phi(\alpha)$ and $\phi(\beta)$ are in the genus-2 subsurface bounded by $\phi(z)$ by assumption and because incidence-preserving maps preserve whether two curves are disjoint or intersect. 
\end{proof}


\subsection{(Extended) mapping class group}

The \emph{mapping class group} $\Mod(\Sigma):=\pi_0(\text{Homeo}^+(\Sigma))$ is the group of isotopy classes of orientation-preserving homeomorphisms of $\Sigma$. The \emph{extended mapping class group} $\Mod^\pm(\Sigma)=\pi_0(\text{Homeo}(\Sigma))$ is an index-2 super-group that includes orientation-reversing mapping classes. It is necessary for us to consider $\Mod^\pm(\Sigma)$ because the automorphism group of $\mathcal C(\Sigma)$ is $\Mod^\pm(\Sigma)$ \cite{ivanov}. 

\section{The finite rigid set of separating curves} \label{sec:rigid-set}

In this section we give the construction of the subcomplex $Y^s\subset\mathcal C^s(\Sigma)$ in Theorem \ref{thm:main}. The complex $Y^s$ will be defined as $Y\cap\mathcal C^s(\Sigma)$ for a certain subcomplex $Y\subset\mathcal C(\Sigma)$.

We define $Y$ as the subcomplex spanned by a certain set of vertices. For ease of exposition, we will not distinguish between $Y$ and its set of vertices. In particular, we will define 
\[Y=C\cup S\cup B\cup U\cup V,\]
where $C,S,B,U,V$ are certain collections of curves on $\Sigma$. We proceed now to define these curves. 

%
%

\paragraph{Chain, separating, and bounding pair curves.} 

Let $C=\{c_0,\ldots,c_{2g+1}\}$ denote the chain of curves in Figure \ref{fig:chain}. 


\begin{figure}[h!]
\labellist
\pinlabel $c_0$ at 260 2595
\pinlabel $\cdots$ at 410 2595
\pinlabel $c_1$ at 315 2575
\pinlabel $c_2$ at 340 2605
\pinlabel $c_3$ at 365 2575
\pinlabel $c_{2g-1}$ at 440 2575
\pinlabel $c_{2g}$ at 500 2595
\pinlabel $c_{2g+1}$ at 400 2615
\pinlabel $s_{\{0,1\}}$ at 620 2575
\pinlabel $b_{[0,4]}^+$ at 690 2610
\pinlabel $b_{[0,4]}^-$ at 690 2570
\endlabellist
\centering
\includegraphics[scale=.8]{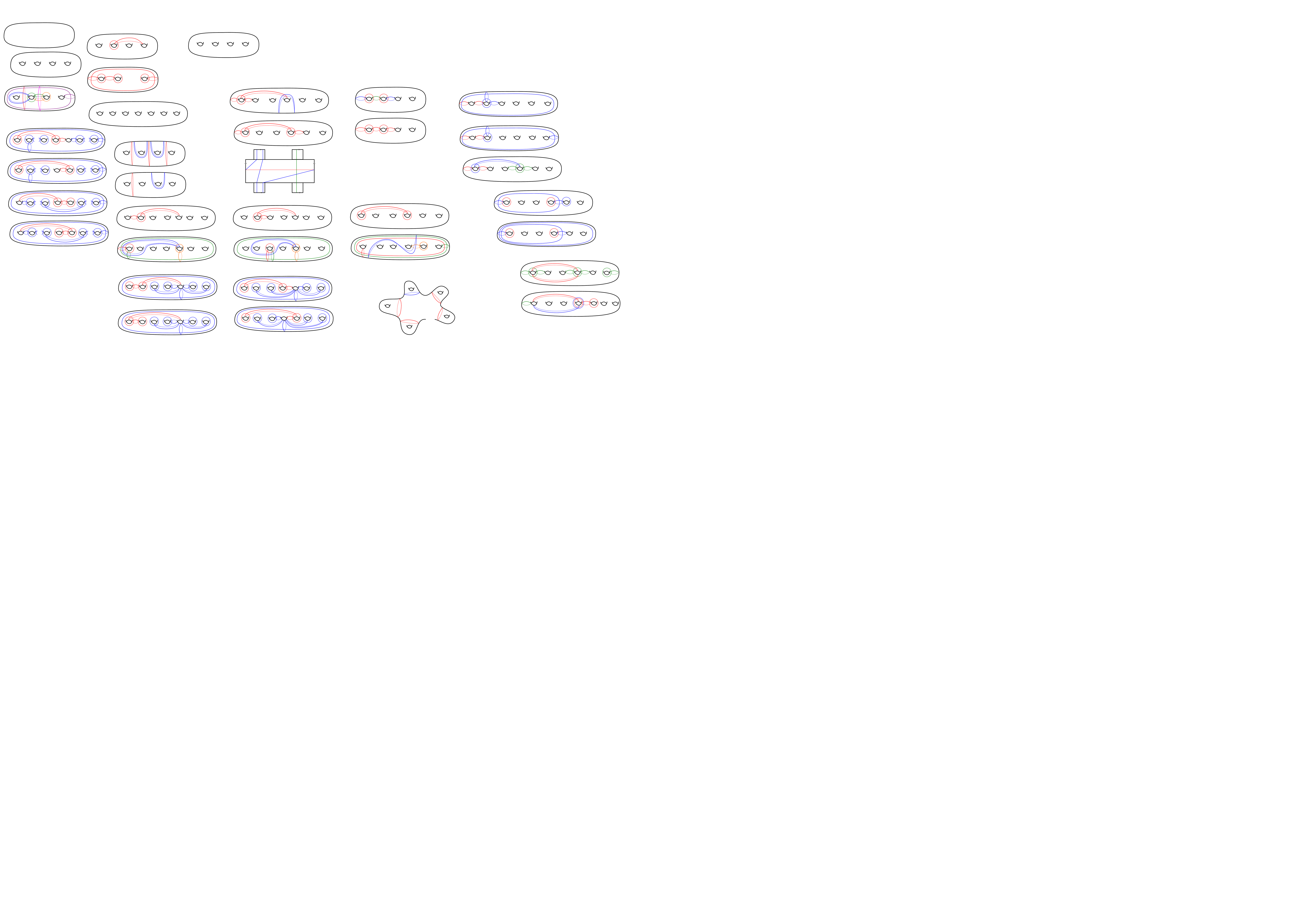}\hspace{.5in}\includegraphics[scale=.8]{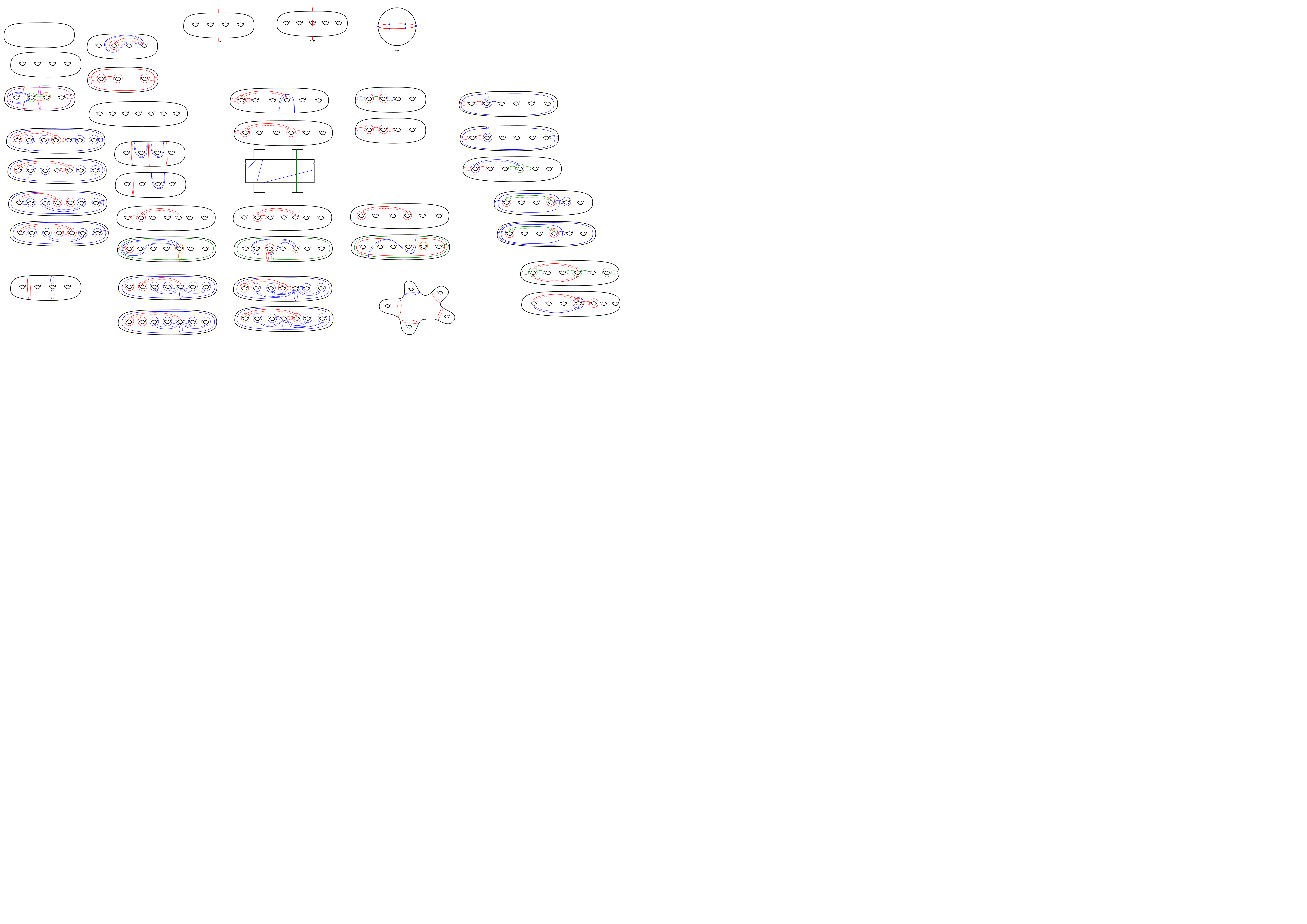}
\caption{Left: chain curves in $C$. Right: sample curves in $S$ and $B$.}
\label{fig:chain}
\end{figure}
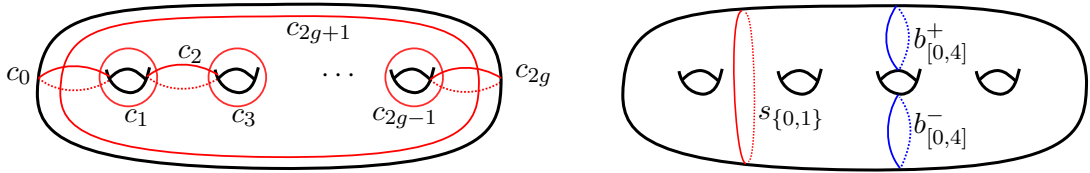

It is convenient to identify $\{0,\ldots,2g+1\}\cong\Z/(2g+2)\Z$, and define an \emph{interval} $J\subset\{0,\ldots,2g+1\}$ to be a sequence of consecutive equivalence classes. Sometimes we write $[i,j]$ for the interval $\{i, i+1,\ldots, j\}$. For an interval $J$, define $c_J=\bigcup_{j\in J}c_j$ and 
\[s_{J}=\partial N(c_J).\]
When $|J|\ge2$ is even, $s_J$ is a separating curve of genus $|J|/2$. When $|J|>1$ is odd, $s_J$ is a bounding pair. Note that the union of the even chain curves $c_0\cup c_2\cup\dots\cup c_{2g}$ separates $\Sigma$ into $\Sigma_e^+$ and $\Sigma_e^-$. Similarly, the union of the odd chain curves $c_1\cup c_3\cup\dots \cup c_{2g+1}$ cuts $\Sigma$ into $\Sigma_o^+$ and $\Sigma_o^-$. Let $b_J^+$ be the curve in $s_J$ that lies in $\Sigma_o^+$ or $\Sigma_e^+$, and let $b_J^-$ be the curve in $s_J$ that lies in $\Sigma_o^-$ or $\Sigma_e^-$. Then $s_J=b_J^+\cup b_J^-$. Write 
$S=\{s_J: |J| \text{ even}\}$ and $B=\{b_J^{\pm}: |J|\text{ odd}\}$ for the collection of these curves. See Figure \ref{fig:chain} for an example.

\paragraph{Genus-1 curves.} Next we consider the following collection $U$ of genus-1 curves. Fix $J=[i,j]$ with $|J|$ is odd. Noting that $i(b_J^\pm,c_{i-1})=1=i(b_J^\pm,c_{j+1})$, we define 
\[u_{i-1,J}^\pm=\partial N(c_{i-1}\cup b_J^\pm)\>\>\>\text{ and }\>\>\>u_{j+1,J}^\pm=\partial N(c_{j+1}\cup b_J^\pm),\]
and we let $U$ denote the collection of all such curves. Figure \ref{fig:U-curve} contains an example.

\begin{figure}[h!]
\labellist
\pinlabel $u_{3,[4,6]}^+$ at 320 2675
\pinlabel $v_{[3,4]*[4,8]*\{9\}}^+$ at 690 2670
\endlabellist
\centering
\includegraphics[scale=.7]{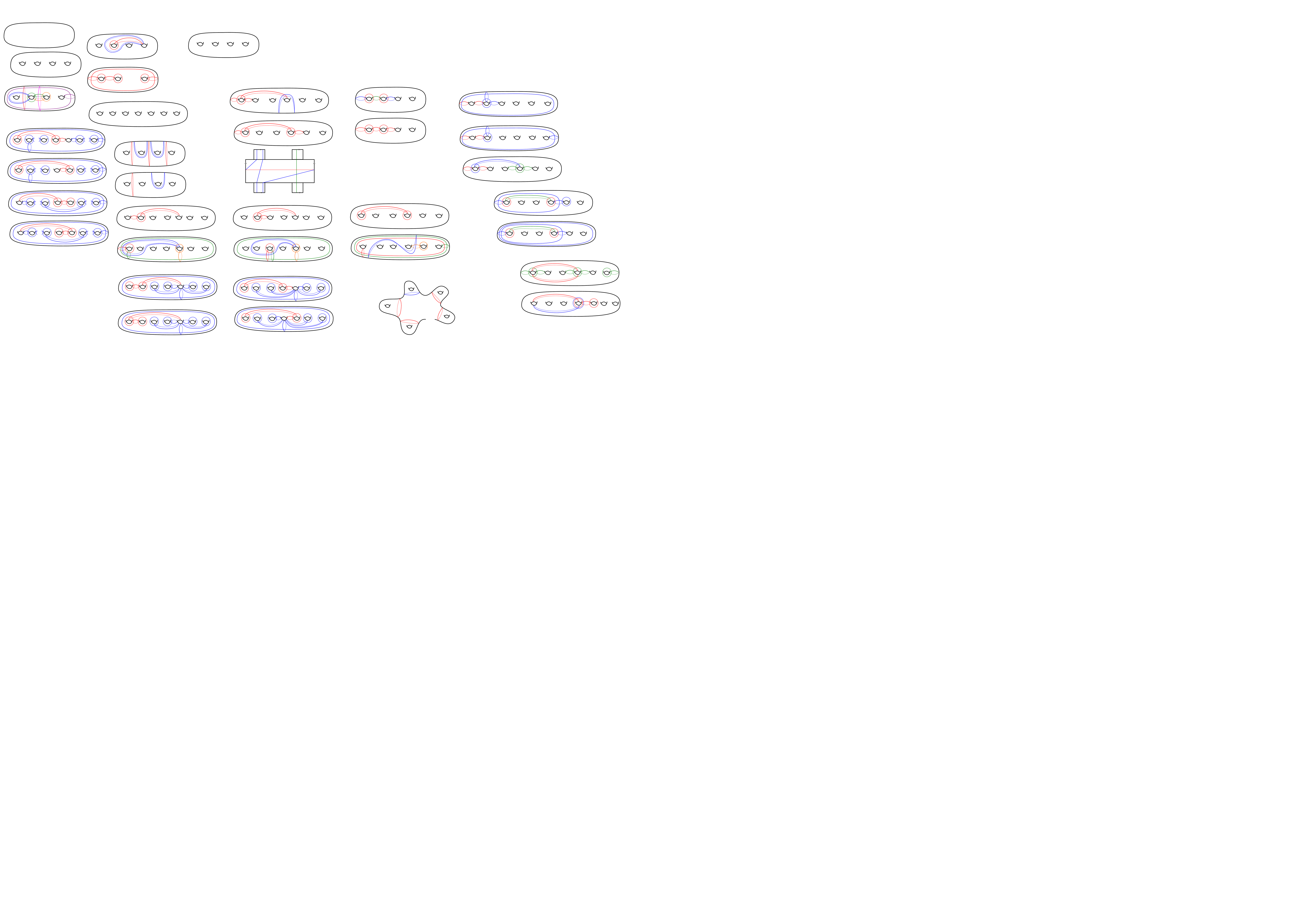}\hspace{.4in}\includegraphics[scale=.7]{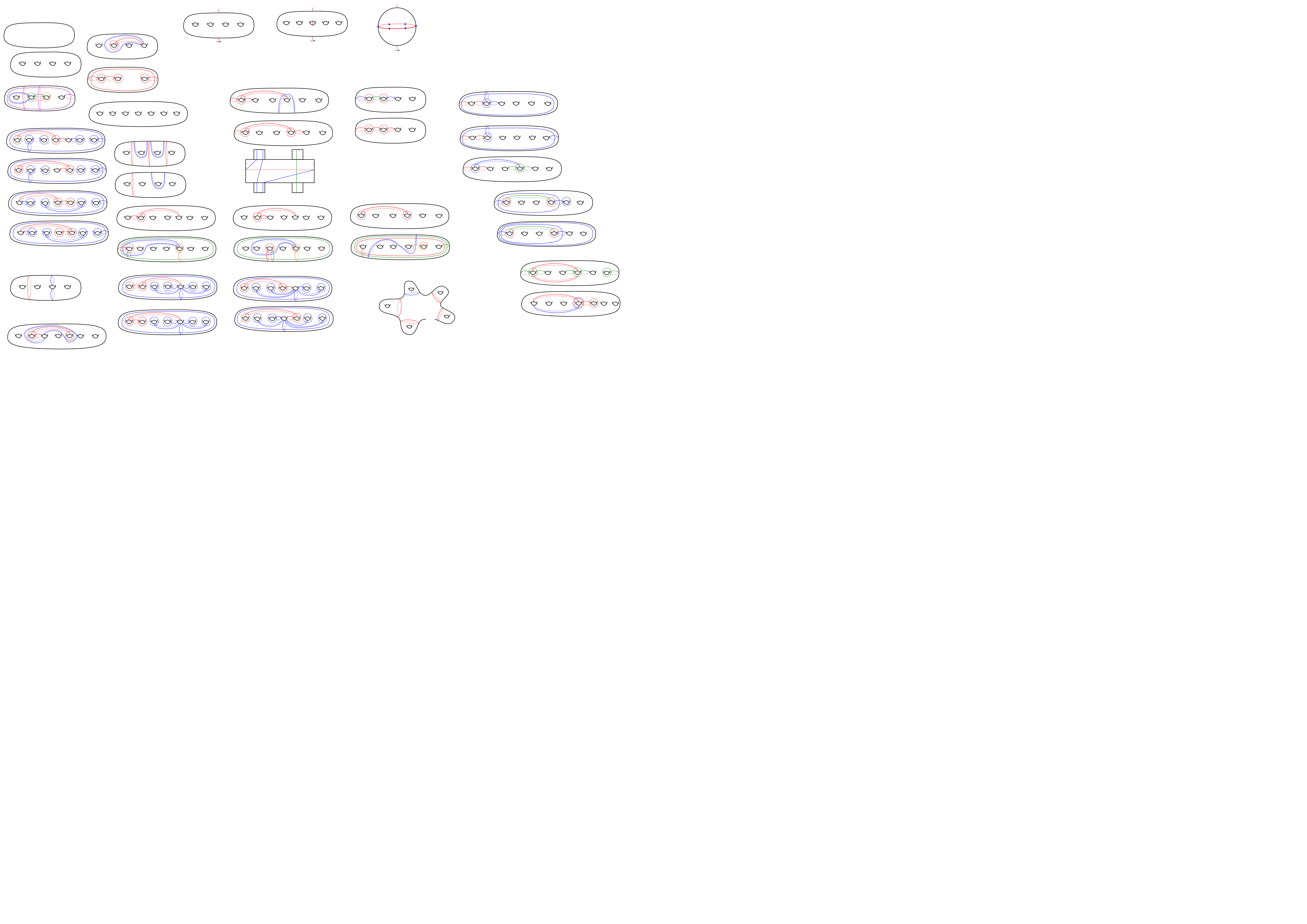}
\caption{Sample curves in $U$ (left) and $V$ (right).}
\label{fig:U-curve}
\end{figure}



\paragraph{Genus-2 curves.} Finally, we consider the following collection $V$ of genus-2 curves. These curves are obtained by taking $v=\partial N(x_1\cup\cdots\cup x_4)$, where $x_1,\ldots,x_4$ is a chain of length 4 with three of the $x_j$ in $C$ and one in $B$. There are various cases depending on the position of the curve in $B$ in the chain. 

To make this precise, write $J=[i,j]$, where $|J|$ is odd. A \emph{predecessor interval} (resp.\ \emph{successor interval}) for $J$ is an interval of the form $[*,i-1]$ or $[i-1,*]$ (resp.\ $[j+1,*]$ or $[*,j+1]$). If $\pi$ and $\sigma$ are predecessor and successor intervals, respectively, and $|\pi|+|\sigma|=3$, then we we define 
\begin{equation}\label{eqn:genus2}v_{\pi*J*\sigma}^\pm=\partial N(c_\pi\cup b_J^\pm\cup c_\sigma),\end{equation}
and we define $V$ as the collection of all such curves. Note that we allow one of $\pi$ or $\sigma$ to be empty. The assumption $|\pi|+|\sigma|=3$ implies that $v_{\pi*J*\sigma}^\pm$ is a genus-2 curve. See Figure \ref{fig:U-curve} for an example. 

%
%
%
%
%
%
%

Now we define 
\begin{equation}\label{eqn:rigid-set}
Y=C\cup S\cup B\cup U\cup V
\end{equation}
and 
\begin{equation}\label{eqn:rigid-set-sep}
Y^s=S\cup U\cup V
\end{equation}

\begin{rmk}\label{rmk:injective}
It is easy to check that $Y^s$ satisfies the assumption of Lemma \ref{lem:injective-criterion}, so we observe that any incidence-preserving map $\phi:Y^s\rightarrow\mathcal C^s(\Sigma)$ is injective. 
\end{rmk}

The following is a more precise statement of Theorem \ref{thm:main}. 

\begin{thm}\label{thm:rigid-set}
Fix $g\ge4$. Let $Y^s\subset\mathcal C^s(\Sigma_g)$ be the subcomplex defined by (\ref{eqn:rigid-set-sep}). Any incidence-preserving map $\phi:Y^s\rightarrow\mathcal C^s(\Sigma)$ is induced by a mapping class. 
\end{thm}

\paragraph{Symmetries of $Y$.} To prove Theorem \ref{thm:rigid-set}, it will be helpful to observe that $Y$ has several symmetries. This will significantly reduce the casework in our arguments. 

\begin{lem}\label{lem:permute}
There are mapping classes $r,s\in\Mod(\Sigma)$ that preserve $C=\{c_0,\ldots,c_{2g-1}\}$ and act on $C$ by 
\[r(c_i)=c_{i+1}\>\>\>\>\text{ and }s(c_i)=c_{-i-2}\]
with indices interpreted modulo $2g+2$.
\end{lem}

Here $s$ acts on $C$ as the ``reflection" fixing $c_g$ and $c_{2g+1}$. 

\begin{proof}[Proof of Lemma \ref{lem:permute}]
The homeomorphism $s$ is pictured in Figure \ref{fig:permute}; the description depends on whether the genus of $\Sigma$ is odd or even. 

\begin{figure}[h!]
\labellist
\pinlabel $s$ at 670 2700
\pinlabel $\cdots$ at 625 2760
\pinlabel $\cdots$ at 715 2760
\pinlabel $\cdots$ at 848 2760
\pinlabel $\cdots$ at 972 2760
\pinlabel $s$ at 910 2700
\pinlabel $\bar r$ at 1095 2700
\endlabellist
\centering
\includegraphics[scale=.7]{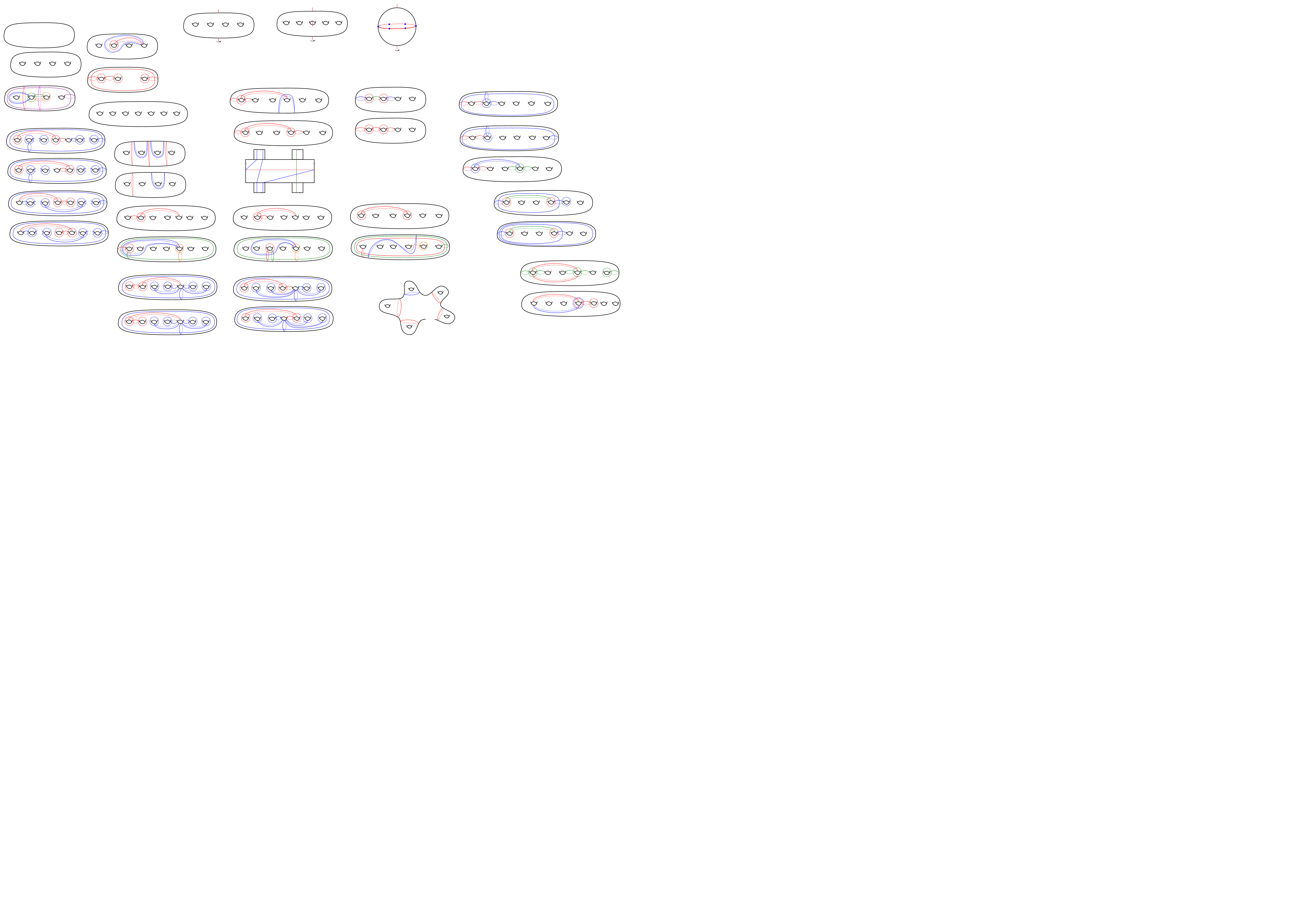}
\includegraphics[scale=.7]{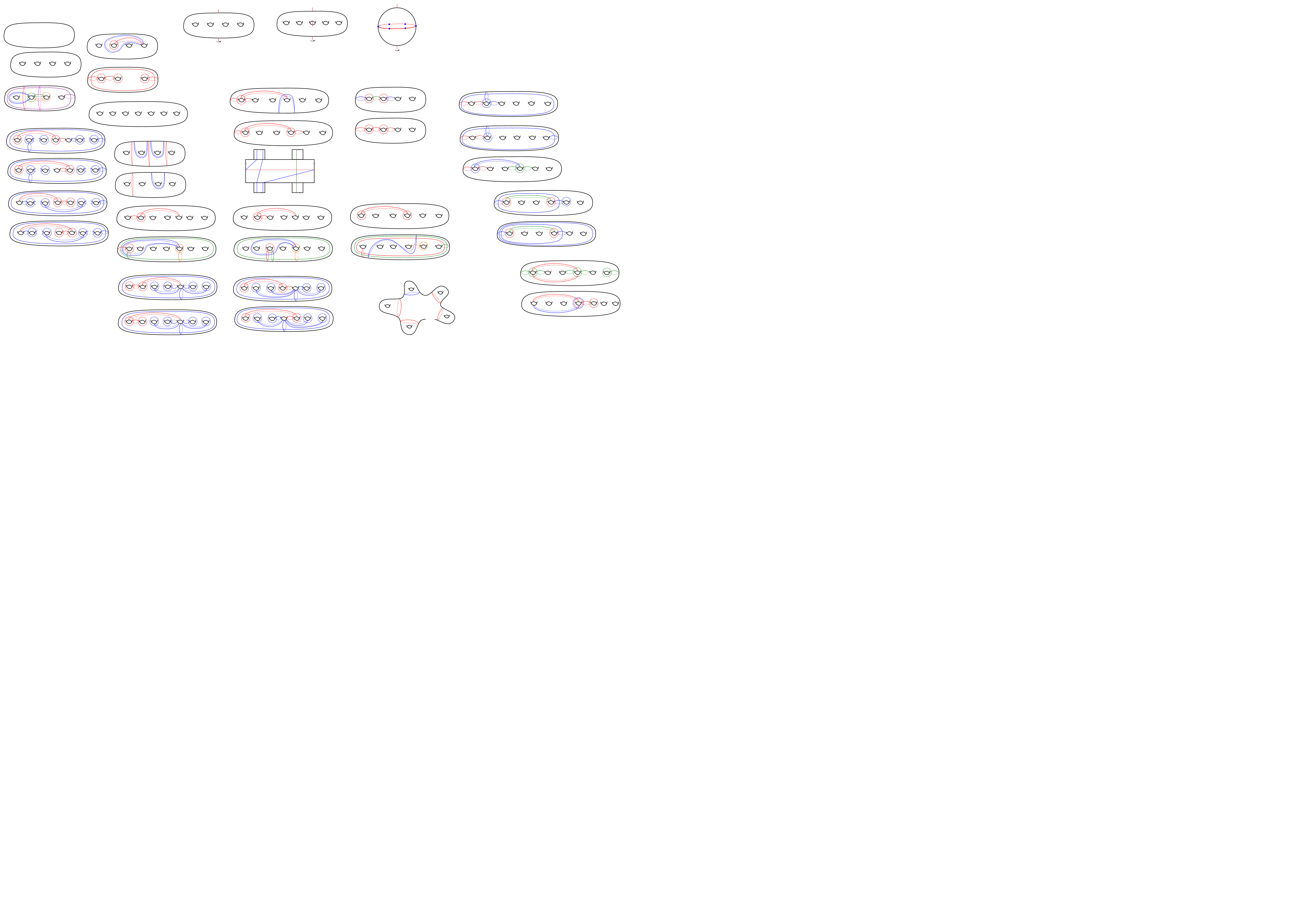}
\includegraphics[scale=.7]{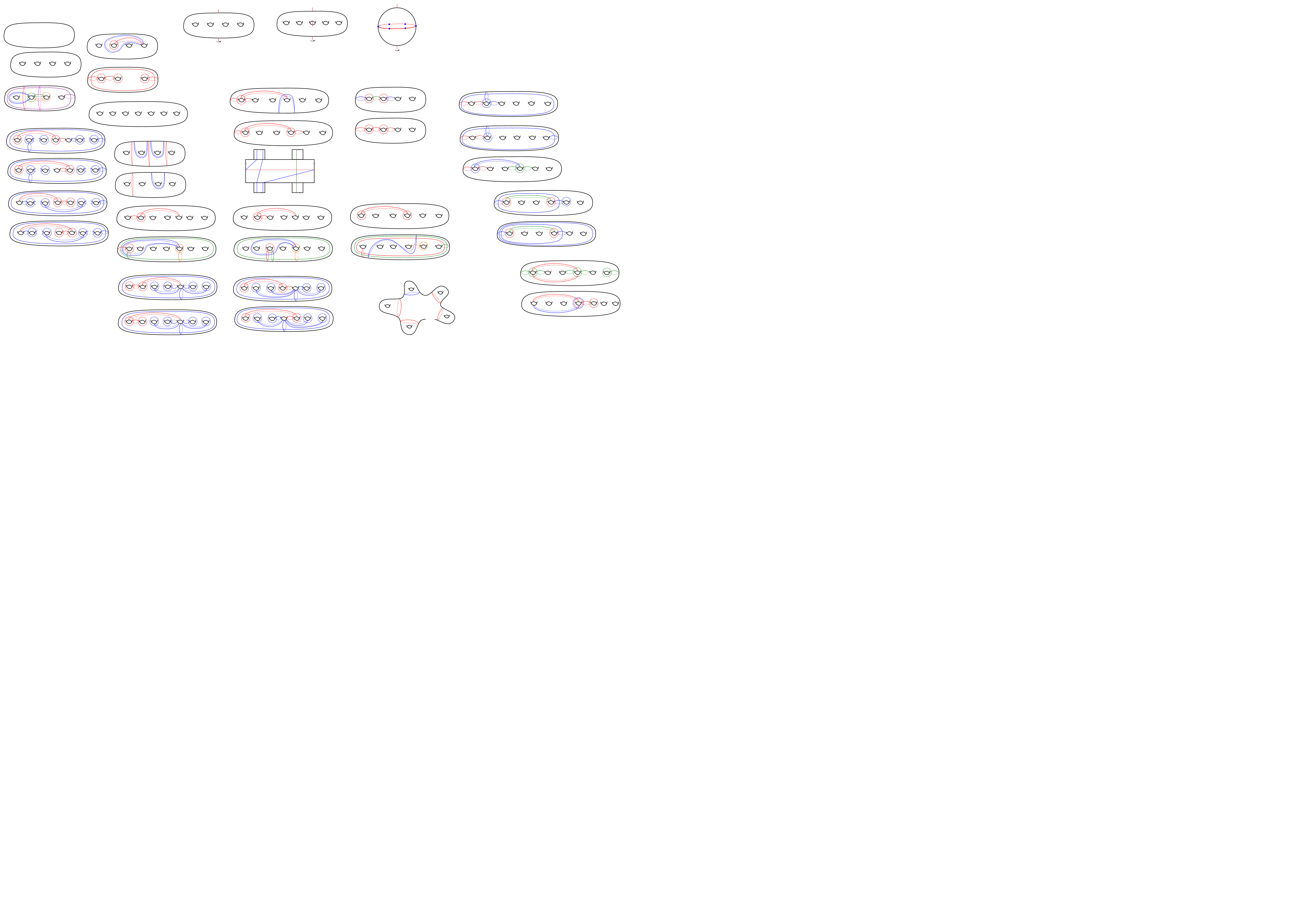}
\caption{Symmetries of $\Sigma$ and $\Sigma/\pair{\iota}$}
\label{fig:permute}
\end{figure}

To define $r$, observe that $C$ is invariant by a hyperelliptic involution $\iota$. The image of $C$ on $\Sigma/\pair{\iota}\cong S^2$ is a sequence of arcs between the branch points, and an order-($2g+2$) rotation $\bar r$ of $S^2$ that permutes the branched points, as pictured in Figure \ref{fig:permute}, lifts to a homeomorphism $r$ of $\Sigma$ with the desired properties. 
\end{proof}

\begin{rmk}\label{rmk:permute} 
By the construction of $Y$, a mapping class that preserves $C$ will preserve $Y$. Therefore, we can also view $r,s$ as automorphisms of $Y$. They generate a dihedral group. In addition, there is an obvious hyperelliptic involution that preserves $Y$. It acts trivially on $C$, but interchanges $b^+\leftrightarrow b^-$ and $u^+\leftrightarrow u^-$ for $b^\pm\in B$ and $u^\pm\in U$. 
\end{rmk}

\section{$Y$ is rigid in $\mathcal C(\Sigma)$}\label{sec:Y-rigid}

In this section we start the proof of Theorem \ref{thm:main}, as outlined in the introduction for the set $Y^s$ defined in \S\ref{sec:rigid-set}. In what follows $\Sigma$ denotes a fixed closed surface of genus $g\ge4$. 

\begin{prop}\label{prop:Y-rigid}
Let $Y\subset\mathcal C(\Sigma)$ be the subcomplex defined in (\ref{eqn:rigid-set}). Any locally injective map $\Phi:Y\rightarrow\mathcal C(\Sigma)$ is induced by an extended mapping class. 
\end{prop}

Using the notation of \S\ref{sec:rigid-set}, we set
\begin{equation}\label{eqn:AL-rigid-set}
X=C\cup S\cup B\cup U.
\end{equation}
This is the finite rigid set of \cite[Thm.\ 5.1]{AL}. 

\begin{thm}[Aramayona--Leininger]\label{thm:AL}
Let $X\subset\mathcal C(\Sigma)$ be the subcomplex defined in (\ref{eqn:AL-rigid-set}). Any locally injective map $\Phi:X\rightarrow\mathcal C(\Sigma)$ is induced by an extended mapping class. 
\end{thm}


\begin{proof}[Proof of Proposition \ref{prop:Y-rigid}]
Recall from \cite[Defn.\ 2.1]{AL} that a vertex $v$ of $\mathcal C(\Sigma)$ is \emph{uniquely determined} by a $A\subset\mathcal C(S)$ if $v$ is the unique isotopy class of curve that can be realized disjointly from every curve in $A$. 

{\it Claim.} Each $v\in V$ is uniquely determined by a subset $A\subset X$. 

First we explain why the claim implies the proposition. Let $\Phi:Y\rightarrow\mathcal C(\Sigma)$ be a locally injective map. Since $Y\supset X$ and $X$ is rigid by \cite[Thm.\ 5.1]{AL} (see Theorem \ref{thm:AL} above), there exists $h\in\Mod^\pm(\Sigma)$ such that $\rest{h}{X}=\rest{\Phi}{X}$. To show $\rest{h}{Y}=\Phi$, it remains to show that $h(v)=\Phi(v)$ for $v\in V$. By the claim, for $v\in V$, there exists $A\subset X$ such that $v$ is uniquely determined by $A$. Since $\Phi$ is simplicial, $\Phi(v)$ is disjoint from $\Phi(A)$. Since $h\in\Mod^\pm(\Sigma)$, there is a unique vertex $v'\in\mathcal C(\Sigma)$ that is disjoint from $h(A)$, and $h(v)=v'$. Thus $\Phi(v)=v'=h(v)$, as desired. 

{\it Proof of Claim.} Fix $v=v_{\pi*J*\sigma}^\epsilon$ in $V$, where $\epsilon\in\{\pm\}$. Using the automorphisms defined in Lemma \ref{lem:permute}, we can assume that $|\pi|<|\sigma|$ and $J=[2,\ldots,2k]$. For concreteness we take $\epsilon=+$; the case $\epsilon=-$ is similar. 

We want to find a collection of curves in $X$ that fill $\Sigma_0:=S\setminus v$. In fact, we will show that there are two chains in $\Sigma_0$ of length $4$ and $2(g-2)$ that fill the two components of $\Sigma_0$ (which have genus $2$ and $g-2$). For the chain of length-4, we use the chain used to define $v$. A choice for the other chain is pictured in Figure \ref{fig:Y-rigid}.
\begin{figure}[h!]
\labellist
\pinlabel $...$ at 112 2405
\pinlabel $...$ at 267 2405
\pinlabel $...$ at 490 2405
\pinlabel $...$ at 645 2406
\pinlabel $...$ at 109 2317
\pinlabel $...$ at 265 2317
\pinlabel $...$ at 486 2318
\pinlabel $...$ at 642 2318

\endlabellist
\centering
\includegraphics[scale=.6]{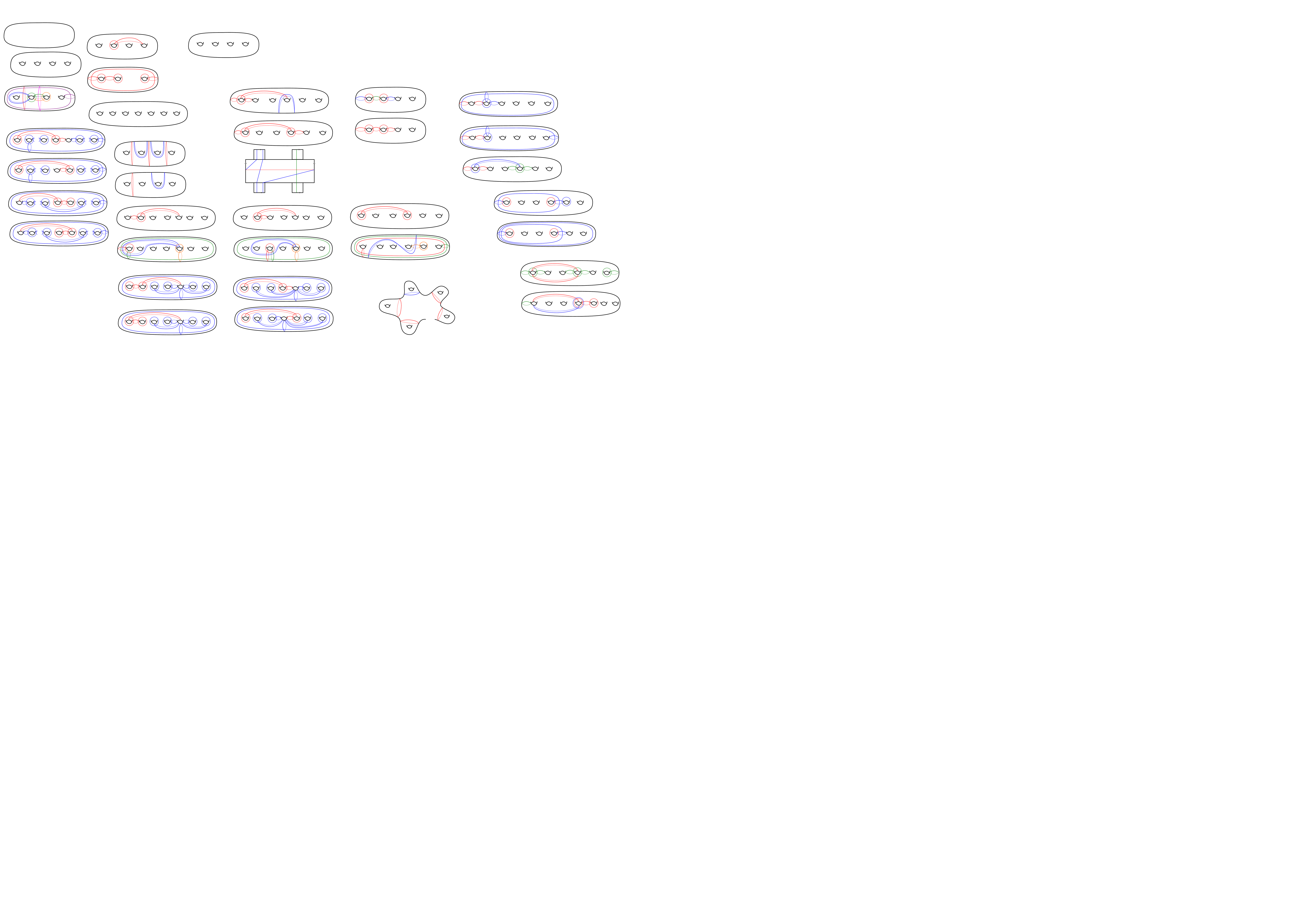}\hspace{.5in}
\includegraphics[scale=.6]{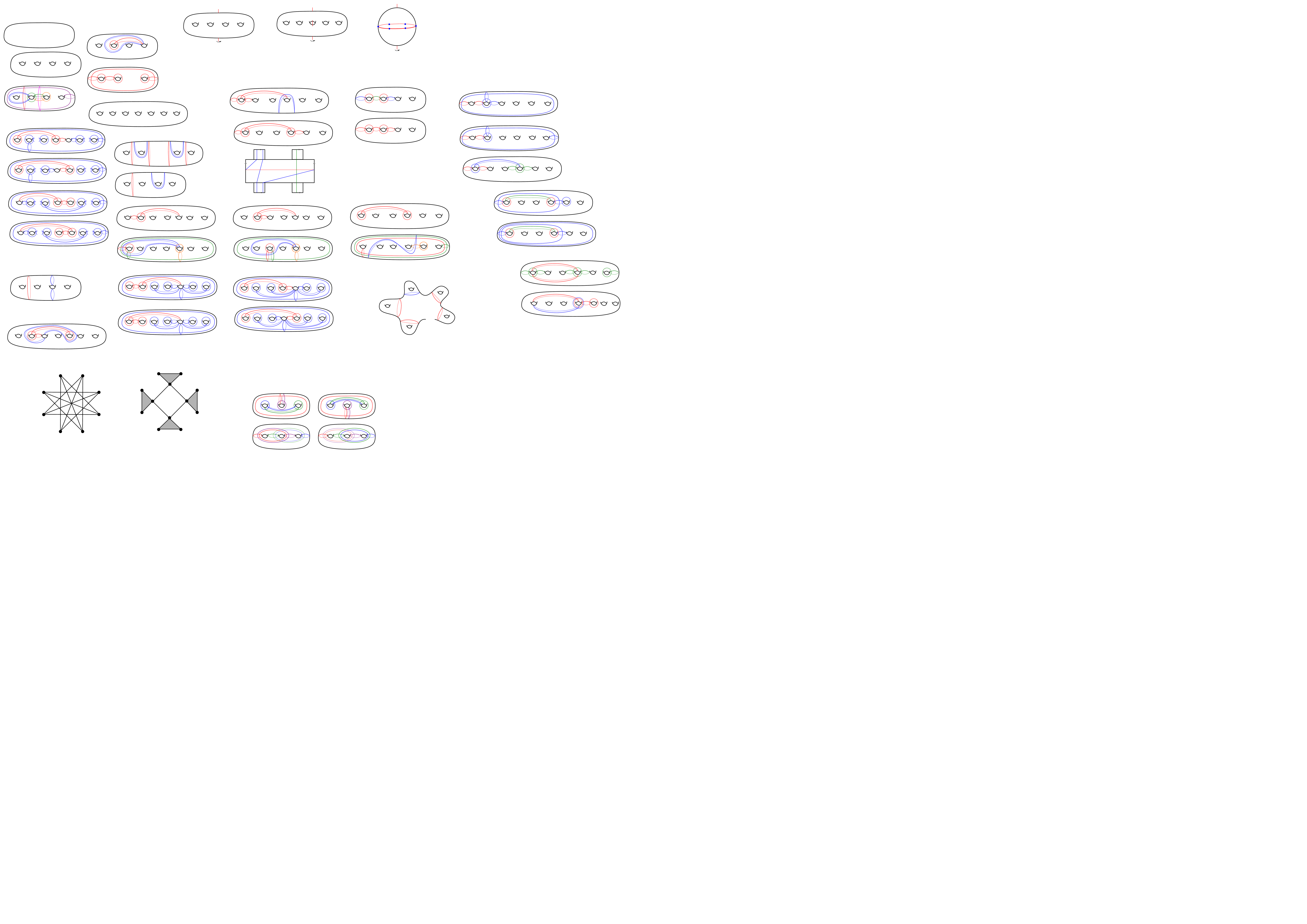}
\includegraphics[scale=.6]{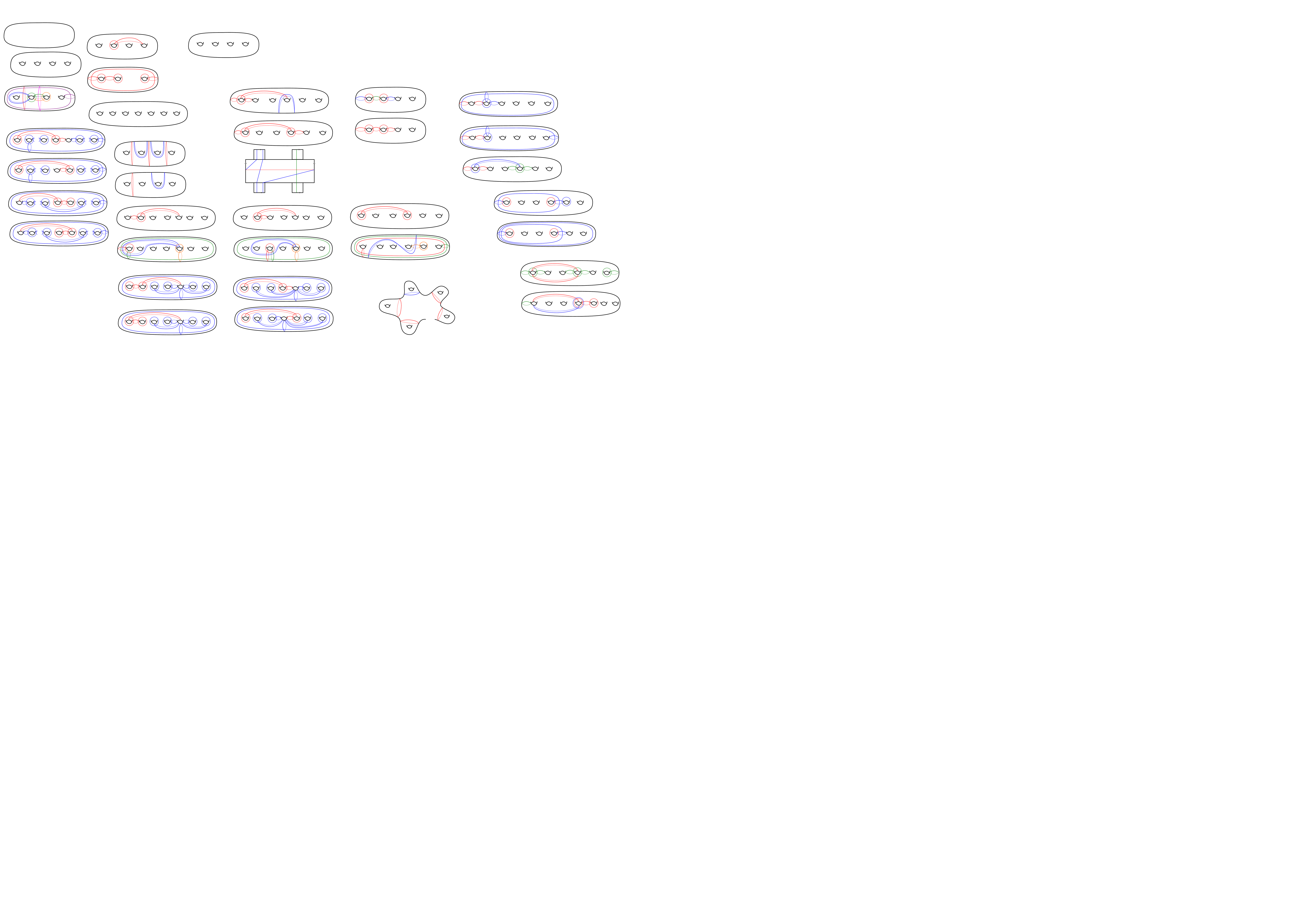}\hspace{.5in}
\includegraphics[scale=.6]{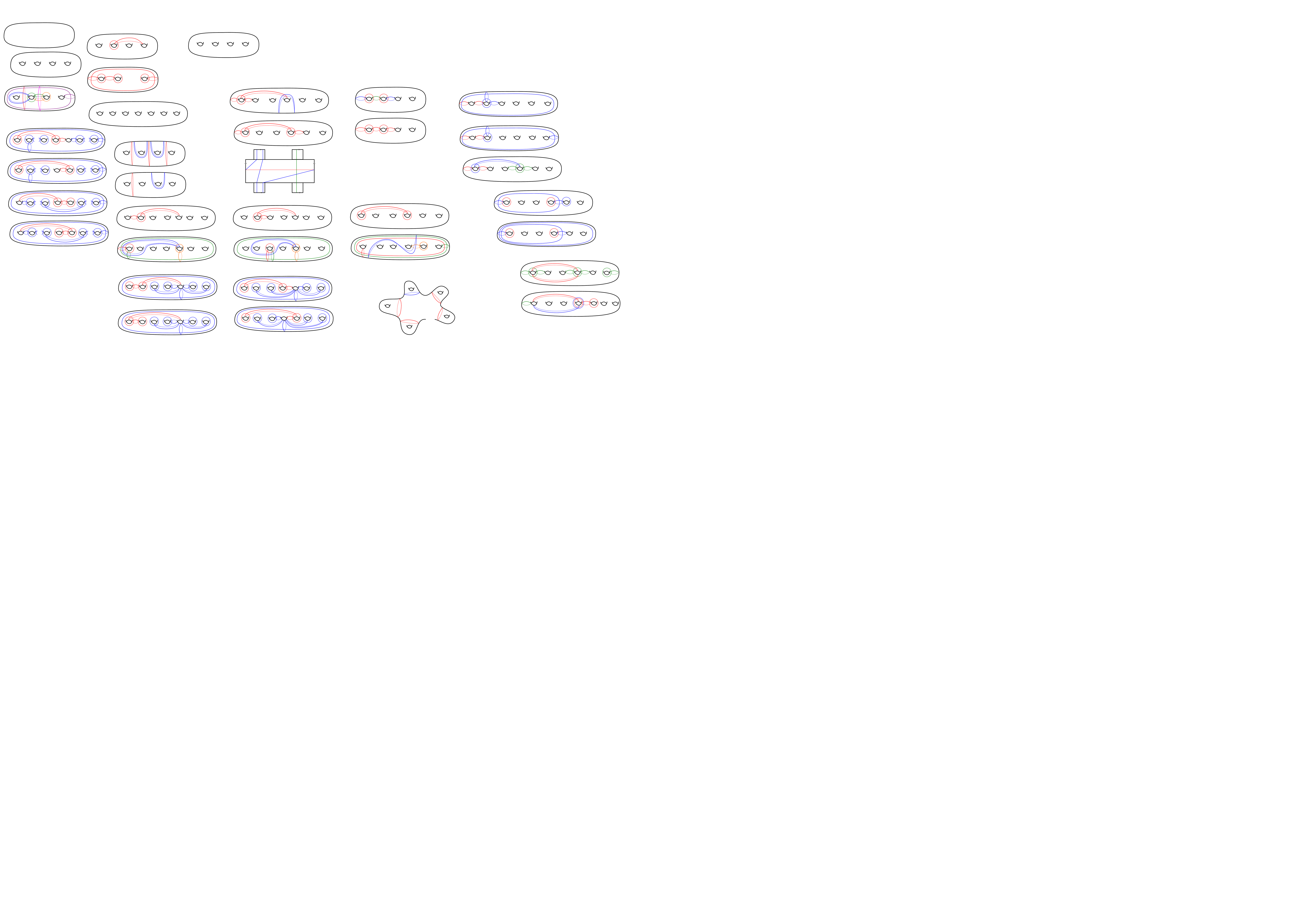}
\caption{A chain of length $2g-4$ that fills the complement of the chain of length 4. Top row: the case $|\pi|=1$. Bottom row: the case $|\pi|=0$.}
\label{fig:Y-rigid}
\end{figure}
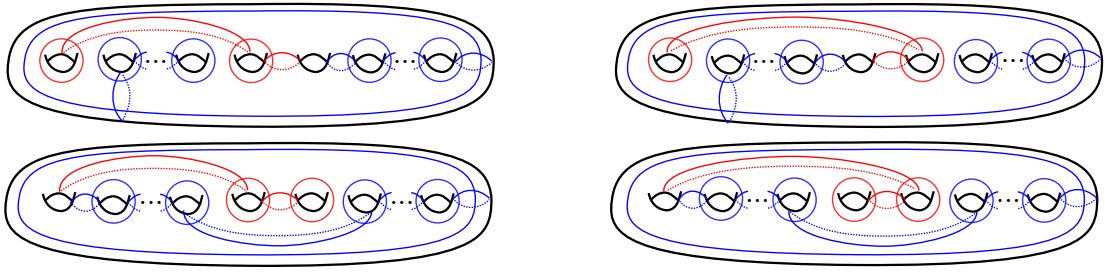
\end{proof}

\section{Incidence-preserving maps $Y^s\rightarrow\mathcal C^s(\Sigma)$} 

The goal of this section is to prove the following two theorems, which say that incidence-preserving maps $\phi:Y^s\rightarrow\mathcal C^s(\Sigma)$ preserve genus and (certain) sharing pairs. 

\begin{thm}[Genus is preserved]\label{thm:genus}
Fix $Y^s$ as defined in (\ref{eqn:rigid-set-sep}), and let $\phi:Y^s\rightarrow\mathcal C^s(\Sigma)$ be an incidence-preserving map. For $\alpha\in Y^s$, the curves $\alpha$ and $\phi(\alpha)$ have the same genus. 
\end{thm}

\begin{thm}[Sharing pair is preserved]\label{thm:sharing-pair}
Let $\phi:Y^s\rightarrow\mathcal C^s(\Sigma)$ be an incidence-preserving map, and let $(\alpha,\beta)$ be a sharing pair with one of the following spines. 
\begin{enumerate}
\item[(i)] $(c_i,c_{i+1},c_{i+2})$
\end{enumerate}
For any interval $[i,j]$ of odd length $\ge3$, 
\begin{enumerate} 
\item[(ii)] $(c_{i-2},c_{i-1},b_{[i,j]}^\pm)$
\item[(iii)] $(c_i,c_{i-1},b_{[i,j]}^\pm)$
\item[(iv)] $(b_{[i,j]}^\pm,c_{j+1},c_{j+2})$
\item[(v)] $(b_{[i,j]}^\pm,c_{j+1},c_{j})$
\item[(vi)] $(c_{i-1},b_{[i,j]}^\pm,c_{j+1})$
\end{enumerate} 
Then $(\phi(\alpha),\phi(\beta))$ is also a sharing pair. 
\end{thm}

We prove these results in a ``bootstrapping" fashion. We prove, in order, the following statements. 
\begin{enumerate}[(A)]
\item The genus is preserved for curves in $X^s$.
\item Sharing pairs of type (i) are preserved. 
\item The genus is preserved for $v_{\pi*J*\sigma}^\pm\in V$ such that either $\pi$ or $\sigma$ is empty. 
\item Sharing pairs of type (ii)--(v) are preserved. 
\item The genus is preserved for $v_{\pi*J*\sigma}^\pm\in V$ such that both $\pi$ and $\sigma$ are nonempty. 
\item Sharing pairs of type (vi) are preserved. 
\end{enumerate} 

Each statement is used to prove the next statement. Together these statements prove Theorems \ref{thm:genus} and \ref{thm:sharing-pair}.

We begin with the proofs. Fix an incidence-preserving map $\phi:Y^s\rightarrow\mathcal C^s(\Sigma)$. 

\boxed{\text{Proof of (A)}}
Fix $\alpha\in X^s$. We want to show that $\alpha$ and $\phi(\alpha)$ have the same genus. Recall that $X^s=S\cup U$. Up to a permutation of $X$ (c.f.\ Remark \ref{rmk:permute}), we can assume $\alpha$ belongs to the collection $\Delta$ of curves in Figure \ref{fig:lem-genus}, which form a maximal simplex of $X^s$, and also a maximal simplex of $\mathcal C^s(\Sigma)$. The curve $\alpha$ gives a decomposition $\Delta=\Delta_L\cup\{\alpha\}\cup\Delta_R$ into curves on either side of $\alpha$. (If $\alpha$ has genus 1, then one of $\Delta_L$ or $\Delta_R$ is empty.)

\begin{figure}[h!]
\labellist
\pinlabel $\cdots$ at 490 2360
\endlabellist
\centering
\includegraphics[scale=.6]{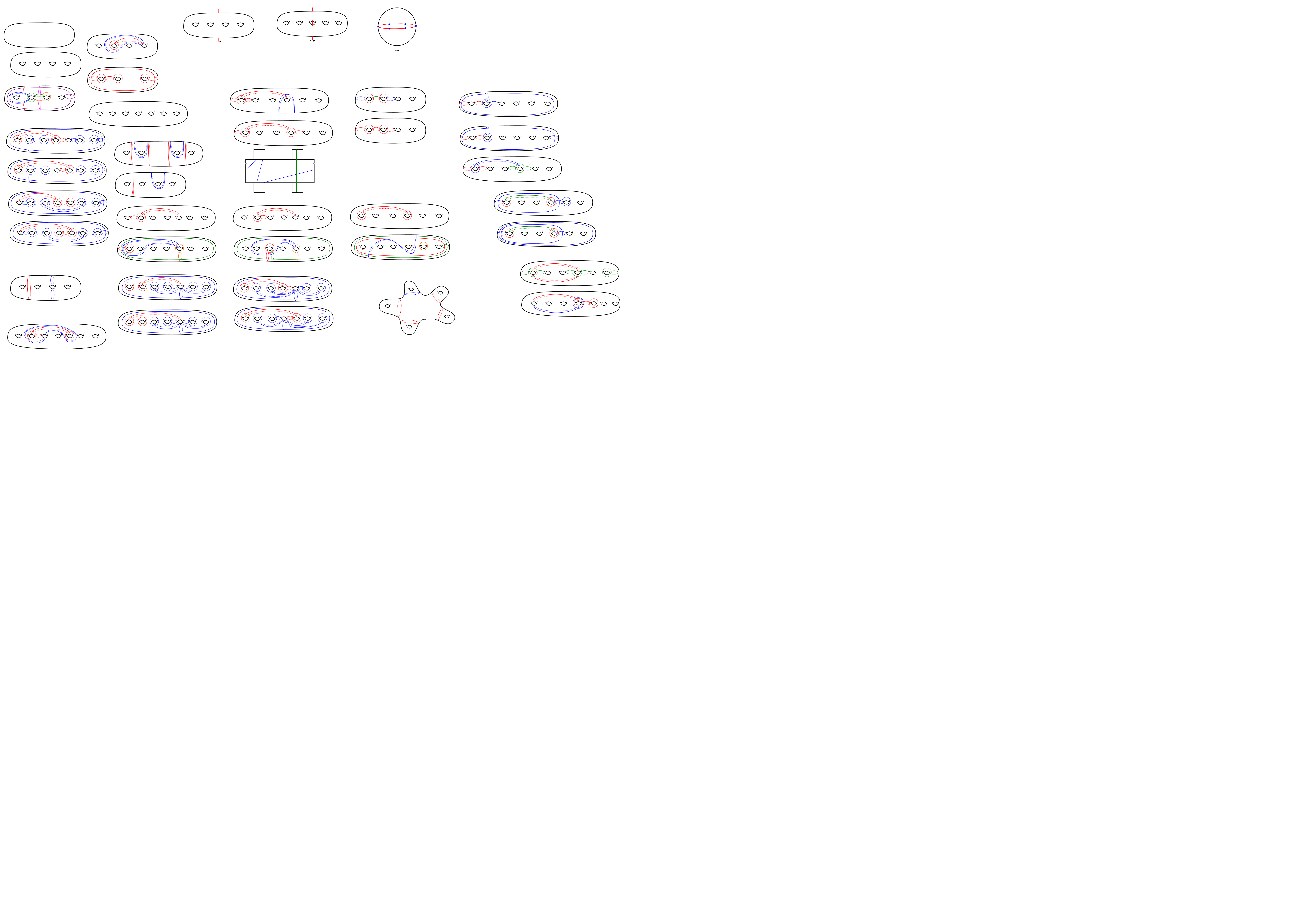}
\caption{Curves in $\Delta$}
\label{fig:lem-genus}
\end{figure}



{\it Claim.} The curves in $\phi(\Delta_L)$ lie on the same side of $\phi(\alpha)$, and similarly for $\phi(\Delta_R)$. 

Given the claim, we deduce that $\alpha$ and $\phi(\alpha)$ have the same genus as follows. The genus of a separating curve $\gamma$ is determined by the maximum number of disjoint separating curves that fit on either of the two components of $\Sigma\setminus\gamma$. The claim implies that this number is the same for $\alpha$ and $\phi(\alpha)$ unless $\phi(\Delta_L)$ and $\phi(\Delta_R)$ lie on the same side of $\phi(\alpha)$. This is not an issue when $\alpha$ has genus 1 (since then one of $\Delta_L$ or $\Delta_R$ is empty). Assume that $\alpha$ has genus $\ge2$ and suppose for a contradiction that $\phi(\Delta_L)$ and $\phi(\Delta_R)$ lie on the same side of $\phi(\alpha)$. Since $\phi(\Delta_L)$ is disjoint from $\phi(\Delta_R)$, this implies that $\phi(\alpha)$ has genus 1. This implies that $\phi(\Delta)$ has more genus-1 curves than $\Delta$, but this is impossible because every maximal collection of separating curves on $\Sigma$ has exactly $g$ curves with genus $1$. 

To prove the claim, it suffices to observe that for each pair $x,x'\in\Delta_L$, there exists $y\in Y^s$ that intersects both $x$ and $x'$ and such that $y$ is disjoint from $\alpha$. This is easy to check and implies that $\phi(y)$ intersects both $\phi(x)$ and $\phi(x')$, and $\phi(y)$ is disjoint from $\phi(\alpha)$. Thus $\phi(x)$ and $\phi(x')$ lie on the same side of $\phi(\alpha)$. The same argument works for $\Delta_R$.

\boxed{\text{Proof of (B)}} The strategy for proving cases of Theorem \ref{thm:sharing-pair} is the same in every case. By Corollary \ref{cor:sharing-pair}, it suffices to find curves $w,x,y,z\in Y^s$ that certify that $(\alpha,\beta)$ is a sharing pair and show the following properties: $\phi(\alpha)$ and $\phi(\beta)$ are genus-1 curves; $\phi(z)$ is a genus-2 curve; $\phi(\alpha)$ and $\phi(\beta)$ lie in the genus-2 subsurface bounded by $\phi(z)$. It is easy enough to find $w,x,y,z$. But to know the genus of $\phi(\alpha),\phi(\beta),\phi(z)$, we need $\alpha,\beta,z$ to belong to the list of cases for which Theorem \ref{thm:genus} has been verified. 

Fix a sharing pair $(\alpha,\beta)$ with spine $(c_i,c_{i+1},c_{i+2})$. By a permutation, it suffices to consider $(c_0,c_1,c_2)$. Here we take $x=c_{\{2g,2g+1\}}$, $y=c_{\{3,4\}}$, $z=c_{[0,3]}$, and $w=c_{\{4,5\}}$, as in Figure \ref{fig:sharing-pair}. Note that here $\alpha=s_{\{0,1\}}$, $\beta=s_{\{1,2\}}$, and $z=c_{[0,3]}$ all belong to $X^s$, so we know that $\phi(\alpha),\phi(\beta),\phi(z)$ have the right genus by (A). Moreover, we know that $\phi(\alpha)$ is in the genus-2 subsurface bounded by $\phi(z)$ by the proof of (A). The same is true for $\phi(\beta)$ because $\phi$ is an incidence-preserving map and $\alpha$ intersects $\beta$. Hence by Corollary \ref{cor:sharing-pair}, the curves $\phi(x),\phi(y),\phi(z),\phi(w)$ certify that $(\phi(\alpha),\phi(\beta))$ is a sharing pair.

\boxed{\text{Proof of (C)}} The strategy for (C) and (E) are similar. We are given a genus-2 curve $\alpha\in V$ and we want to show $\phi(\alpha)$ has genus 2. We will reduce to showing the following two statements. 

\begin{enumerate}
\item[($\dagger$)] There exists a sharing pair $(\beta_1,\beta_2)$ that is contained in the genus-2 component of $\Sigma\setminus\alpha$ and that belongs to the list of sharing pairs for which Theorem \ref{thm:sharing-pair} has been verified. 
\item[($\ddagger$)] There exist distinct, disjoint genus-1 curves $x_1,\ldots,x_{g-2}$ in $X^s$ that lie in the genus-$(g-2)$ component of $\Sigma\setminus\alpha$ and such that for every pair $x_j,x_k$, there exists a third curve $\gamma\in Y^s$ that intersects both $x_j$ and $x_k$ and is disjoint from $\alpha$.
\end{enumerate} 

Once we prove $(\dagger)$ and $(\ddagger)$, we conclude that $\phi(\alpha)$ has genus 2 as follows. First, the curves $\phi(x_1),\ldots,\phi(x_{g-2})$ are disjoint (because $x_1,\ldots,x_{g-2}$ are disjoint) genus-1 curves (because $x_1,\ldots,x_{g-2}$ are in $X^s$), and they lie in one component of $\Sigma\setminus\phi(\alpha)$ (by an argument similar to the argument for (A)). This shows that there is a component of $\Sigma\setminus\phi(\alpha)$ that has genus $\ge g-2$. Then $\phi(\alpha)$ either has genus 1 or 2. 

Suppose for a contradiction that $\phi(\alpha)$ has genus 1. Consider the sharing pair $(\beta_1,\beta_2)$ given in $(\ddagger)$. Since Theorem \ref{thm:sharing-pair} has been verified for $(\beta_1,\beta_2)$, we know that $(\phi(\beta_1),\phi(\beta_2))$ is a sharing pair. Let $\Sigma_j$ be the genus-1 subsurface bounded by $\phi(\beta_j)$, $j=1,2$. Since we assume that $\phi(\alpha)$ has genus 1, the curves $\phi(\beta_1),\phi(\beta_2)$ lie in the component of $\Sigma\setminus\phi(\alpha)$ with genus $g-1$. Since $\phi(\beta_1)$ and $\phi(\beta_2)$ are disjoint from the genus-1 curves $\phi(x_1),\ldots,\phi(x_{g-2})$, we deduce that $\phi(\beta_1),\phi(\beta_2)$ lie in a subsurface $\Sigma'$ with genus 1 and $g-1$ boundary components. See Figure \ref{fig:thm-genus2}. One component $\Sigma''$ of $\Sigma'\setminus\phi(\beta_1)$ has genus 0 and $g$ boundary components. Any arc on $\Sigma''$ with endpoints on a single component of $\partial\Sigma''$  (such as $\phi(\beta_2)\cap\Sigma''$) separates $\Sigma''$, and since the components of $\partial\Sigma''$ are separating in $\Sigma$, it follows that $\Sigma\setminus(\Sigma_1\cup\Sigma_2)$ is disconnected, contradicting the fact that $\phi(\beta_1),\phi(\beta_2)$ is a sharing pair. Therefore, $\phi(\alpha)$ does not have genus 1. 

\begin{figure}[h!]
\labellist
\pinlabel $\phi(\alpha)$ at 1215 1890
\pinlabel $\phi(\beta_1)$ at 1270 1918
\pinlabel $\phi(x_1)$ at 1330 1900
\pinlabel $\phi(x_2)$ at 1328 1870
\pinlabel $\phi(x_{g-2})$ at 1260 1860
\pinlabel $\Sigma''$ at 1275 1890
\pinlabel $...$ at 1320 1855
\endlabellist
\centering
\includegraphics[scale=.8]{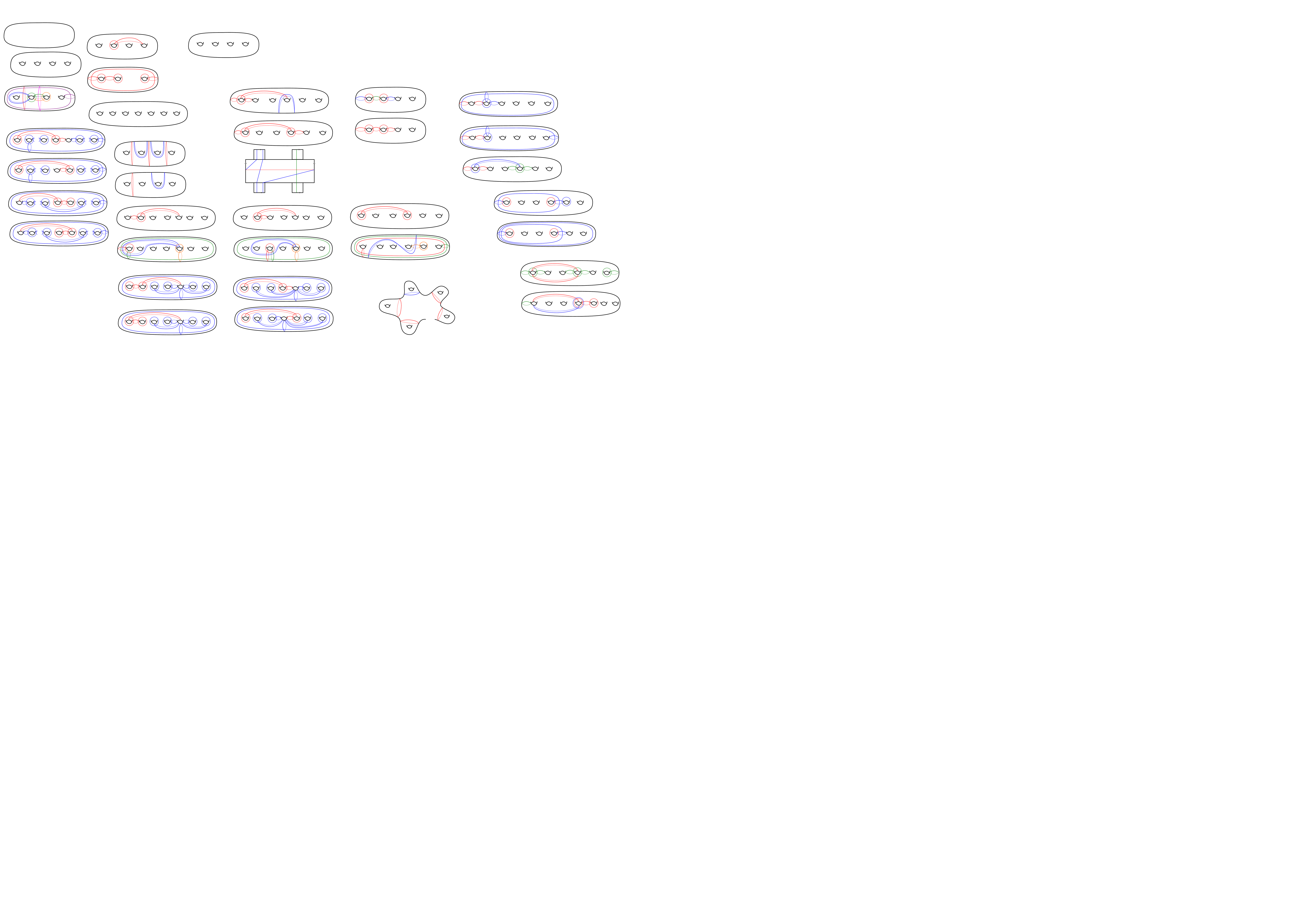}
\caption{}
\label{fig:thm-genus2}
\end{figure}

Now we prove $(\dagger)$ and $(\ddagger)$ for $\alpha=v_{\pi*J*\sigma}^\pm$ with $\pi$ or $\sigma$ empty. After a permutation, we can assume $\sigma=\varnothing$, $\pi=\{1,2,3\}$ and either $J=[2,2k]$ or $J=[4,2k]$. See Figure \ref{fig:statement-C}. 
\begin{figure}[h!]
\labellist
\endlabellist
\centering
\includegraphics[scale=.6]{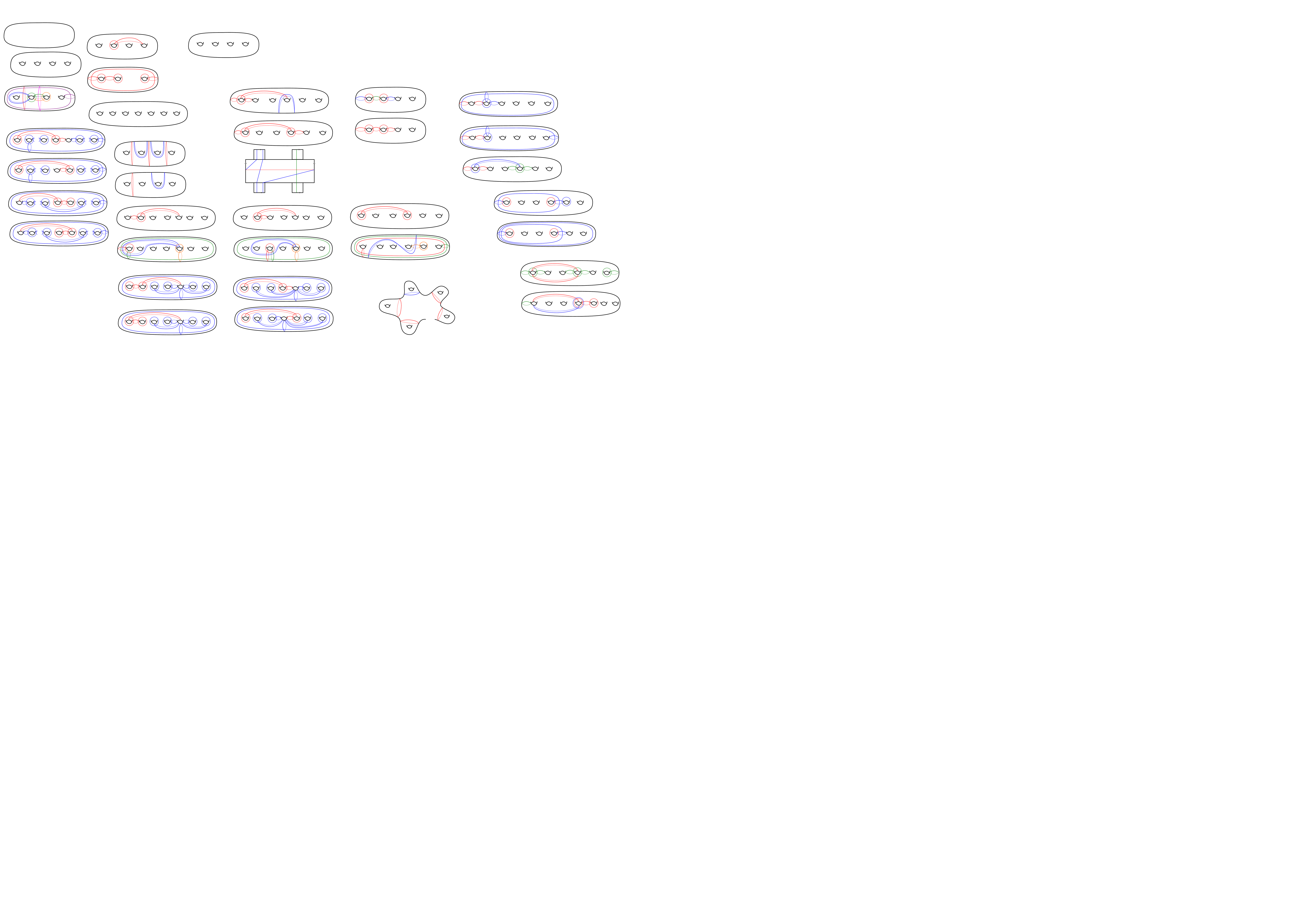}\hspace{.5in}
\includegraphics[scale=.6]{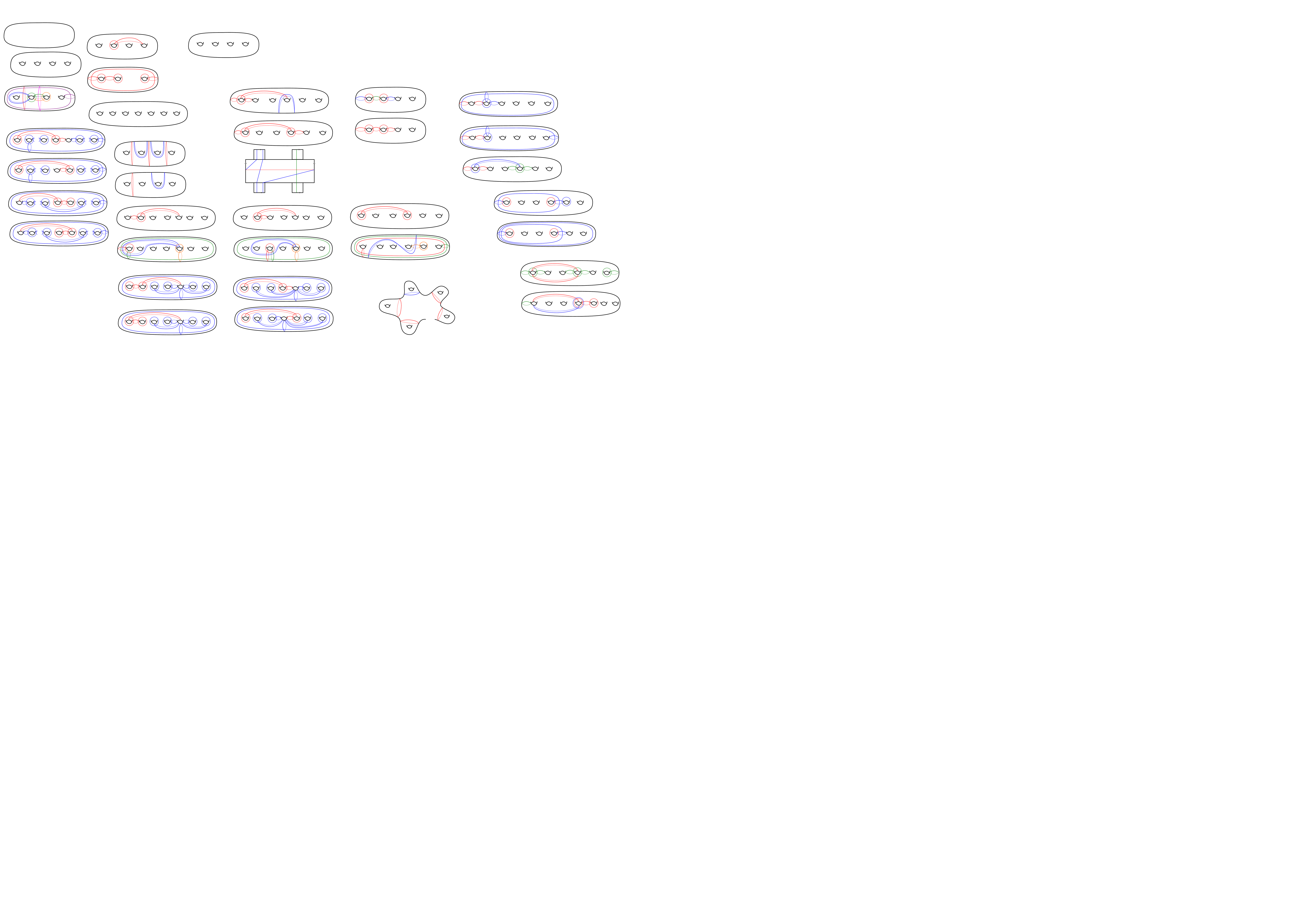}
\caption{}
\label{fig:statement-C}
\end{figure}
In these cases, the genus-2 component of $\Sigma\setminus\alpha$ contains the sharing pair with spine $(c_1,c_2,c_3)$, which is covered by (B). Therefore $(\dagger)$ holds. Figure \ref{fig:statement-C} also shows a choice of $g-2$ disjoint genus-1 curves $x_1,\ldots,x_{g-2}\in X^s$ in the genus-$(g-2)$ component of $\Sigma\setminus\alpha$. Now $(\ddagger)$ is easily verified. 

\boxed{\text{Proof of (D)}} We proceed just as in (B). For case (ii) it suffices to treat $(c_2,c_3,b_{[4,2k]}^+)$. Here we choose $x=c_{\{0,1\}}$, $y=u_{2k+1,[0,2k]}^-$, $z=v_{\{1,2,3\}*[4,2k]}^+$, and $w=c_{\{2g+1,0\}}$. See Figure \ref{fig:statement-D} (top right). Observe that the curves $\alpha,\beta\in X^s$ and $z\in V$ are of the form covered by (C). We also know by the proof of (C) that $\phi(\alpha),\phi(\beta)$ are in the genus-2 subsurface bounded by $\phi(z)$ because the sharing pair $(\phi(\alpha)=\phi(s_{\{2,3\}}), \phi(s_{\{1,2\}}))$ is in the genus-2 subsurface bounded by $\phi(z)$. Case (iii) is similar: it suffices to treat $(c_4, c_3, b_{[4,2k]}^+)$, and we choose $w,x,y,z$ in Figure \ref{fig:statement-D} (bottom right).

\begin{figure}[h!]
\labellist
\pinlabel $w$ at 770 2130
\pinlabel $x$ at 730 2170
\pinlabel $y$ at 920 2150
\pinlabel $z$ at 820 2150
\pinlabel $w$ at 865 2035
\pinlabel $x$ at 840 2035
\pinlabel $y$ at 930 2033
\pinlabel $z$ at 790 2093
\pinlabel $...$ at 490 2168
\pinlabel $...$ at 605 2168
\endlabellist
\centering
\includegraphics[scale=.6]{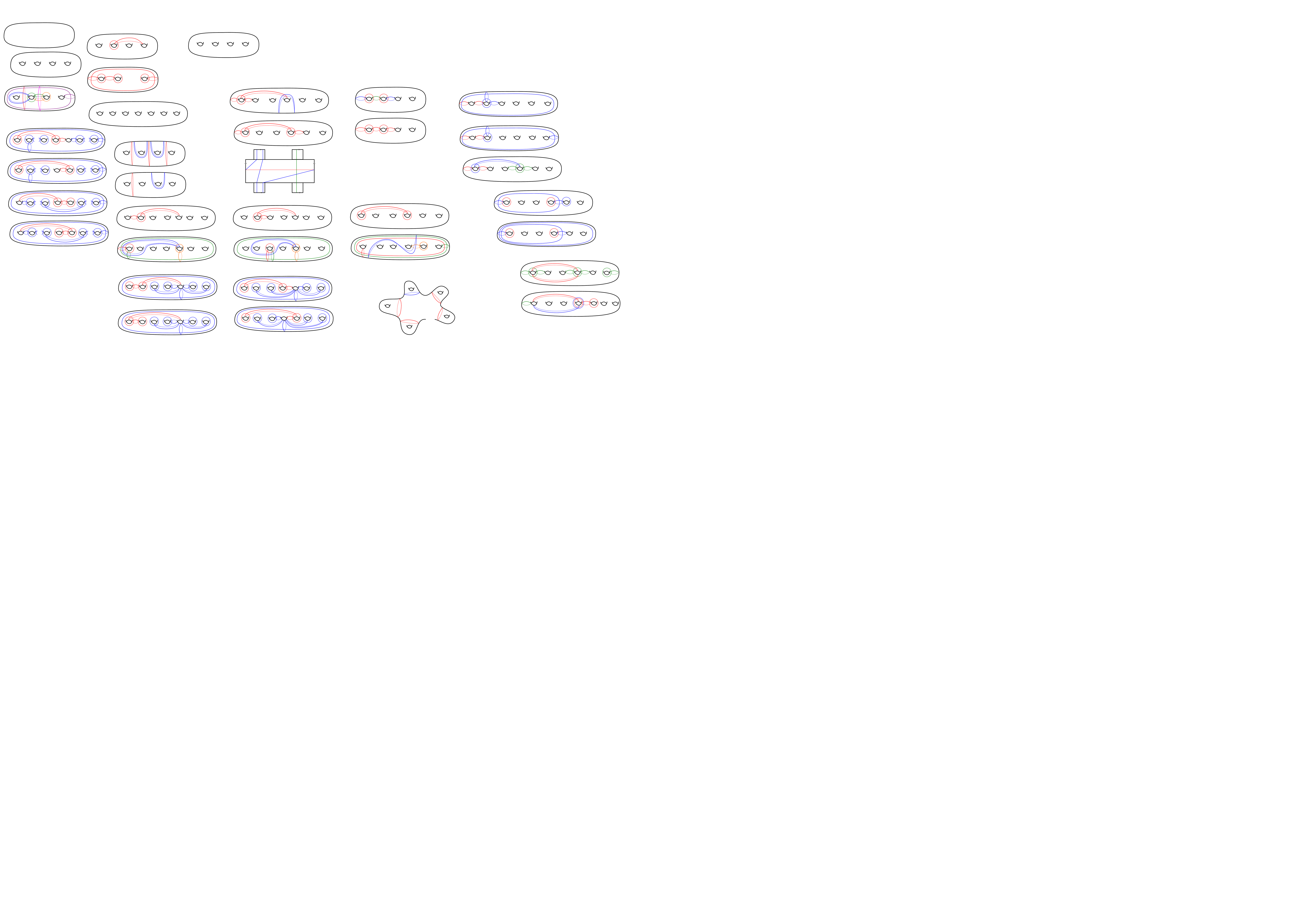}\hspace{.5in}
\includegraphics[scale=.6]{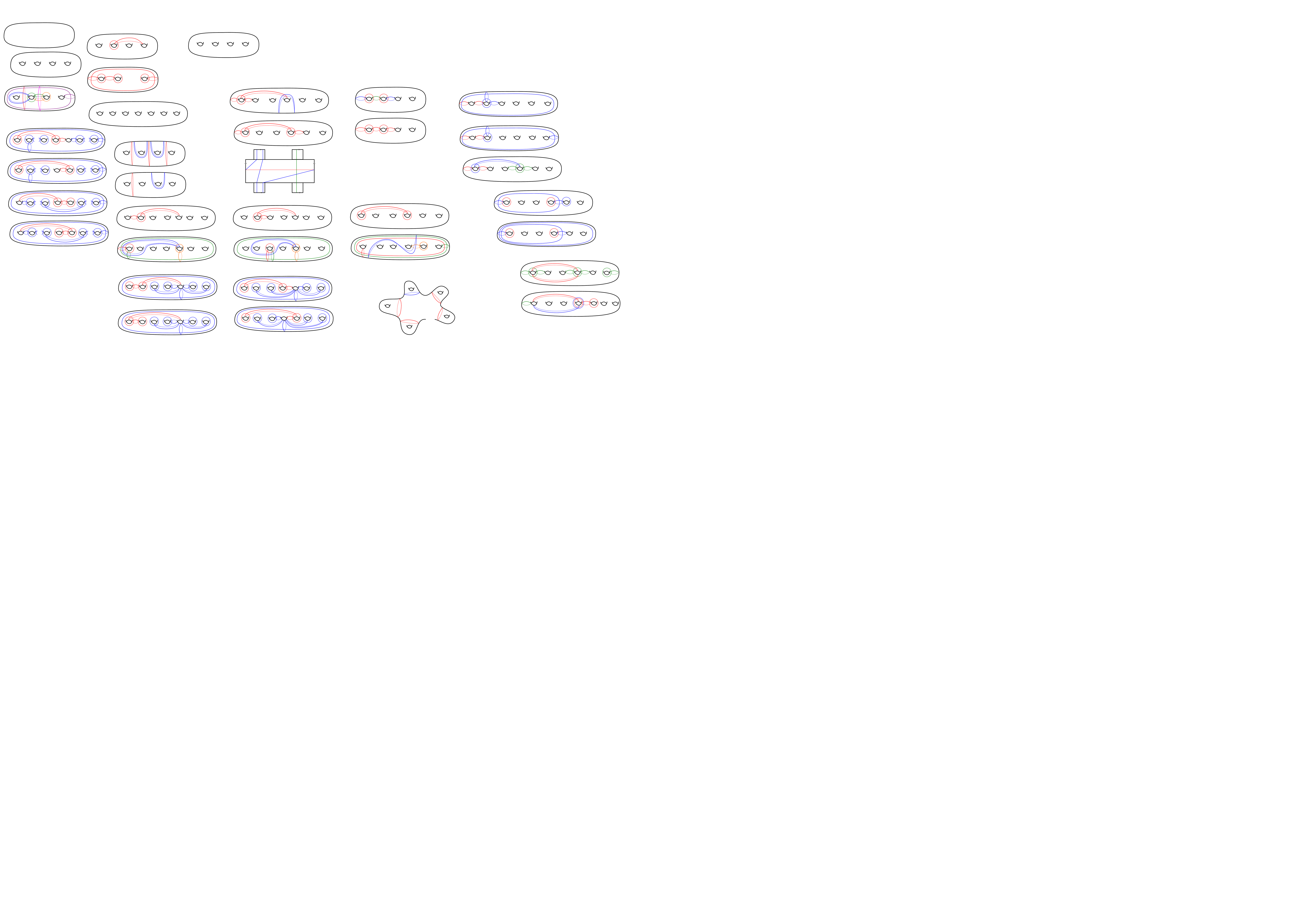}
\includegraphics[scale=.6]{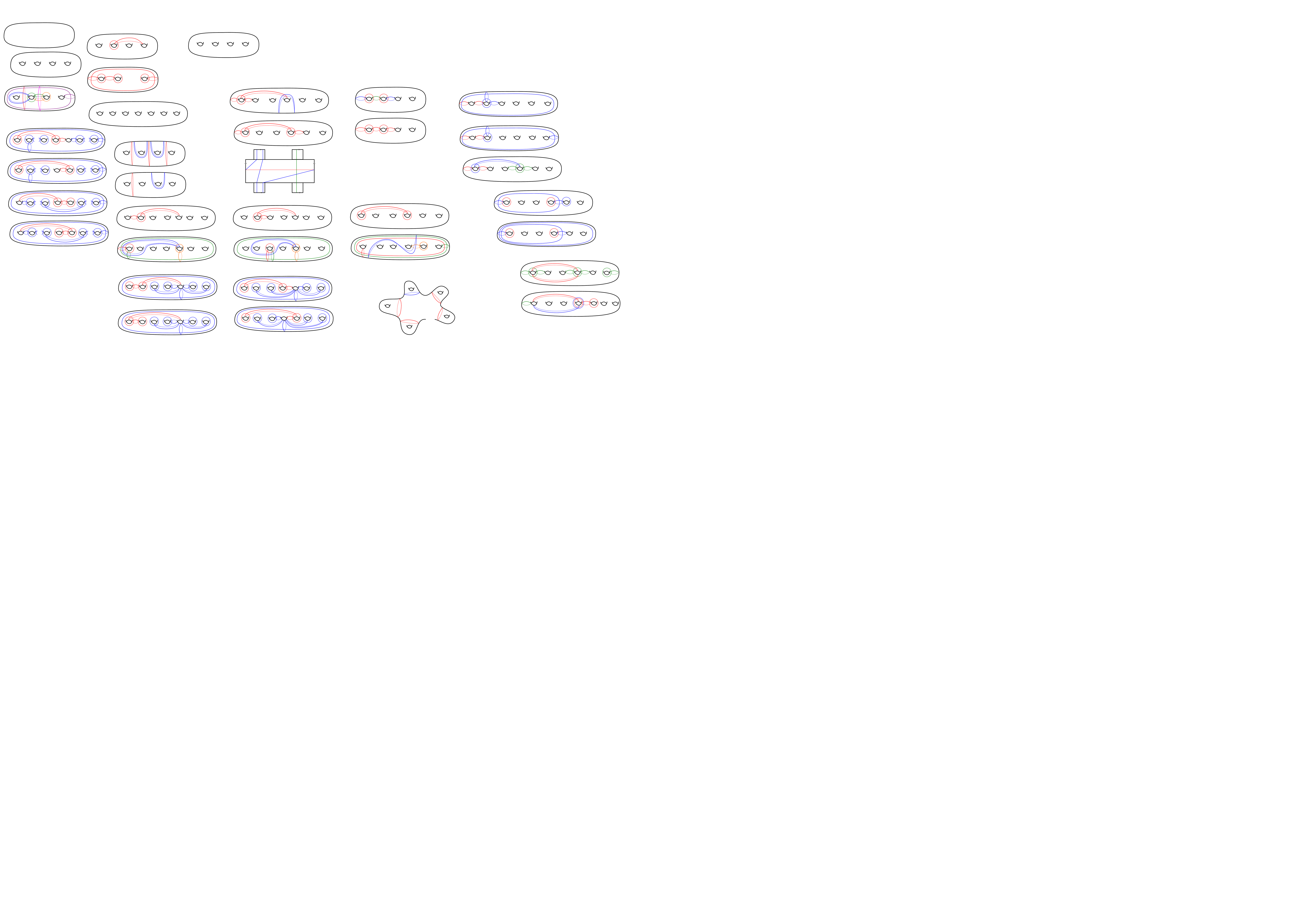}\hspace{.5in}
\includegraphics[scale=.6]{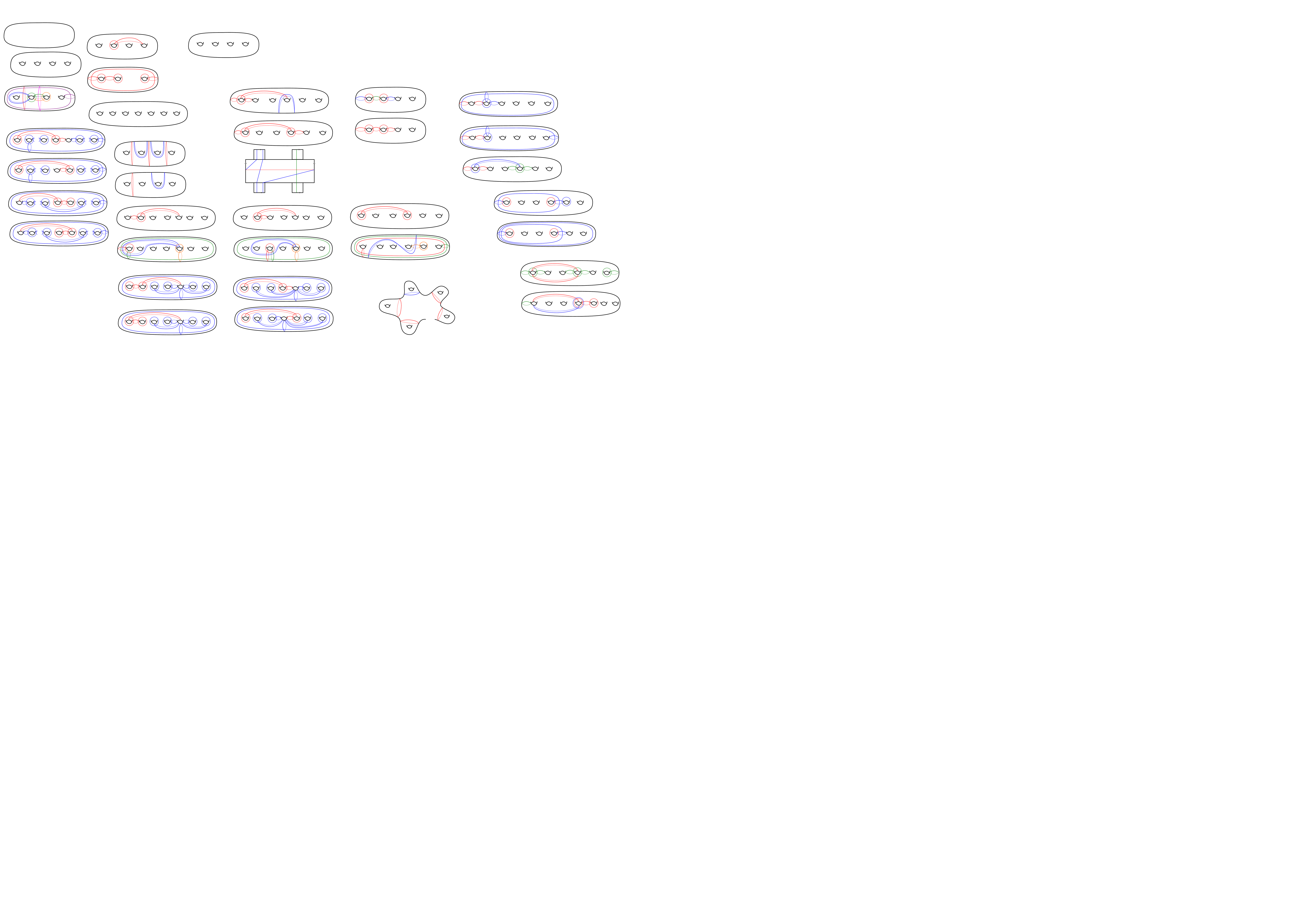}
\caption{Top row: sharing pair of Theorem \ref{thm:sharing-pair}(ii). Bottom row: sharing pair of Theorem \ref{thm:sharing-pair}(iii).}
\label{fig:statement-D}
\end{figure}
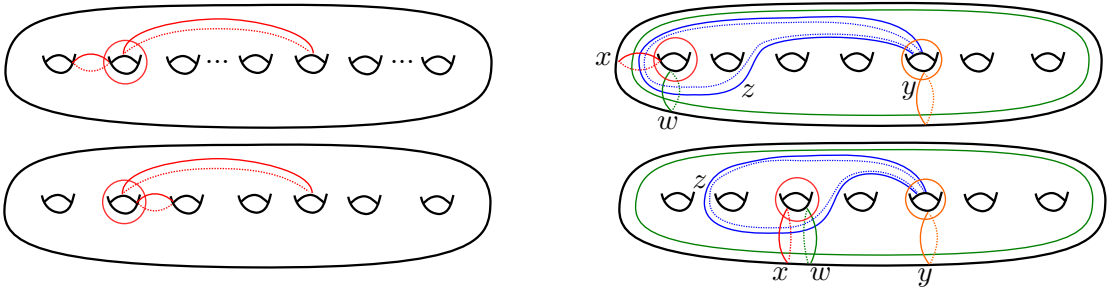



Finally, observe that cases (iv), (v) are equivalent to (ii) and (iii), up to automorphisms of $X$. 

\boxed{\text{Statement (E)}} As in (C), we proceed by proving the statements $(\dagger)$ and $(\ddagger)$. 

Fix $\alpha=v_{\pi*J*\sigma}^\pm$ with both $\pi$ and $\sigma$ nonempty. After a permutation, it suffices to consider the case $\pi=\{1\}$, $J=[2,2k]$, and either $\sigma=\{2k+1,2k+2\}$ or $\sigma=\{2k,2k+1\}$. See Figure \ref{fig:statement-E}.

\begin{figure}[h!]
\labellist
\pinlabel $...$ at 810 1960
\pinlabel $...$ at 962 1960
\endlabellist
\centering
\includegraphics[scale=.6]{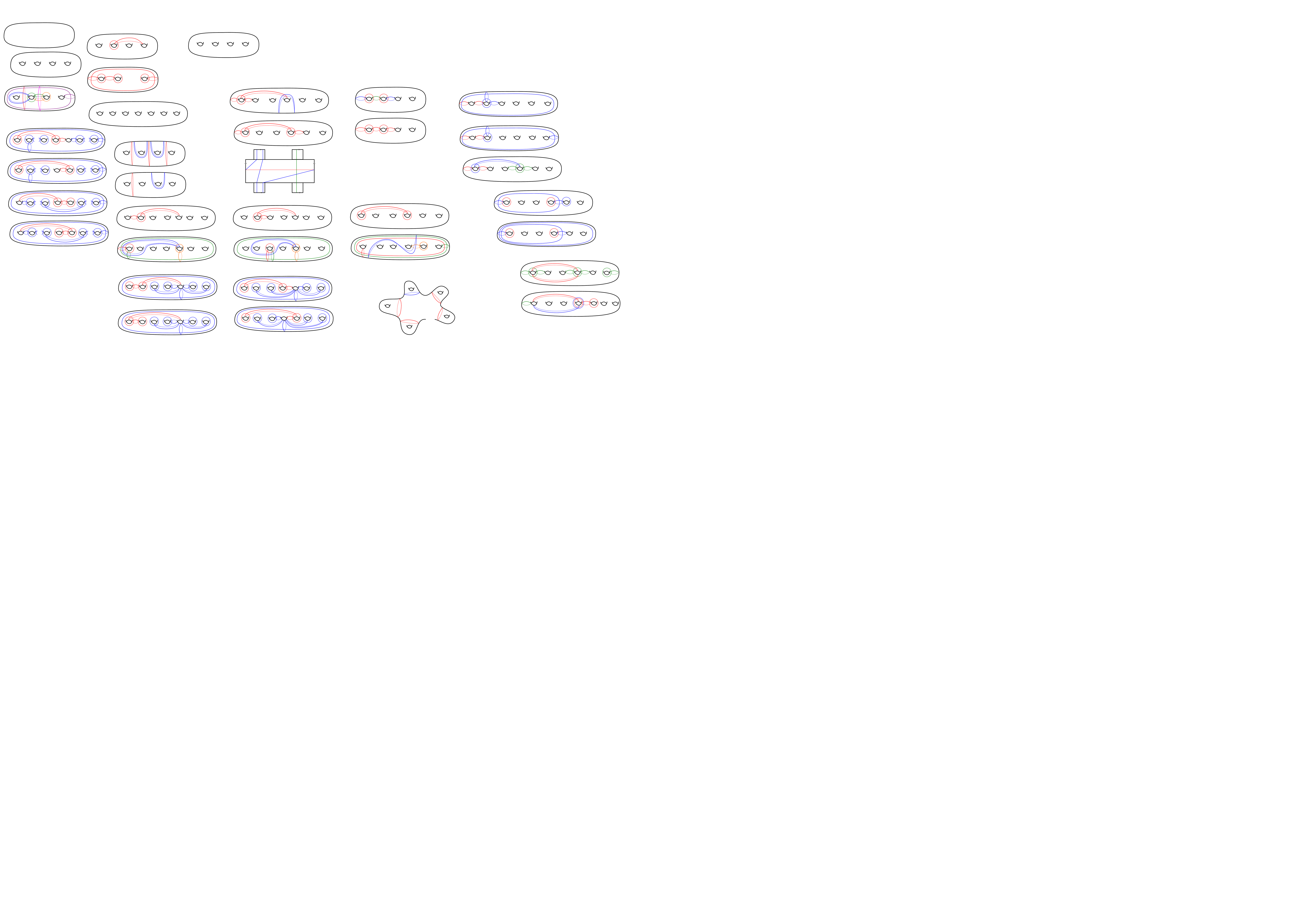}\hspace{.5in}
\includegraphics[scale=.6]{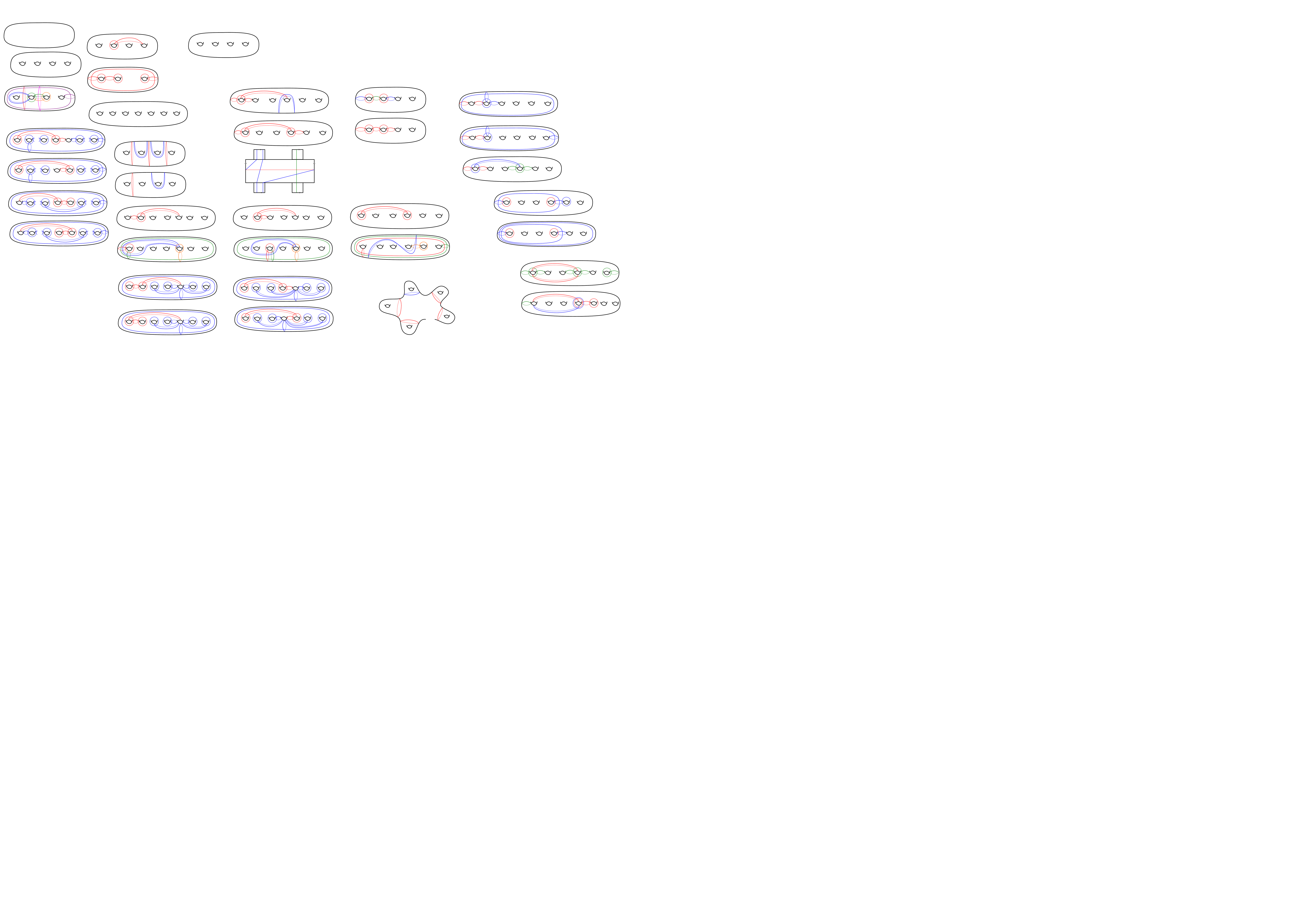}
\caption{}
\label{fig:statement-E}
\end{figure}

In these cases, the genus-2 component of $\Sigma\setminus\alpha$ contains the sharing pair with spine either $(b_J,c_{2k+1},c_{2k+2})$ or $(b_J,c_{2k+1},c_{2k})$, both of which are covered by (D). Therefore $(\dagger)$ holds. Figure \ref{fig:statement-E} also shows a choice of $g-2$ disjoint genus-1 curves $x_1,\ldots,x_{g-2}\in X^s$ in the genus-$(g-2)$ component of $\Sigma\setminus\alpha$. Now $(\ddagger)$ is easily verified.

\boxed{\text{Statement (F)}} We proceed just as in (B) and (D). For case (vi) it suffices to treat $(c_1,b_{[2,2k]}^+,c_{2k+1})$. Here we choose $x=c_{\{0,1\}}$, $y=s_{\{2k+2,2k+3\}}$, $z=v_{\{0,1\}*[2,2k]*\{2k+1\}}^+$, and $w=c_{\{2g,2g+1\}}$. See Figure \ref{fig:statement-F}. Observe that the curves $\alpha,\beta\in X^s$ and $z\in V$ are of the form covered by (E). Using arguments similar to the proof of (D) we can deduce that the curves $\phi(\alpha)$ and $\phi(\beta)$ are in the genus-2 subsurface bounded by $\phi(z)$.

\begin{figure}[h!]
\labellist
\pinlabel $w$ at 1766 2170
\pinlabel $x$ at 1490 2130
\pinlabel $y$ at 1695 2170
\pinlabel $z$ at 1510 2170
\pinlabel $...$ at 1180 2170
\pinlabel $...$ at 1320 2170
\pinlabel $...$ at 1562 2170
\endlabellist
\centering
\includegraphics[scale=.6]{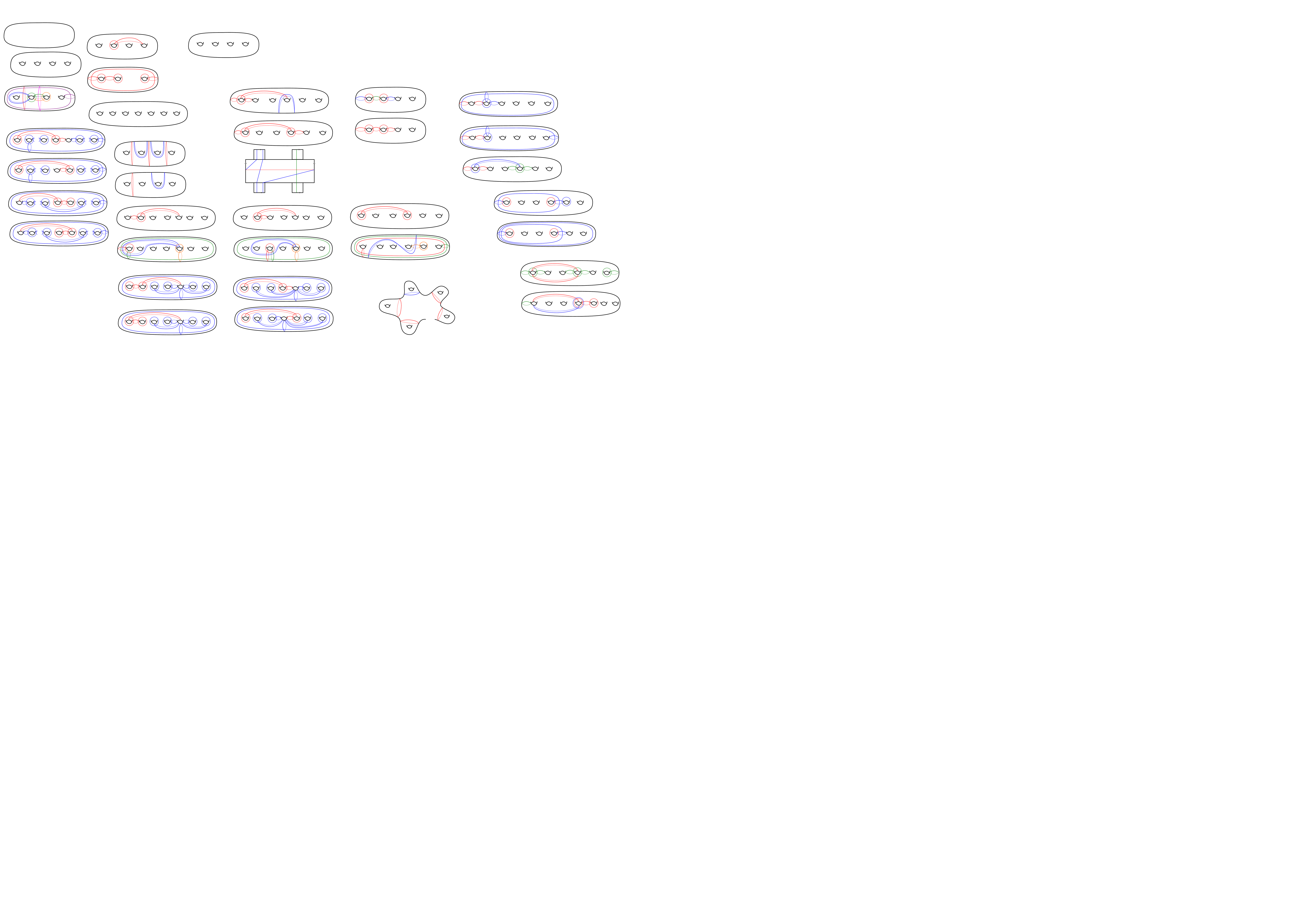}\hspace{.5in}
\includegraphics[scale=.6]{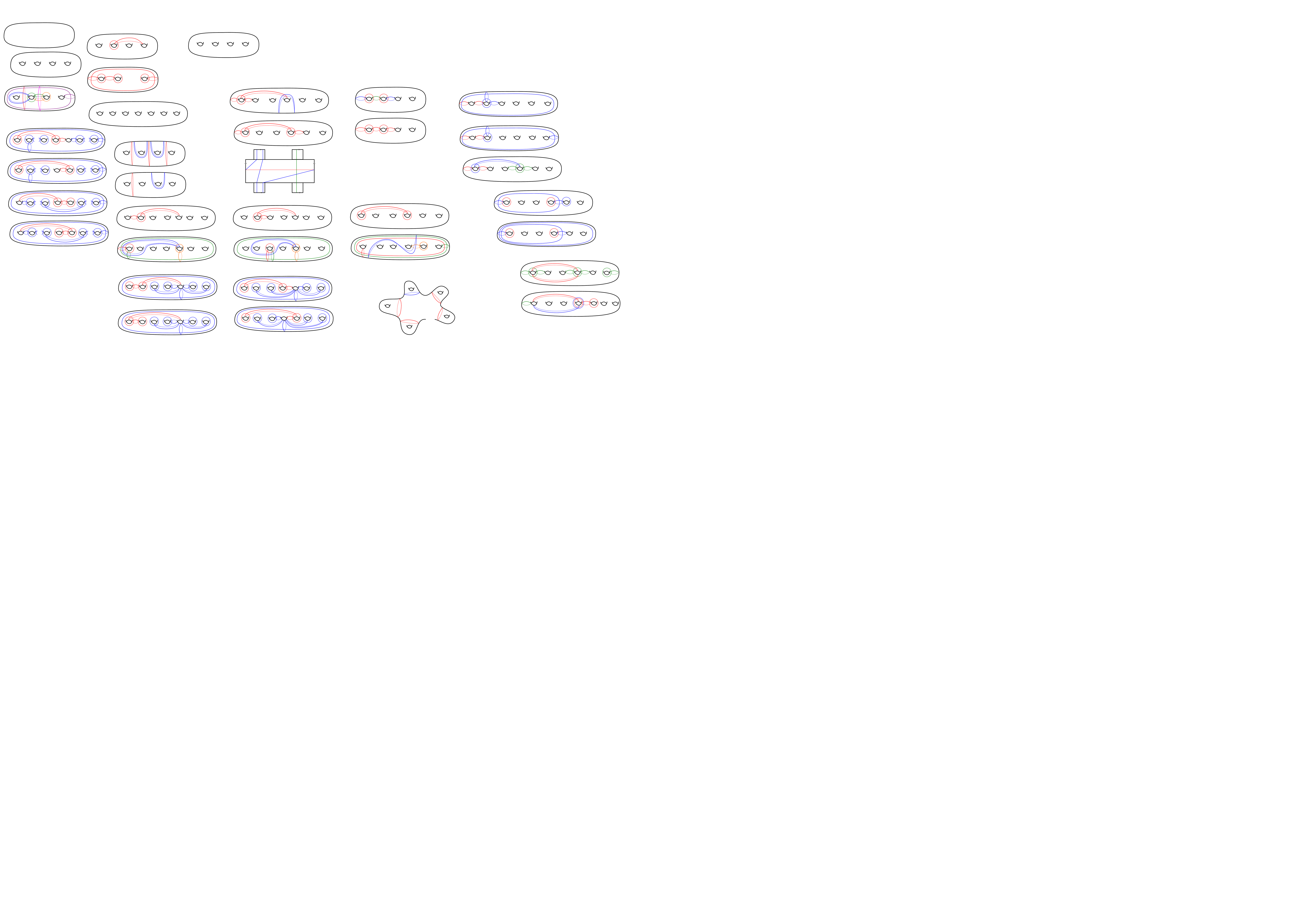}
\caption{}
\label{fig:statement-F}
\end{figure}

\section{Extending the incidence-preserving map} 

Recall that we have a defined subcomplexes $Y^s\subset Y\subset\mathcal C(\Sigma)$ in (\ref{eqn:rigid-set}) and (\ref{eqn:rigid-set-sep}). In this section we prove the last step in the proof of Theorem \ref{thm:main}: 

\begin{thm}\label{thm:extend}
Any incidence-preserving map $\phi:Y^s\rightarrow\mathcal C^s(\Sigma)$ admits a locally injective extension $\widetilde\phi:Y\rightarrow\mathcal C(\Sigma)$. 
\end{thm}

\subsection{Defining the extension $\widetilde\phi$}

Fix an incidence-preserving map $\phi:Y^s\rightarrow\mathcal C^s(\Sigma)$. Since $Y=Y^s\cup C\cup B$, we want to define $\widetilde\phi(y)$ for $y\in C\cup B$. 

For $y\in B$, there is a unique sharing pair $\alpha,\beta$ with spine $(x,y,z)$ with $x,z\in C$. By Theorem \ref{thm:sharing-pair}, $\phi(\alpha),\phi(\beta)$ is a sharing pair, so we define $\widetilde\phi(y)$ to be the curve shared by $\phi(\alpha),\phi(\beta)$.

For $y\in C$, choose a sharing pair $\alpha,\beta$ with spine $(x,y,z)$ such that $x,z\in C\cup B$ and at most one of $x,z$ is in $B$. Again $\phi(\alpha),\phi(\beta)$ is also a sharing pair by Theorem \ref{thm:sharing-pair}, and we define $\widetilde\phi(y)$ to be the curve shared by $\phi(\alpha),\phi(\beta)$. 

When $y\in C$, there are multiple choices for the spine $(x,y,z)$ and hence for $\alpha,\beta$. As such, we need to show that $\widetilde\phi$ is well-defined, independent of this choice. For this, we follow the approach from \cite[\S4.2]{BM}. Two spines $(x,y,z)$ and $(x,y,z')$ are said to \emph{differ by a move} if $(z,y,z')$ is also a spine. The proof of the following lemma is exactly the same as \cite[Lem.\ 4.4]{BM}. 

\begin{lem}\label{lem:well-defined}
Let $\phi:Y^s\rightarrow\mathcal C^s(\Sigma)$ be an incidence-preserving map, and fix sharing pairs $(\alpha,\beta)$ and $(\alpha,\beta')$ with spines $(x,y,z)$ and $(x,y,z')$ that differ by a move. Assume that $\alpha,\beta,\beta',x,y,z,z'$ are curves in $Y$. Assume that (i) there exists $\gamma\in Y^s$ such that $\gamma$ intersects $z'$ but intersects none of $x,y,z$, and (ii) each of $(\phi(\alpha),\phi(\beta))$, $(\phi(\alpha),\phi(\beta'))$, and $(\phi(\beta),\phi(\beta'))$ are sharing pairs. Then $(\phi(\alpha),\phi(\beta))$ and $(\phi(\alpha),\phi(\beta'))$ share the same curve. 
\end{lem}

Now we can show $\widetilde\phi$ is well defined on $C$. First note that any two spines used in the definition of $\widetilde\phi(y)$ for $y\in C$ differ by at most two moves (they all differ by a single move from the (unique) spine consisting only of chain curves). It is easy to check that condition (i) in Lemma \ref{lem:well-defined} holds; see e.g.\ Figure \ref{fig:well-defined}. 

\begin{figure}[h!]
\labellist
\pinlabel $x$ at 700 2530
\pinlabel $y$ at 740 2507
\pinlabel $z$ at 765 2518
\pinlabel $z'$ at 870 2555
\pinlabel $\gamma$ at 850 2500
\pinlabel $\cdots$ at 810 2530
\pinlabel $\cdots$ at 955 2530
\endlabellist
\centering
\includegraphics[scale=.8]{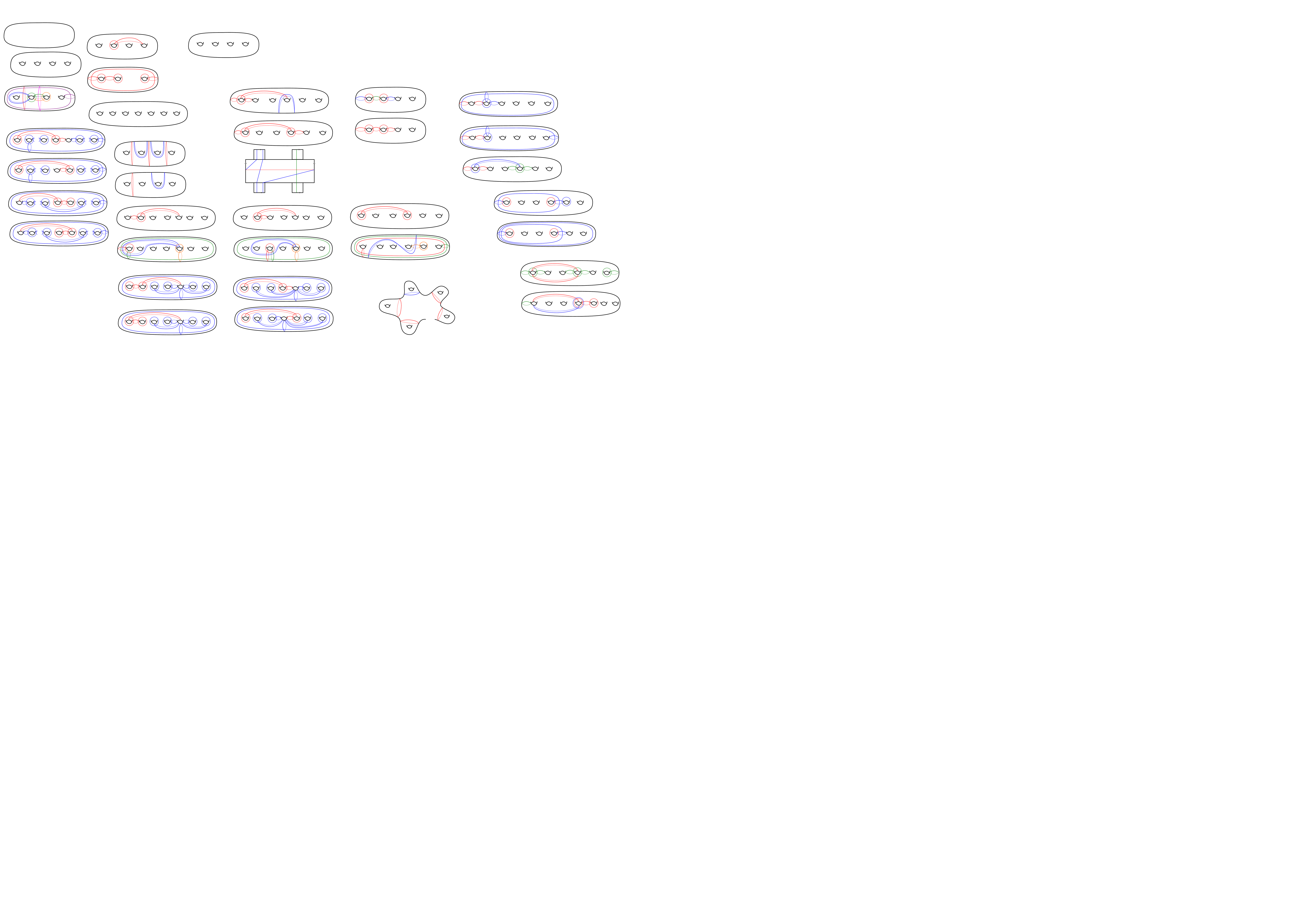}
\caption{Proving that $\widetilde\phi$ is well-defined.}
\label{fig:well-defined}
\end{figure}
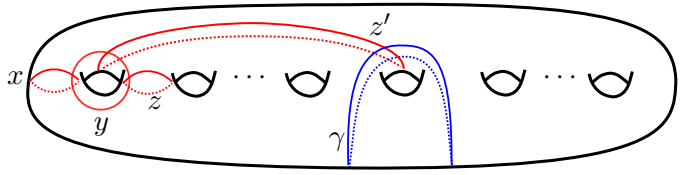


Finally, condition (ii) in Lemma \ref{lem:well-defined} holds by Theorem \ref{thm:sharing-pair}. Therefore, $\widetilde\phi$ is well defined. 

\begin{rmk}\label{rmk:simplicial}
Having defined $\widetilde\phi$ on the vertices of $Y$, to show that $\widetilde\phi$ extends simplicially to the subcomplex spanned by these vertices (which we are also denoting by $Y$), it suffices to show that if $\alpha,\beta\in Y$ are disjoint, then $\widetilde\phi(\alpha),\widetilde\phi(\beta)$ are also disjoint. This is proved in Proposition \ref{prop:distance1} below. 
\end{rmk}

\subsection{The extension $\widetilde\phi$ is locally-injective} 

This section contains the final step in the proof of Theorem \ref{thm:main}. We need to show that $\widetilde\phi$ is simplicial (c.f.\ Remark \ref{rmk:simplicial}) and that $\widetilde\phi$ is injective on the star of each vertex $v\in Y$. These statements follow from Propositions \ref{prop:distance1} and \ref{prop:distance2} below (which in fact show that $\widetilde\phi$ is injective on $Y$). 

\begin{prop}\label{prop:distance1}
If $\alpha\neq \beta$ are vertices in $Y$ and $i(\alpha,\beta)=0$, then $\widetilde\phi(\alpha)\neq\widetilde\phi(\beta)$ and $i(\widetilde\phi(\alpha),\widetilde\phi(\beta))=0$. 
\end{prop}


\begin{prop}\label{prop:distance2} 
If $\alpha,\beta$ are vertices in $Y$ and $i(\alpha,\beta)\neq0$, then $\widetilde\phi(\alpha)\neq\widetilde\phi(\beta)$. 
\end{prop}

The proofs of Propositions \ref{prop:distance1} and \ref{prop:distance2} rely on the following lemma.

\begin{lem}[Chains under $\widetilde\phi$]\label{lem:chains}
Fix $x_1, x_2\in X\setminus X^s\equiv C\cup B$. 
\begin{enumerate}
\item[(i)] Assume that $x_1$ and $x_2$ do not form a bounding pair. If $x_1$ and $x_2$ are distinct and disjoint, then $\widetilde\phi(x_1)$ and $\widetilde\phi(x_2)$ are distinct and disjoint. 
\item[(ii)] Assume that at most one of the $x_1,x_2$ belongs to $B$. Then if $i(x_1,x_2)=1$, then $i(\widetilde\phi(x_1),\widetilde\phi(x_2))=1$. 
\end{enumerate} 
\end{lem}

\begin{proof}
\boxed{(i)} It suffices to observe that there exists $y_1,y_2\in C$ so that $i(x_j,y_k)=\delta_{jk}$ (Kronecker delta). For then one can find $z_1,z_2\in C$ so that $(y_j,x_j,z_j)$ is the spine for a sharing pair $(\alpha_j,\beta_j)$ for $x_j$ (this is illustrated in Figure \ref{fig:lem-chains0}); by construction $\alpha_1=\partial N(x_1\cup y_1)$ and $\alpha_2=\partial N(x_2\cup y_2)$ are distinct and disjoint. Then so too are $\widetilde\phi(\alpha_1)$ and $\widetilde\phi(\alpha_2)$ (by Remark \ref{rmk:injective}), and from this we deduce that $\widetilde\phi(x_1)$ and $\widetilde\phi(x_2)$ are disjoint and distinct because they lie in the disjoint genus-1 subsurfaces bounded by $\widetilde\phi(\alpha_1)$ and $\widetilde\phi(\alpha_2)$ (by Theorem \ref{thm:sharing-pair}).

\begin{figure}[h!]
\labellist
\pinlabel $x_1$ at 1078 2535
\pinlabel $y_1$ at 1130 2555
\pinlabel $z_1=z_2$ at 1155 2507
\pinlabel $y_2$ at 1180 2555
\pinlabel $x_2$ at 1205 2520
\pinlabel $\uparrow$ at 1155 2522
\endlabellist
\centering
\includegraphics[scale=.7]{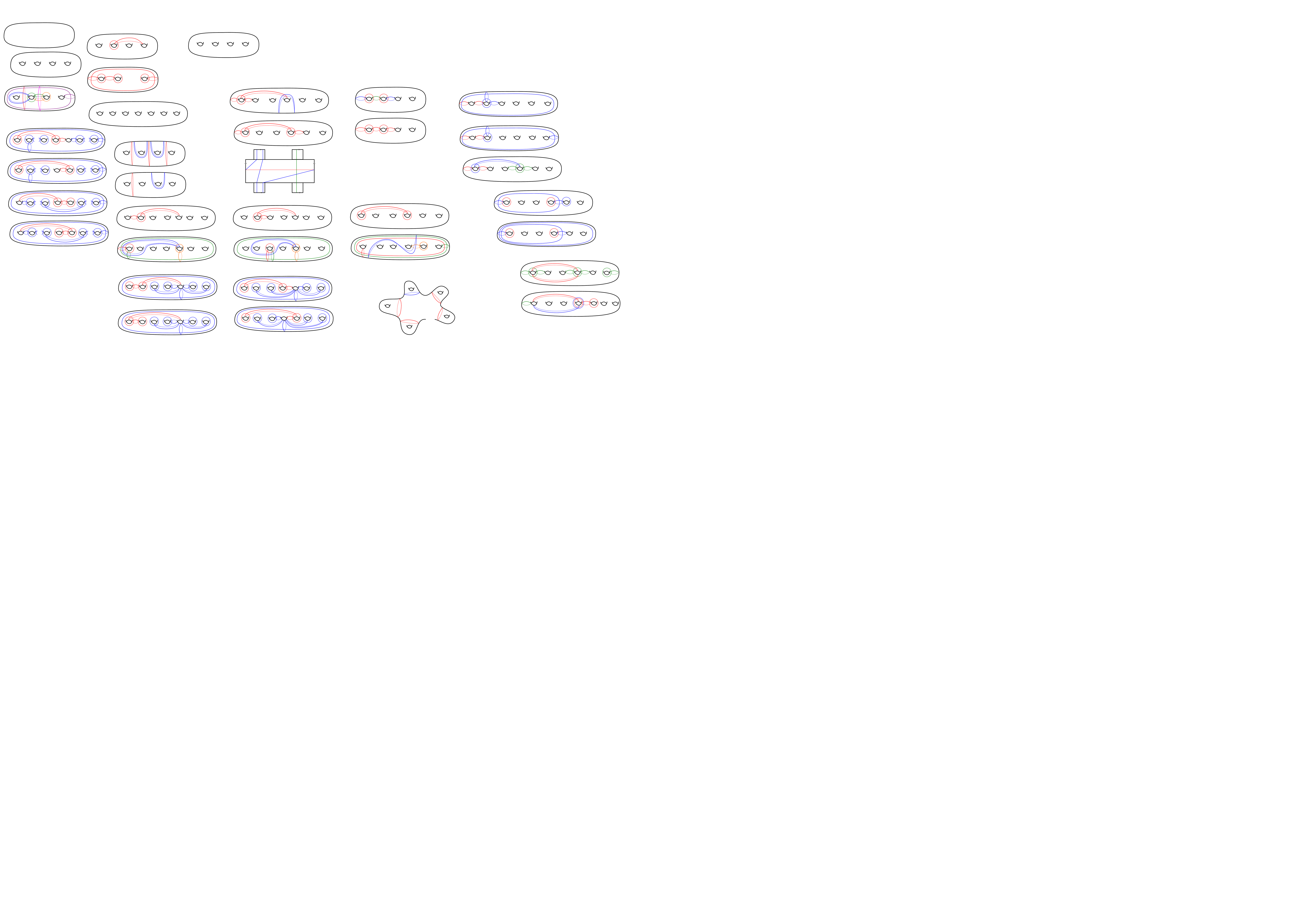}
\caption{}
\label{fig:lem-chains0}
\end{figure}


\boxed{(ii)} 
There are two cases to consider: either $x_1,x_2\in C$ or $x_1\in C$ and $x_2\in B$. 

First assume that $x_1,x_2\in C$. Up to a permutation of $X$, we can assume that $x_1=c_1$ and $x_2=c_2$. Consider the chain $(x_0,x_1,x_2,x_3,x_4)$ in Figure \ref{fig:lem-chains} (left).

\begin{figure}[h!]
\labellist
\pinlabel $x_0$ at 1078 2440
\pinlabel $x_1$ at 1130 2458
\pinlabel $x_2$ at 1155 2425
\pinlabel $x_3$ at 1180 2458
\pinlabel $x_4$ at 1205 2425
\pinlabel $x_0$ at 1340 2440
\pinlabel $x_1$ at 1388 2417
\pinlabel $x_2$ at 1460 2453
\pinlabel $x_3$ at 1518 2420
\pinlabel $x_4$ at 1555 2427
\pinlabel $\cdots$ at 1458 2435
\endlabellist
\centering
\includegraphics[scale=.7]{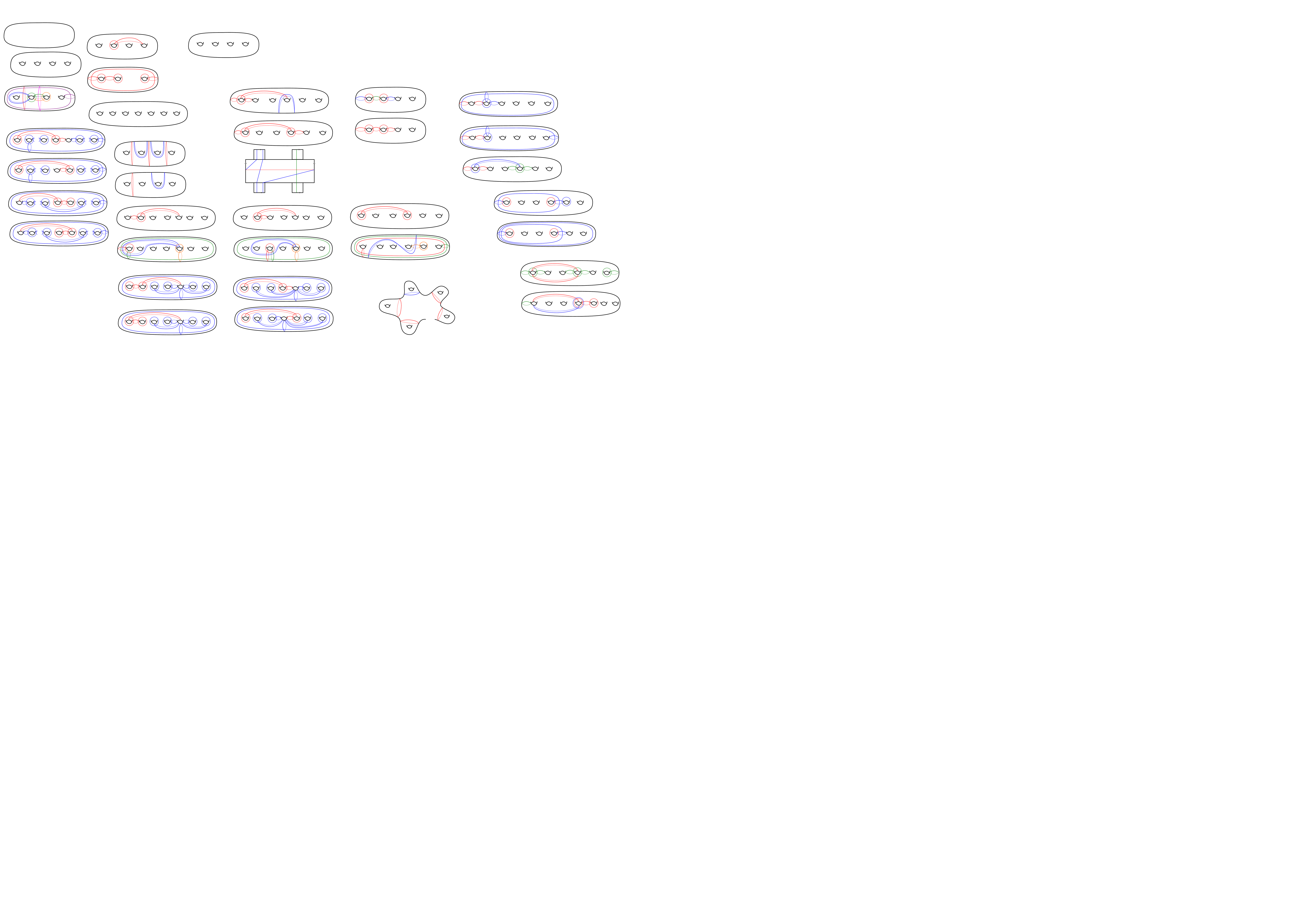}\hspace{.3in}
\includegraphics[scale=.7]{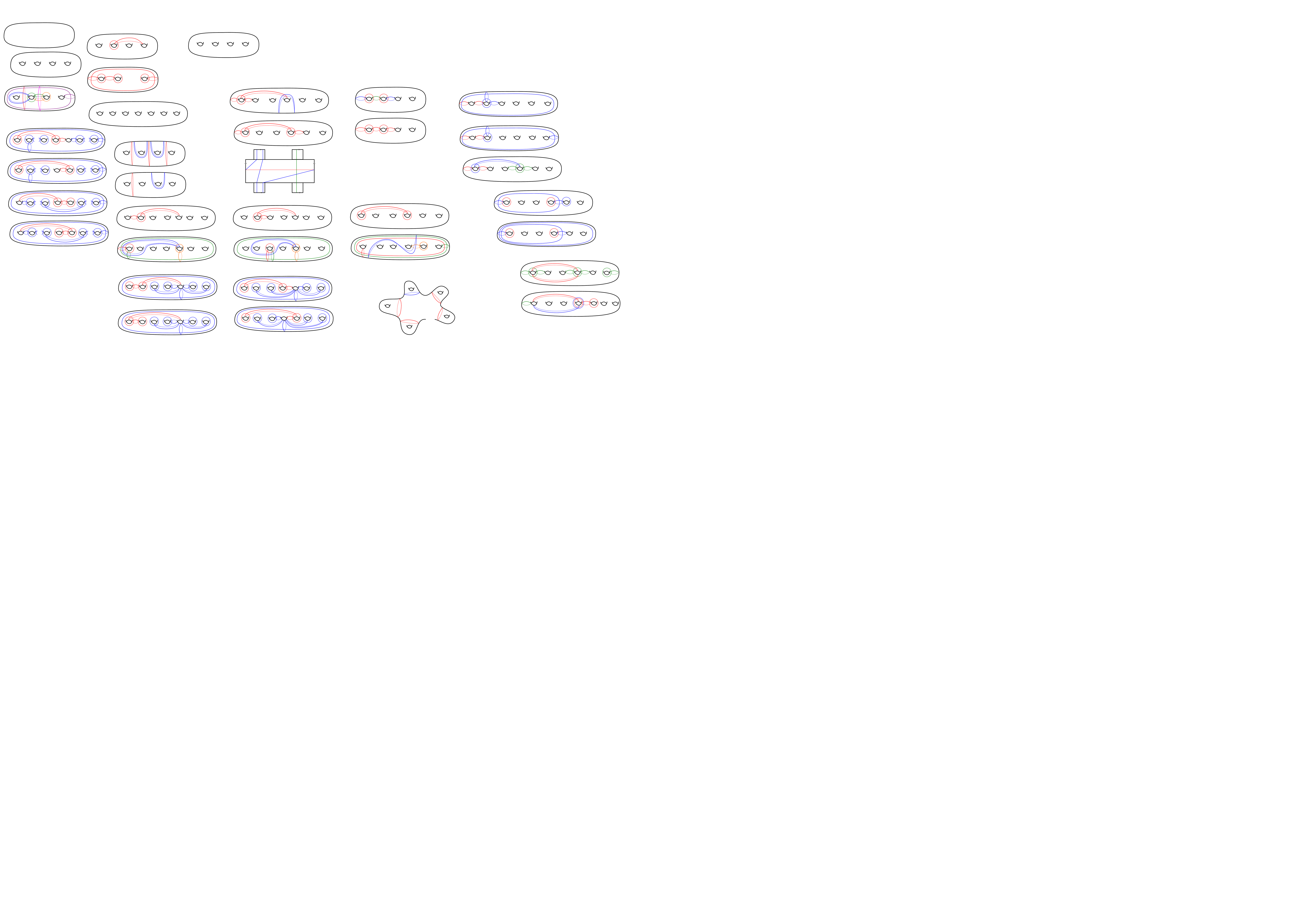}
\caption{}
\label{fig:lem-chains}
\end{figure}


Each triple $(x_{j-1},x_{j},x_{j+1})$ is a spine of a sharing pair $\alpha_j,\beta_j$ for $x_j$ (note with this notation that $\alpha_{j+1}=\beta_j$). By Theorem \ref{thm:sharing-pair} and the definition of $\widetilde\phi$, $\phi(\alpha_j),\phi(\beta_j)$ is a sharing pair for $\widetilde\phi(x_j)$. Denoting $\Sigma_j$ the genus-1 component of $\Sigma\setminus\widetilde\phi(\beta_j)$ for $j=1,2$, we deduce that $\widetilde\phi(x_j)$ and $\widetilde\phi(x_{j+1})$ are contained in $\Sigma_j$ (for example, $\widetilde\phi(\beta_1=\alpha_2)$ forms a sharing pair with both $\widetilde\phi(\alpha_1)$ and $\widetilde\phi(\beta_2)$, which implies that $\Sigma_1$ contains both $\widetilde\phi(x_1)$ and $\widetilde\phi(x_2)$). 

Since $\widetilde\phi(x_1)$ and $\widetilde\phi(x_3)$ are distinct and disjoint by part (i) of the lemma, it follows that $\widetilde\phi(x_1)\neq\widetilde\phi(x_2)$ and $\widetilde\phi(x_3)\neq\widetilde\phi(x_2)$. This implies that $\widetilde\phi(x_1),\widetilde\phi(x_3)$ both intersect $\widetilde\phi(x_2)$ (distinct curves on a torus with one boundary component must intersect). 

Since $\phi(\alpha_2=\beta_1),\phi(\beta_2)$ is a sharing pair for $\widetilde\phi(x_2)$, the subsurface $\Sigma_1\cup\Sigma_2$ is obtained from an annular neighborhood $A$ of $\widetilde\phi(x_2)$ by attaching two rectangles as pictured in Figure \ref{fig:lem-chains2}. 

\begin{figure}[h!]
\labellist
\pinlabel $\widetilde\phi(x_1)$ at 818 2290
\pinlabel $\widetilde\phi(x_2)$ at 855 2325
\pinlabel $\widetilde\phi(x_3)$ at 930 2288
\pinlabel $A$ at 745 2284
\endlabellist
\centering
\includegraphics[scale=.7]{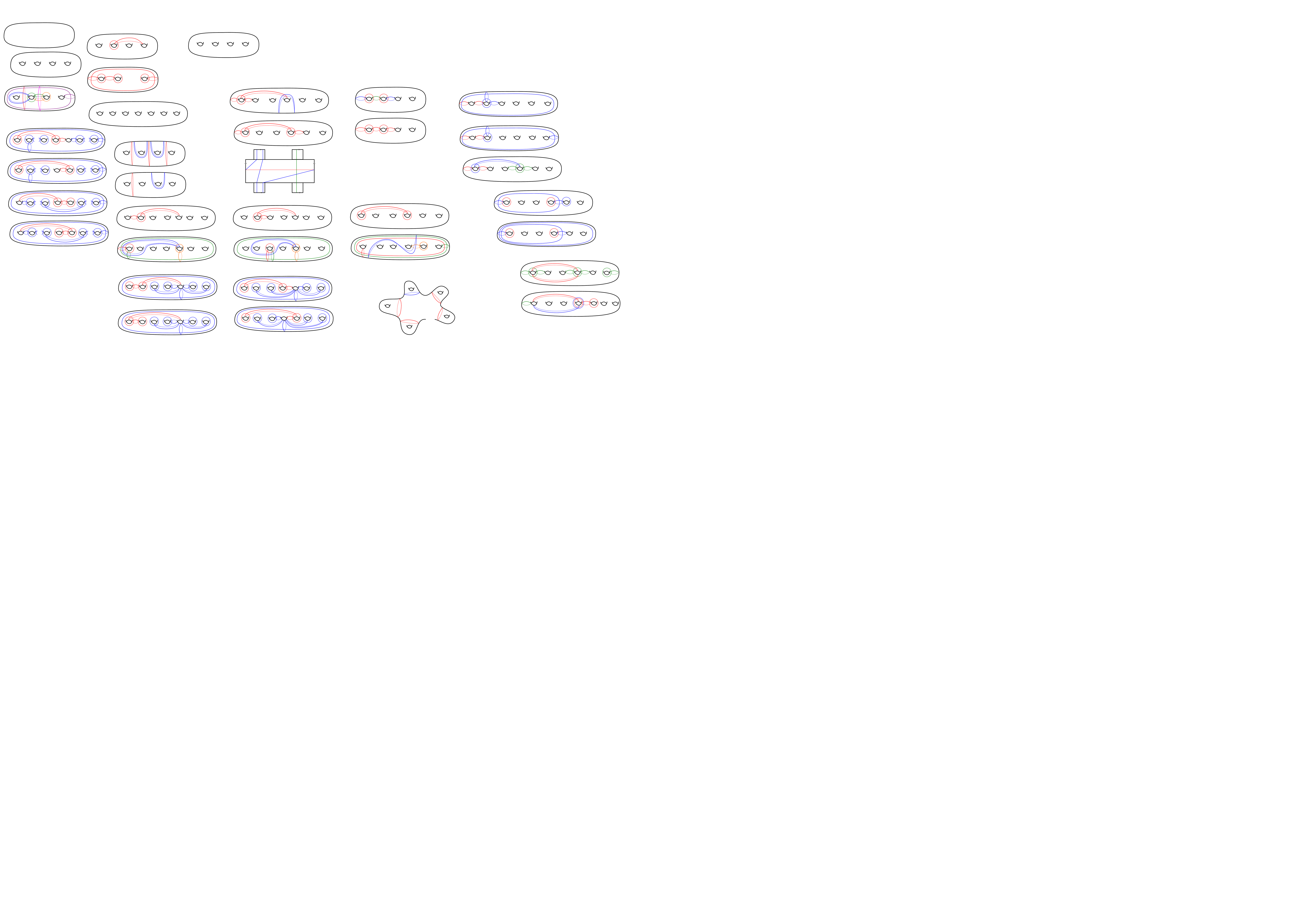}
\caption{The subsurface $\Sigma_1\cup\Sigma_2$ is homeomorphic to an annulus $A$ with two rectangles attached along the boundary.}
\label{fig:lem-chains2}
\end{figure}
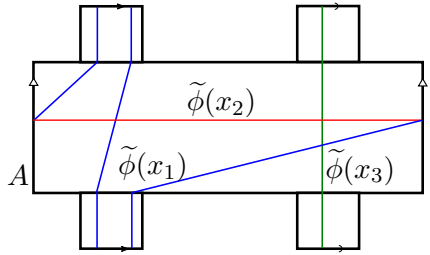


Observe that if $i(\widetilde\phi(x_1),\widetilde\phi(x_2))>1$, then $\widetilde\phi(x_1)$ and $\widetilde\phi(x_3)$ are not disjoint, which is a contradiction (this is because the condition $i(\widetilde\phi(x_1),\widetilde\phi(x_2))>1$ implies that $\widetilde\phi(x_1)\cap A$ separates the two intervals in $\partial A$ where the rectangle for $\Sigma_2$ is attached). From this we conclude that $i(\widetilde\phi(x_1),\widetilde\phi(x_2))=1$, as desired. 

The proof in the case $x\in C$ and $y\in B$ is entirely similar, except we work with the chain $(x_0,\ldots,x_4)$ pictured in Figure \ref{fig:lem-chains} (right). 
\end{proof}

\begin{proof}[Proof of Proposition \ref{prop:distance1}]
We consider cases depending on whether or not $\alpha,\beta$ are separating, i.e.\ whether or not $\alpha,\beta\in Y^s$. 

\paragraph{Case 1:} $\alpha\in Y^s$ and $\beta\in Y^s$. We already know that $\widetilde\phi(\alpha)$ and $\widetilde\phi(\beta)$ are disjoint because $\rest{\widetilde\phi}{Y^s}=\phi$, and $\phi$ is incidence preserving. Furthermore, $\widetilde\phi(\alpha)\neq\widetilde\phi(\beta)$ by Remark \ref{rmk:injective}. 

\paragraph{Case 2:} $\alpha\in Y^s$ and $\beta\notin Y^s$. 

First observe that it suffices to find a sharing pair $\gamma,\delta$ for $\alpha$ such that $(\gamma,\delta)$ is on the list in Theorem \ref{thm:sharing-pair} and $\gamma$ is disjoint from $\beta$. For then $\phi(\gamma),\phi(\delta)$ is a sharing pair for $\widetilde\phi(\alpha)$ (by Theorem \ref{thm:sharing-pair} and by definition of $\widetilde\phi$), which implies in particular that $\phi(\gamma)$ is a genus-1 curve with $\widetilde\phi(\alpha)$ contained in the genus-1 component of $\Sigma\setminus\phi(\gamma)$, and since $\widetilde\phi(\beta)=\phi(\beta)$ is a separating curve disjoint from $\phi(\gamma)$, this implies that $\phi(\beta)$ lies in the genus-$(g-1)$ component of $\Sigma\setminus\phi(\gamma)$. Therefore $\widetilde\phi(\alpha)$ and $\widetilde\phi(\beta)$  are both disjoint and distinct since they lie in different components of $\Sigma\setminus\phi(\gamma)$.  

\paragraph{Case 2(i):} Assume that $\alpha\in C$. Without loss of generality $\alpha=c_1$. Consider $(c_0,c_1,c_2)$, which is the spine of a sharing pair for $\alpha=c_1$. If $\beta$ is disjoint from either $c_0$ or $c_1$, then we are done. 

{\it Claim.} If $\beta$ intersects both $c_0$ and $c_2$, then $\beta$ has one of the following forms (see also Figure \ref{fig:prop-disjoint})
\[\partial N(c_3\cup c_4\cup b_{[0,2]}^\pm\cup c_{2g+1})\>\>\>\text{ or }\>\>\> \partial N(c_3\cup b_{[0,2]}^\pm\cup c_{2g+1}\cup c_{2g})\>\>\>\text{ or }\>\>\> \partial N(c_1\cup b_{[2,2k]}^{\pm})\]
\[\text{ or }\>\>\>\partial N(c_1\cup b_{[2,2k]}^{\pm}\cup c_{2k+1}\cup c_{2k})\>\>\>\text{ or }\>\>\>\partial N(c_1\cup b_{[2,2k]}^{\pm}\cup c_{2k+1}\cup c_{2k+2}).\]
The details of this claim are not hard to check: since $\beta\in S\cup U\cup V$, we can write $\beta=\partial N(x_1\cup\cdots\cup x_\ell)$, where $\ell$ is even, the $x_j$ belong to $C\cup B$, and at most one $x_j$ is in $B$; in order for $\beta$ to intersect $c_0$ and $c_2$, at least one of the $x_j$ has to intersect $c_0$ and similarly for $c_2$, and none of the $x_j$ can be equal to $c_0$ or $c_2$; examining the possibilities yields the claim.  

\begin{figure}[h!]
\labellist
\pinlabel $c_0$ at 1395 2520
\pinlabel $c_2$ at 1470 2530
\pinlabel $\beta$ at 1505 2500
\pinlabel $...$ at 1605 2520
\pinlabel $...$ at 1675 2430
\endlabellist
\centering
\includegraphics[scale=.6]{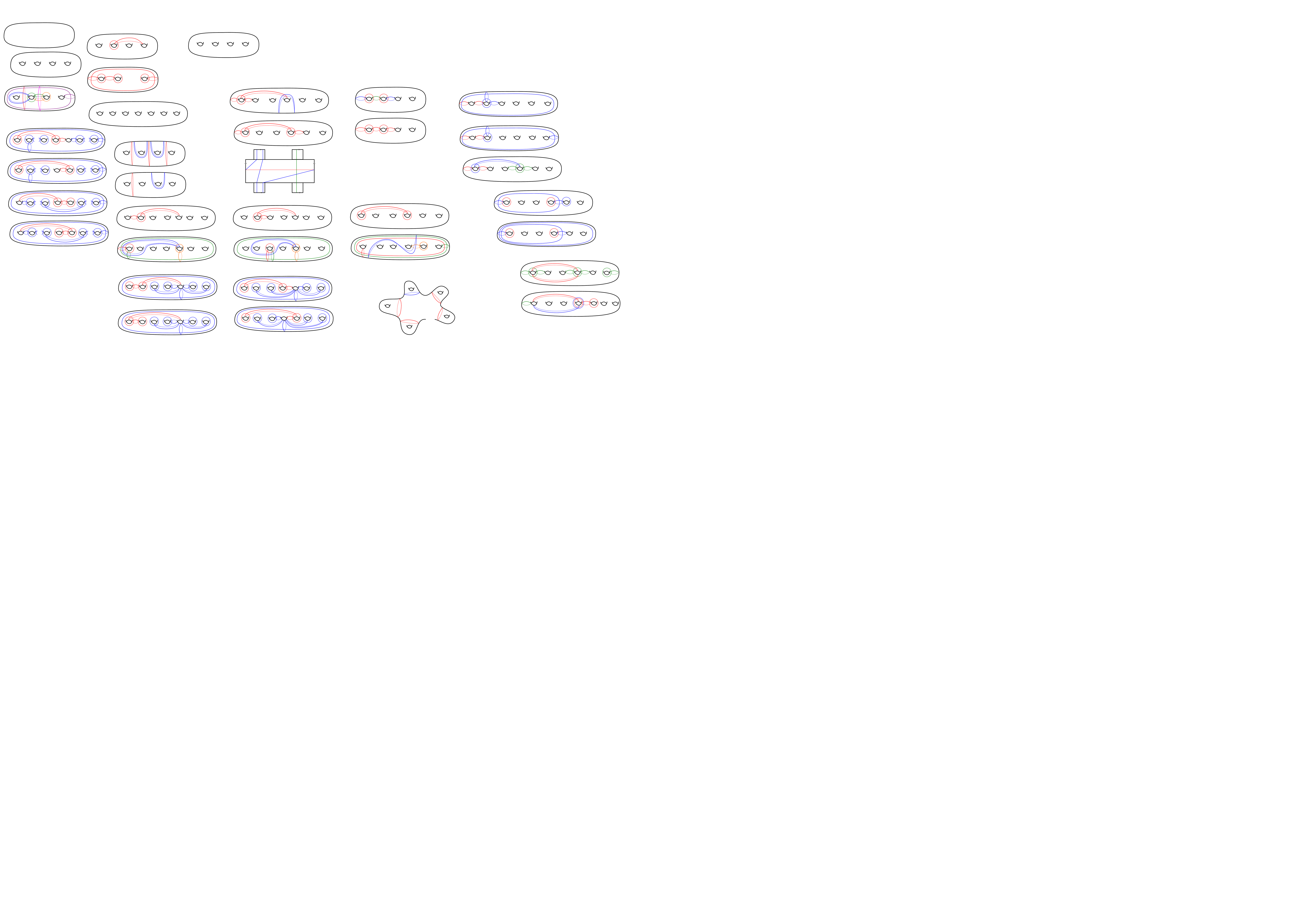}
\includegraphics[scale=.6]{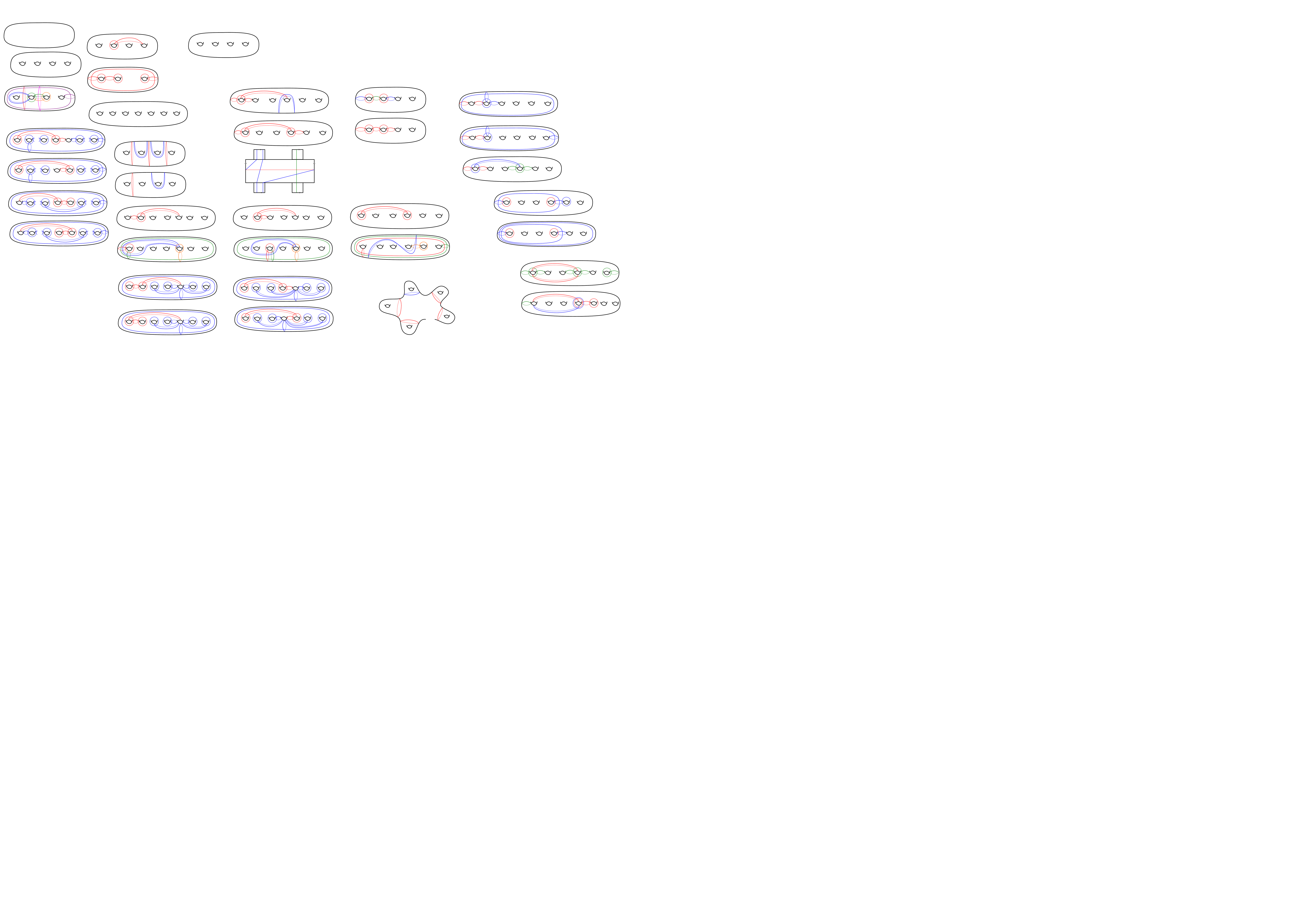}
\includegraphics[scale=.6]{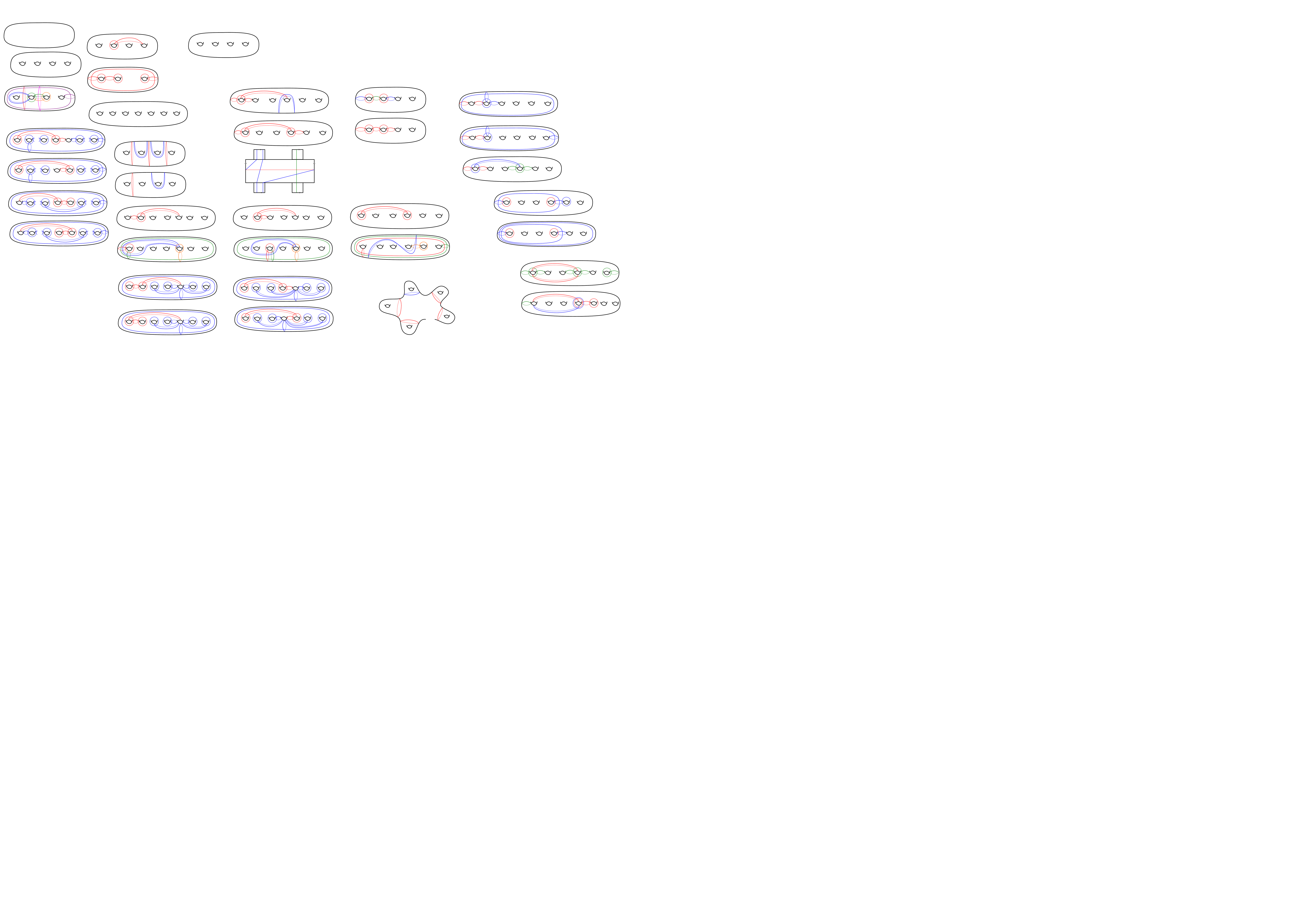}
\caption{Curves $\beta$ in $Y^s$ that intersect both $c_0$ and $c_2$. In the bottom figure, one can either ignore the two green curves or include two adjacent green curves in defining $\beta$.}
\label{fig:prop-disjoint}
\end{figure}


Observe that in every case there is a different spine $(c_0,c_1,\star)$ for a sharing pair for $\alpha=c_1$ with $\beta$ disjoint from $\star$, so again we are done. In the first two examples in Figure \ref{fig:prop-disjoint}, one can choose $\star=b_{[2,4]}^-$. In the last case, one can take $\star =b_{[3,j]}^+$. 

\paragraph{Case 2(ii):} Assume that $\alpha\in B$. Without loss of generality $\alpha=b_{[2,2k]}^+$. Consider the spine $(c_1,\alpha,c_{2k+1})$ for a sharing pair for $\alpha$ (note: there is only one sharing pair for $\alpha$ in $Y^s$). As in Case 2(i), if $\beta$ disjoint from either $c_1$ or $c_{2k+1}$, then we are done, so assume $\beta$ intersects both $c_1$ and $c_{2k+1}$. Examining the cases, we conclude that either 
\[\beta=\partial N(c_0\cup b_{[1,2k+1]}^\pm\cup c_{2k+2}\cup c_{2k+3})\]
or 
\[\beta=\partial N(c_{2g+1}\cup c_0\cup b_{[1,2k+1]}^\pm\cup c_{2k+2}).\]
See the Figure \ref{fig:prop-disjoint2}.

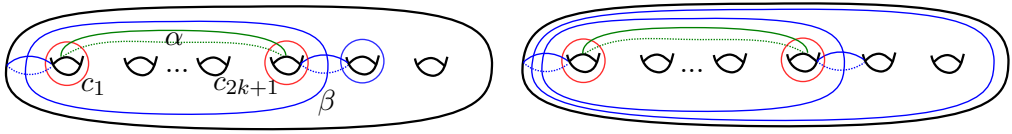
\begin{figure}[h!]
\labellist
\pinlabel $c_1$ at 1570 2200
\pinlabel $c_{2k+1}$ at 1665 2200
\pinlabel $\alpha$ at 1620 2230
\pinlabel $\beta$ at 1715 2190
\pinlabel $...$ at 1622 2210
\pinlabel $...$ at 1942 2210
\endlabellist
\centering
\includegraphics[scale=.6]{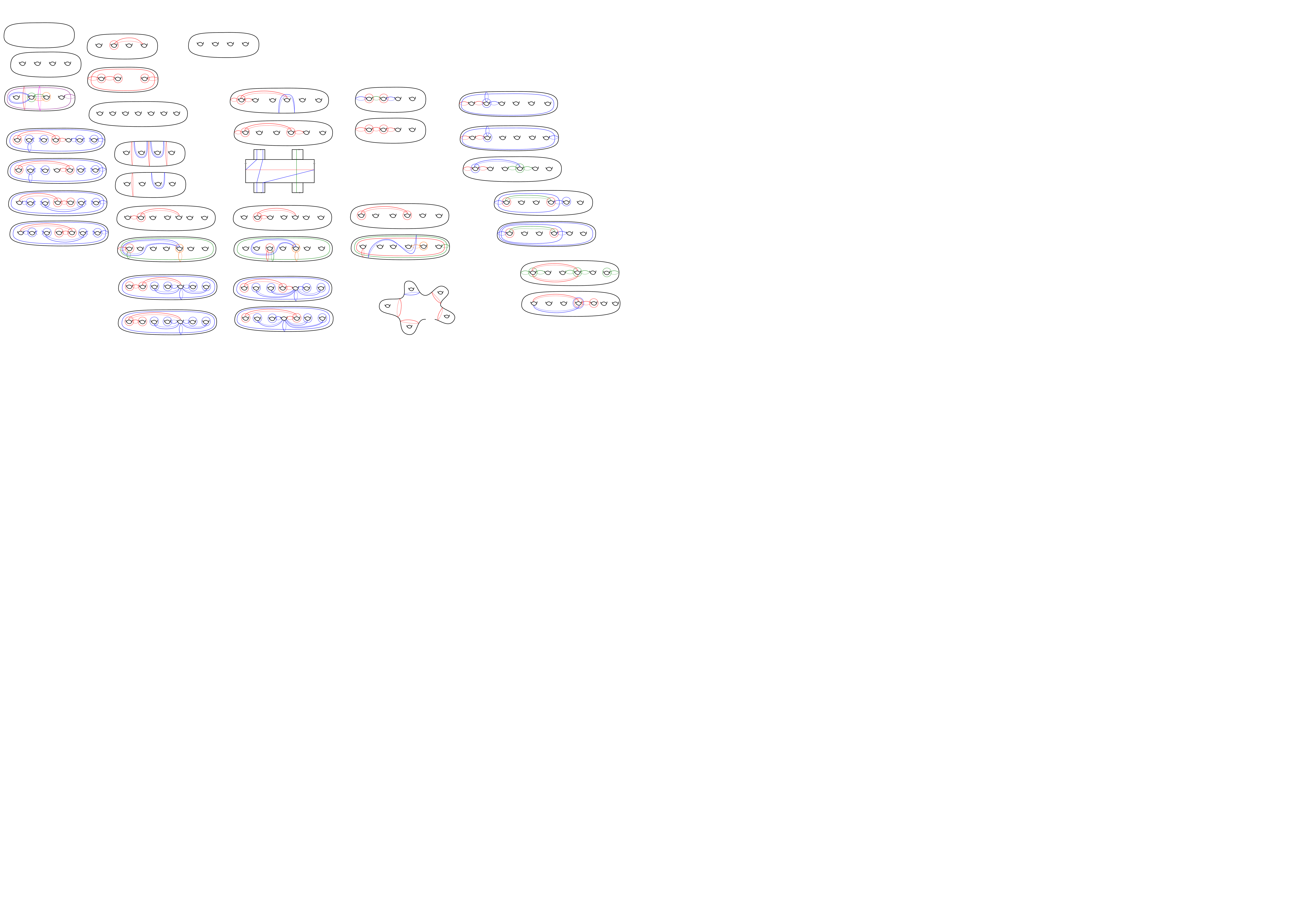}
\includegraphics[scale=.6]{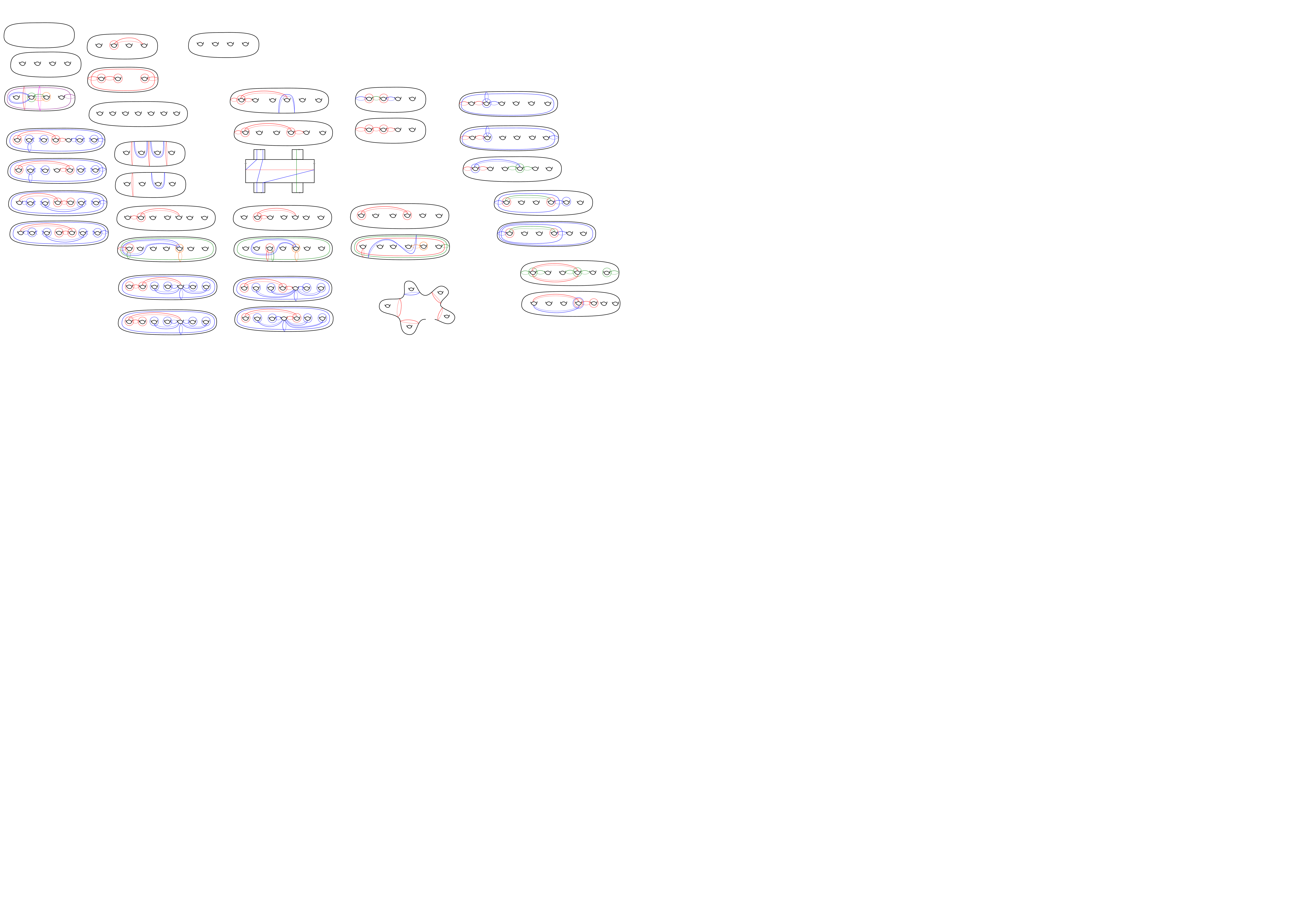}
\caption{Curves $\beta$ in $Y^s$ that intersect both $c_1$ and $c_{2k+1}$ and are disjoint from $\alpha$. }
\label{fig:prop-disjoint2}
\end{figure}


These cases can be treated similarly. For concreteness, take the former definition of $\beta$. By Lemma \ref{lem:chains} and Theorem \ref{thm:genus}, the image of the chain $(c_0,b_{[1,2k+1]}^+, c_{2k+2}, c_{2k+3})$ under $\widetilde\phi$ is also a chain and is supported in the genus-2 component of $\Sigma\setminus\phi(\beta)$. Since $\widetilde\phi(\alpha)$ is disjoint from the curves in this chain (again by Lemma \ref{lem:chains}), we deduce that $\alpha$ is supported in the genus-$(g-2)$ component of $\Sigma\setminus\phi(\beta)$. This shows $\widetilde\phi(\alpha)$ and $\widetilde\phi(\beta)$ are disjoint. (They are necessarily distinct because one is separating and the other is nonseparating.)

\paragraph{Case 3:} $\alpha\notin Y^s$ and $\beta\notin Y^s$. 

The case when $\alpha$ and $\beta$ do not form a bounding pair is treated in Lemma \ref{lem:chains}. Assume then that $\alpha$ and $\beta$ form a bounding pair. Without loss of generality, assume $\alpha=b_{[2,2k]}^+$ and $\beta=b_{[2,2k]}^-$. 

First we show that $i(\widetilde\phi(\alpha),\widetilde\phi(\beta))=0$. Consider the chain $c=(c_0,\ldots,c_{2g})$. By Lemma \ref{lem:chains}, $\widetilde\phi(c)$ is also a chain and 
\[i(\widetilde\phi(\alpha),\widetilde\phi(c_j))=i(\widetilde\phi(\beta),\widetilde\phi(c_j))=\begin{cases}1&j=1,2k+1\\0&\text{else}\end{cases}\]
Since $\widetilde\phi(c)$ is a chain of length $2g+1$, its complement $\Sigma\setminus\widetilde\phi(c)$ is a union of two disks; similarly 
\begin{equation}\label{eqn:spheres}\Sigma\setminus\big[\widetilde\phi(c_0)\cup\widetilde\phi(c_2)\cup\cdots\cup\widetilde\phi(c_{2g})\big]\end{equation} is a union of two genus-0 surfaces, each with $g+1$ boundary components. Since $\widetilde\phi(\alpha)$ and $\widetilde\phi(\beta)$ are disjoint from each $\widetilde\phi(c_{2k})$, they define curves on (\ref{eqn:spheres}). If $\widetilde\phi(\alpha)$ and $\widetilde\phi(\beta)$ lie in different components of (\ref{eqn:spheres}), we are done. If $\widetilde\phi(\alpha)$ and $\widetilde\phi(\beta)$ lie on the same component, then they must be isotopic (there is only one curve on the punctured sphere with the given intersection data with the $\widetilde\phi(c_j)$). In any case, we have shown that $i(\widetilde\phi(\alpha),\widetilde\phi(\beta))=0$. 

It remains to show that $\widetilde\phi(\alpha)$ and $\widetilde\phi(\beta)$ are actually distinct. Suppose for a contradiction that $\widetilde\phi(\alpha)=\widetilde\phi(\beta)$. Consider the curves 
\[\gamma=\partial N(\underbrace{b_{[2,2k]}^+}_{\alpha},c_{2k+1},c_{2k+2},c_{2k+3})\>\>\>\text{ and }\>\>\>\delta=\partial N(\underbrace{b_{[2,2k]}^-}_\beta,c_{2k+1}).\] 
See Figure \ref{fig:prop-disjoint3}. 

\begin{figure}[h!]
\labellist
\pinlabel $\alpha$ at 1700 2020
\pinlabel $\beta$ at 1700 1980
\pinlabel $\gamma$ at 2020 2018
\pinlabel $\delta$ at 2020 1984
\pinlabel $...$ at 1703 2000
\pinlabel $...$ at 2020 2000
\endlabellist
\centering
\includegraphics[scale=.6]{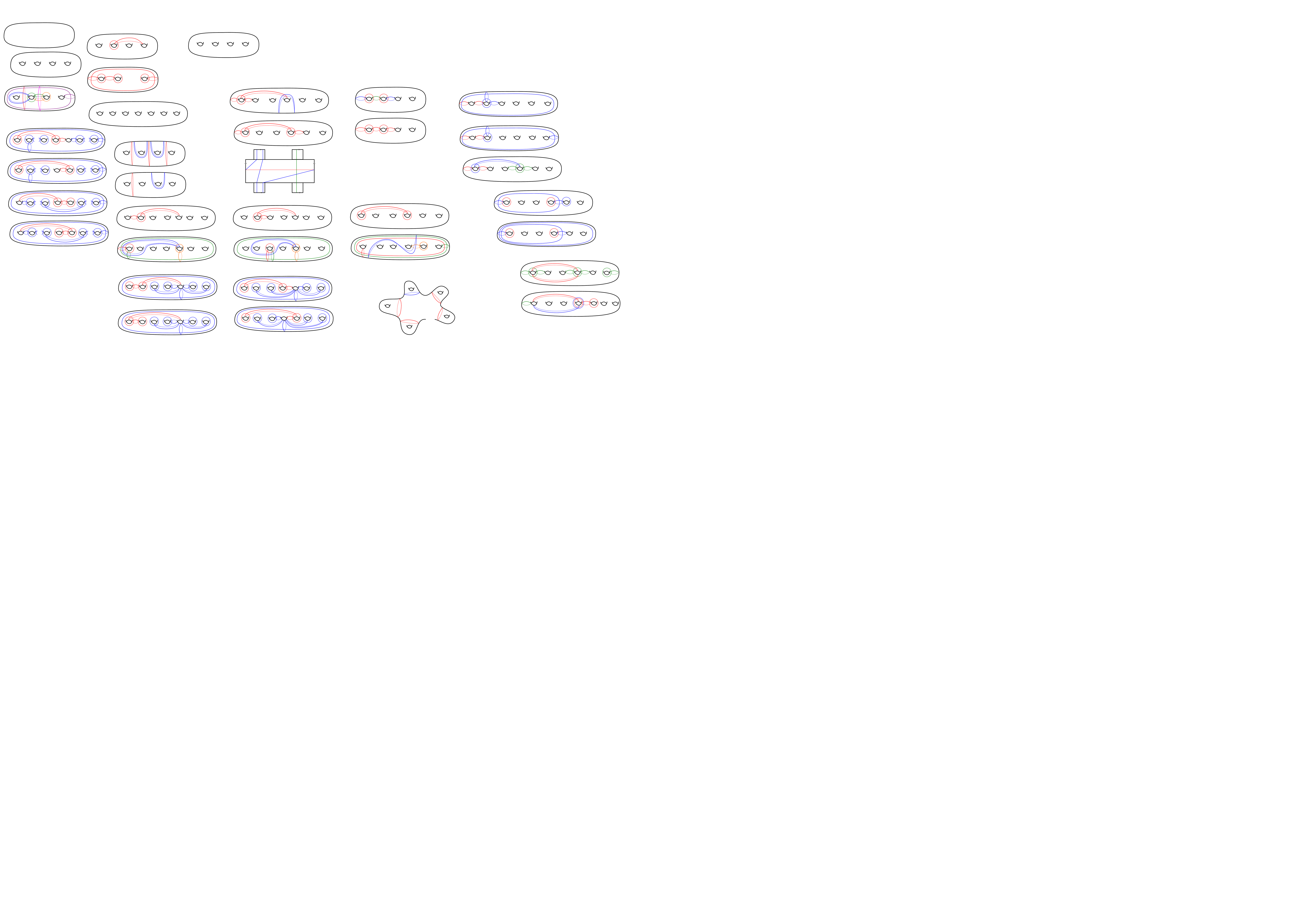}
\includegraphics[scale=.6]{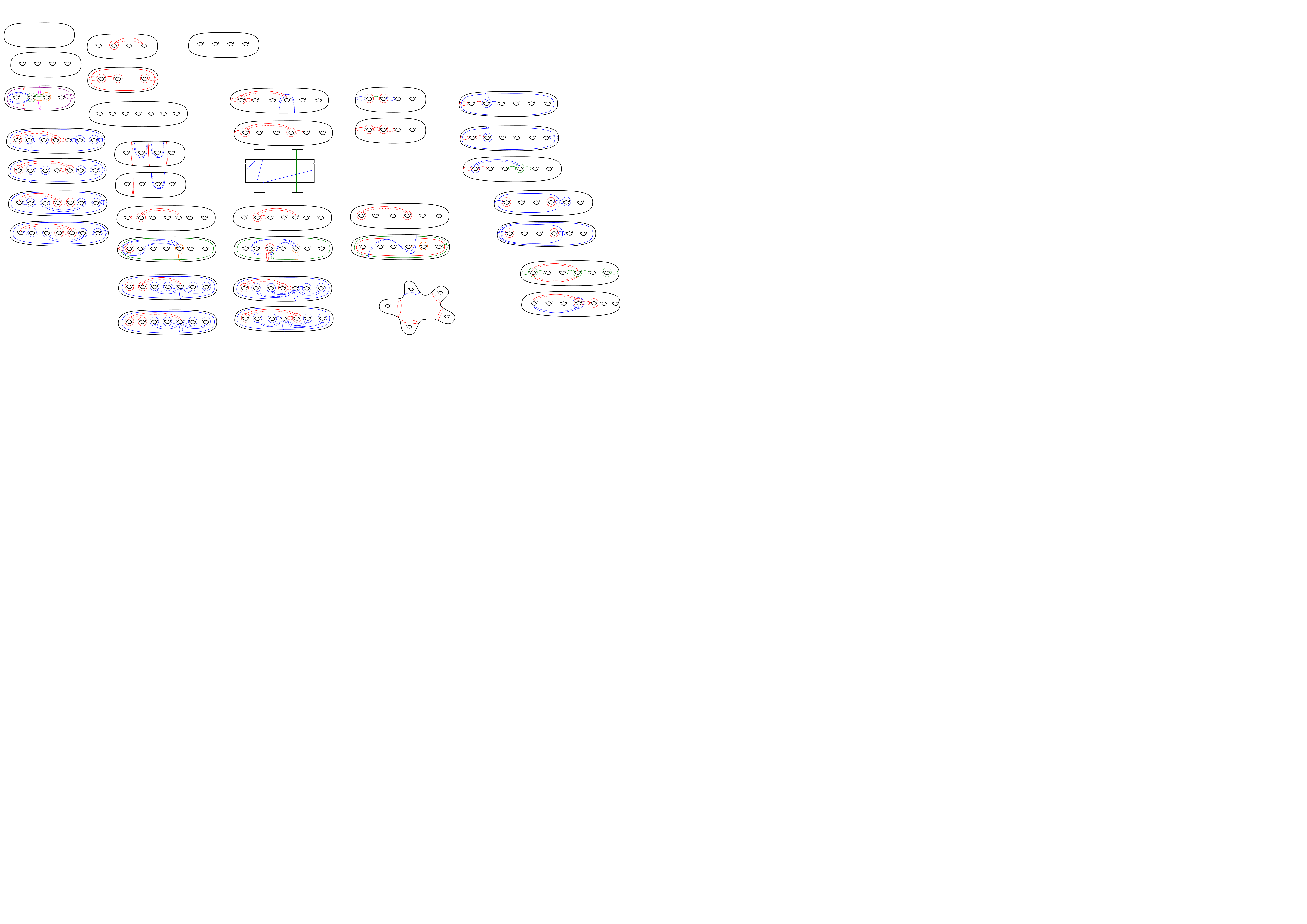}
\caption{Left: chain curves $c_0,\ldots,c_{2g}$ and the curves $\alpha,\beta$. Right: chains used to define $\gamma$ and $\delta$.}
\label{fig:prop-disjoint3}
\end{figure}
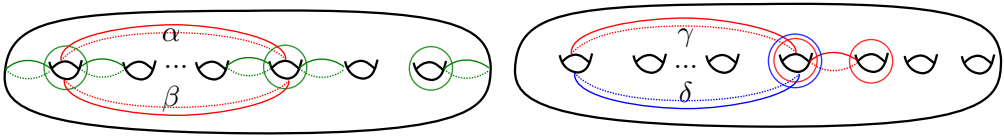

On the one hand, $\gamma,\delta\in Y^s$ intersect, so $\widetilde\phi(\gamma)=\phi(\gamma)$ and $\widetilde\phi(\delta)=\phi(\delta)$ also intersect. By Theorem \ref{thm:genus} and Lemma \ref{lem:chains}, $\phi(\delta)$ has genus 1 and $\widetilde\phi(\beta),\widetilde\phi(c_{2k+1})$ are a chain in the genus-1 component of $\Sigma\setminus\phi(\delta)$. Similarly, $\phi(\gamma)$ is a genus-2 curve that contains the chain $(\widetilde\phi(\alpha),\widetilde\phi(c_{2k+1}),\widetilde\phi(c_{2k+2}),\widetilde\phi(c_{2k+3}))$ in the genus-2 component of $\Sigma\setminus\phi(\gamma)$. In particular, $\widetilde\phi(\beta)=\widetilde\phi(\alpha)$ and $\widetilde\phi(c_{2k+1})$ are disjoint from $\widetilde\phi(\gamma)$, but this implies that $\phi(\delta)=\partial N(\widetilde\phi(\beta)\cup\widetilde\phi(c_{2k+1}))$ is disjoint from $\phi(\beta)$, a contradiction. This shows that $\widetilde\phi(\alpha)\neq\widetilde\phi(\beta)$, as desired. 
\end{proof}

\begin{proof}[Proof of Proposition \ref{prop:distance2}]
If $\alpha$ and $\beta$ are both separating, then the proposition follows from the fact that $\widetilde\phi$ extends $\phi$, and $\phi$ is incidence-preserving. If exactly one of $\alpha,\beta$ is separating, then the proposition holds because then exactly one of $\widetilde\phi(\alpha),\widetilde\phi(\beta)$ is separating. 

Now consider the case $\alpha,\beta\in C\cup B$ and both nonseparating. If $\alpha$, say, belongs to $C$, then we can apply Lemma \ref{lem:chains} to conclude that $i(\widetilde\phi(\alpha),\widetilde\phi(\beta))=1$, which implies that $\widetilde\phi(\alpha)\neq\widetilde\phi(\beta)$. If both $\alpha$ and $\beta$ belong to $B$, then $\alpha$ and $\beta$ do not form a bounding pair because we assume that $\alpha$ and $\beta$ intersect. This implies that $\alpha$ and $\beta$ have a different intersection pattern with the curves in $C$. Since this intersection pattern is preserved under $\widetilde\phi$ by Lemma \ref{lem:chains}, this implies $\widetilde\phi(\alpha)\neq\widetilde\phi(\beta)$. 
\end{proof}

\section{A finite rigid set of separating curves for $g=3$}\label{sec:genus3}

In this section, let $\Sigma$ be a closed oriented surface of genus 3. As is observed in \S\ref{sec:introduction}, we can take a single vertex in $\mathcal{C}^s(\Sigma)$ to be a finite rigid set because every separating curve in $\Sigma$ is genus-1. However, larger finite rigid sets may also exist and identifying them may shed some light on Question \ref{q:homology} and Question \ref{q:diameter}. We will present explicitly another finite rigid set $X^s\subset\mathcal{C}^s(\Sigma)$ similar to the ones constructed in \S\ref{sec:rigid-set} for $g\geq4$. The homotopy type of $X^s$ is determined as well.

The collections of curves $C,S,B$ and $U$ on $\Sigma$ are defined as in \S\ref{sec:rigid-set}. Let 
\begin{equation}\label{eqn:rigid-set-3}
X=C\cup S\cup B\cup U.
\end{equation}
and
\begin{equation}\label{eqn:rigid-set-sep-3}
X^s=S\cup U
\end{equation}

In fact $X$ is the finite rigid set for $\Sigma$ defined in \cite{AL} and $X^s=X\cap\mathcal{C}^s(\Sigma)$. Here is the main results of this section. 

\begin{thm}\label{thm:rigid-set-sep3}
Let $\Sigma$ be a closed, oriented surface of genus $3$ and set $X^s\subset\mathcal{C}^s(\Sigma)$ be the subcomplex defined by (\ref{eqn:rigid-set-sep-3}). Then the simplicial complex $X^s$ has the homotopy type of a wedge of $21$ circles, and any incidence-preserving map $\phi:X^s\to\mathcal{C}^s(\Sigma)$ is induced by a mapping class.
\end{thm}

\begin{proof}
The same method used above to prove Theorem \ref{thm:rigid-set} can be applied to verify that $X^s$ is rigid. Below we prove that $X^s$ is homotopy equivalent to a wedge of 21 circles.

First we understand the vertices of $X^s$, starting with $S\subset X^s$. (For this, it is helpful to recall the notation from \S\ref{sec:rigid-set}.) Note that every curve $s_J\in S$ has the form $s_{\{i,i+1\}}$, and there are eight of these curves. To see this, first observe that in genus 3, there are $8$ chain curves, so curves in $S$ have the form $s_J$ where $|J|\in\{2,4,6\}$. If $|J|=6$, then $s_J$ is null-homotopic, and if $|J|=4$, then $s_J=s_{J'}$ where $|J'|=2$ (concretely, if $J=[j,j+3]$, then $J'=\{j+5,j+6\}$). For simplicity, we denote 
\[s_i:=s_{\{i,i+1\}}\>\>\>\text{ for }\>\>\>i=0,\ldots,7.\]
Similarly, we find that the curves in $U\subset X^s$ all have the form 
\[u_i^\pm:=\partial N(c_{i}\cup b_{\{i+1,i+2,i+3\}}^\pm)\>\>\>\text{ for }\>\>\>i=0,\ldots,7.\]
(Note in particular that $b_{\{i+1,i+2,i+3\}}^\pm=b_{\{i-3,i-2,i-1\}}^\pm$ for each $i$.) Therefore $X^s=S\cup U$ has 24 vertices. 

Next we consider the edges in $X^s$. It is easy to check the following incidence relations 
\[
i(s_k,s_j)\text{ is }\begin{cases}
=0&\text{if }|k-j|>2\\
\neq0&\text{else}
\end{cases}
\hspace{.5in}
i(u_k,s_j)\text{ is }\begin{cases}
=0&\text{if }j=k+2\text{ or }j=k-3\\
\neq0&\text{else}
\end{cases}
\]
and 
\[
i(u_k^\pm,u_j^\epsilon)\text{ is }\begin{cases}
=0&\text{if }j=k\pm2 \text{ and }\epsilon=\mp\\
\neq0&\text{else}
\end{cases}
\]
Here $|k-j|$ denotes the distance in $\Z/8\Z=\{0,\ldots,7\}$. For checking this, it is helpful to observe that if $a=\partial N(x\cup y)$ and $b=\partial N(z\cup w)$ are distinct genus-1 curves (here $x,y$ and $z,w$ are chains of length 2), then $i(a,b)=0$ if and only if $x\cup y$ is disjoint from $z\cup w$. 

From the intersection data, we find that the subcomplex $S\subset X^s$ is 1-dimensional (see Figure \ref{fig:thm-homotopy} (left)), and that the subcomplex $U\subset X^s$ is the disjoint union of 4 simplicial circles, each with 4 vertices. These circles are pictured in Figure \ref{fig:thm-homotopy2}.

Finally, each vertex of $U$ is connected to two vertices of $S$, and together these three vertices span a 2-simplex in $X^s$. See Figure \ref{fig:thm-homotopy} (right). Observe that the edges $\{s_i,s_{i+5}\}$ each belong to two 2-simplices (the other vertices are $u_{i+3}^\pm$), while each edge connecting a vertex in $U$ to a vertex in $S$ belongs to a single 2-simplex. Therefore, each 2-simplex of $X^s$ has a ``free edge", which implies that $X^s$ deformation retracts to a 1-dimensional complex.  The simplicial complex $X^s$ has 24 vertices, 60 edges and 16 faces, so it has Euler characteristic
\[\chi(X^s)=24-60+16=-20\]
and we conclude that $X^s$ is homotopy equivalent to a wedge of 21 circles.
\begin{figure}[h!]
\labellist
\pinlabel $s_0$ at 180 1700
\pinlabel $s_1$ at 260 1700
\pinlabel $s_5$ at 180 1490
\pinlabel $s_4$ at 260 1490
\pinlabel $s_7$ at 105 1630
\pinlabel $s_6$ at 105 1560
\pinlabel $s_2$ at 335 1630
\pinlabel $s_3$ at 335 1560
\pinlabel $u_7^+$ at 582 1620
\pinlabel $u_3^+$ at 582 1564
\pinlabel $u_5^-$ at 555 1595
\pinlabel $u_1^-$ at 610 1595
\pinlabel $s_1$ at 540 1695
\pinlabel $s_4$ at 620 1695
\pinlabel $s_5$ at 540 1490
\pinlabel $s_0$ at 620 1490
\pinlabel $s_7$ at 475 1630
\pinlabel $s_2$ at 475 1560
\pinlabel $s_3$ at 690 1630
\pinlabel $s_6$ at 690 1560
\endlabellist
\centering
\includegraphics[scale=.5]{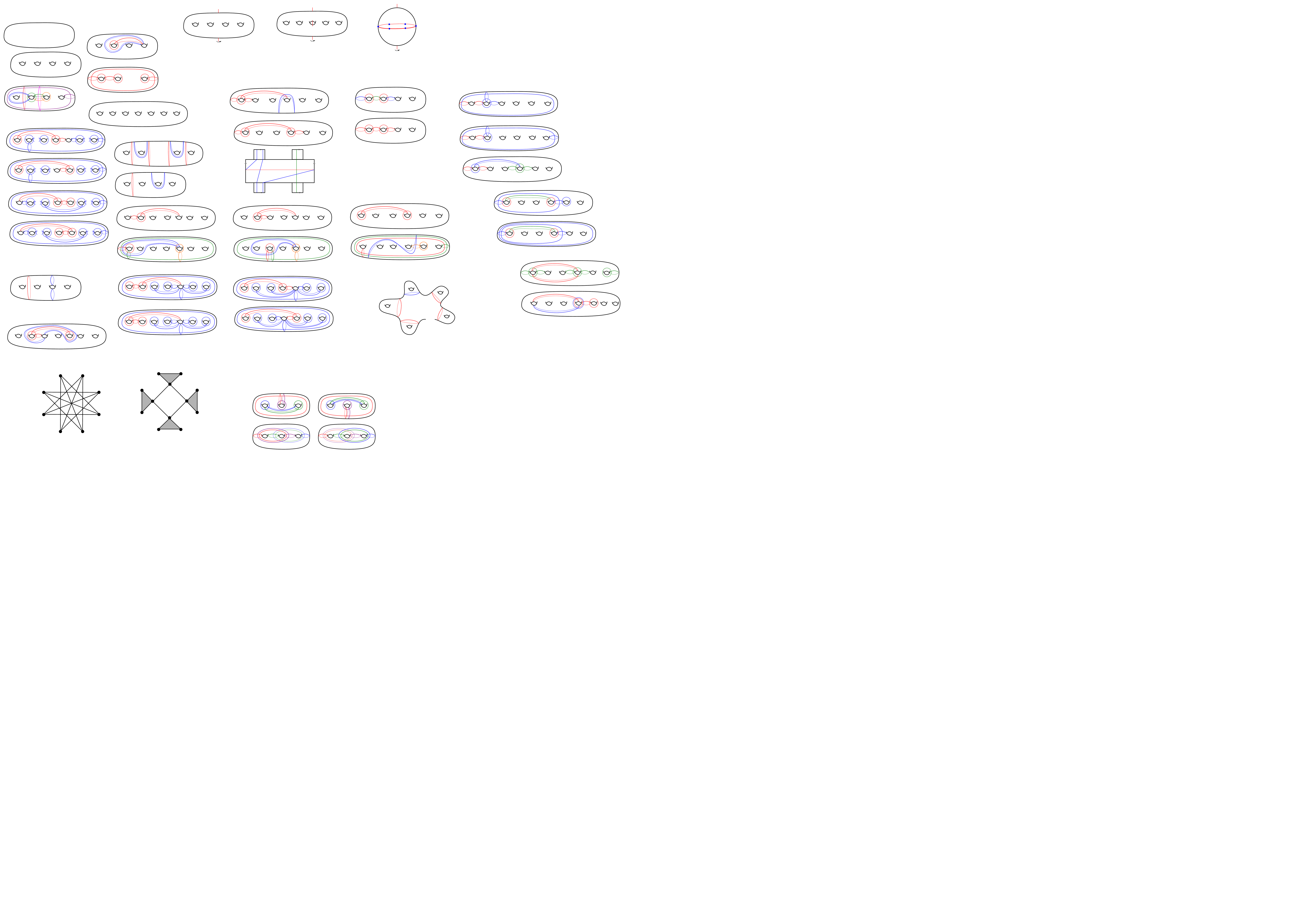}\hspace{1in}
\includegraphics[scale=.5]{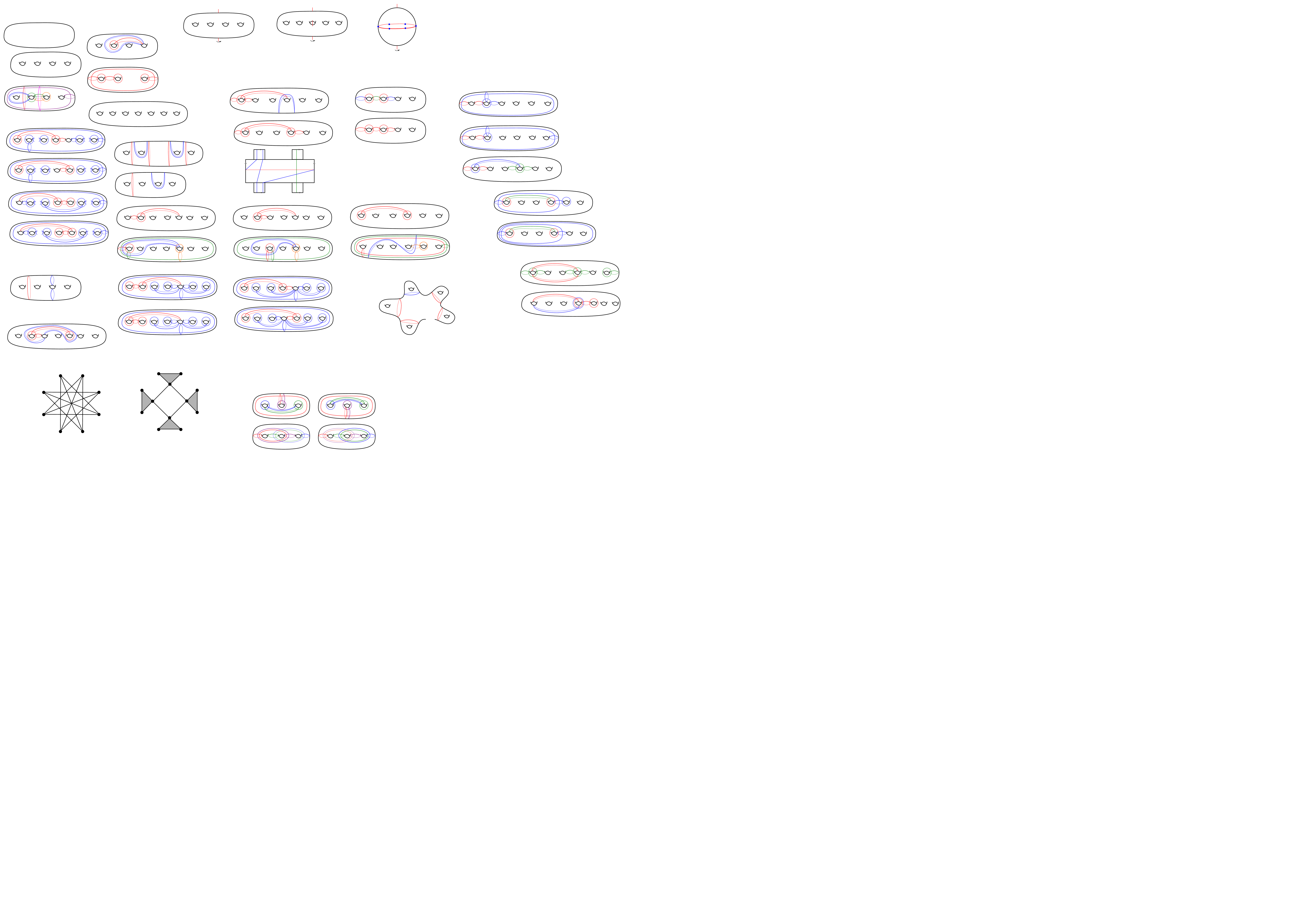}
\caption{Left: the subcomplex $S\subset X^s$. Right: the circle $\{a,b,c,d\}$ is connected to $S$ by four 2-simplices.}
\label{fig:thm-homotopy}
\end{figure}
\end{proof}

\begin{figure}[h!]
\labellist
\small
\endlabellist
\centering
\includegraphics[scale=.6]{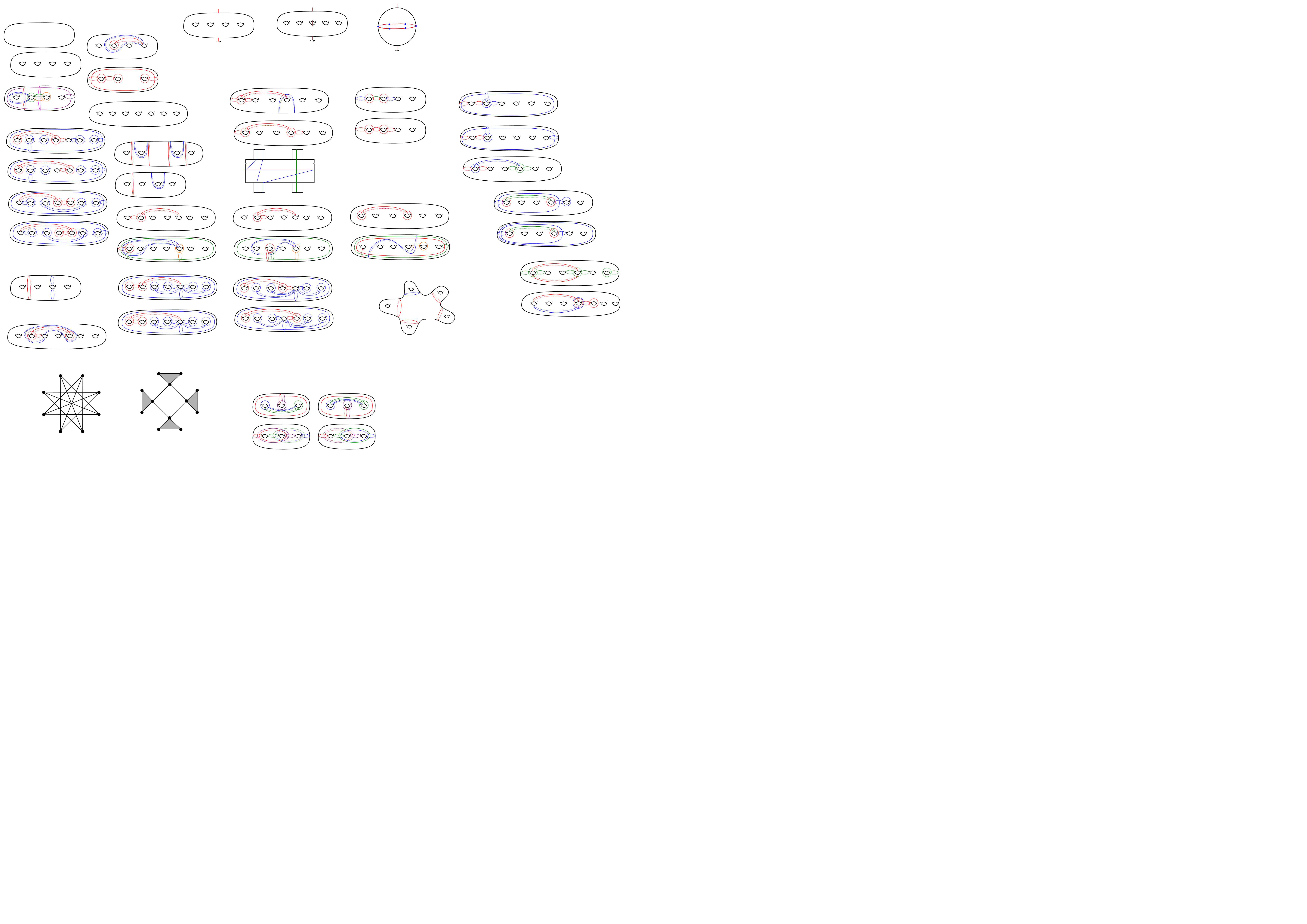}
\caption{Each frame is a circle in $U$ with four vertices. In total, $U$ is the disjoint union of these circles. }
\label{fig:thm-homotopy2}
\end{figure}

%
%

\bibliographystyle{amsalpha}
\bibliography{refs}

Junzhi Huang\\
Peking University\\
\texttt{huangjz42@pku.edu.cn}

Bena Tshishiku\\
Department of Mathematics, Brown University\\ 
\texttt{bena\_tshishiku@brown.edu}

\end{document}